\pdfoutput=1
\documentclass{amsart}

\usepackage{amsmath,amsthm,amssymb,amscd}
\usepackage{bm,mathrsfs,eucal}
\usepackage{mathrsfs,bm,graphicx} 
\usepackage[usenames,dvipsnames]{color}
\usepackage[colorlinks=true, linkcolor=red, citecolor=blue, urlcolor=blue]{hyperref}

\theoremstyle{plain}
\newtheorem{thm}{Theorem}[section]
\newtheorem{prop}[thm]{Proposition}

\newtheorem{lem}[thm]{Lemma}

\newtheorem{thmalpha}{Theorem}
\newtheorem{propalpha}[thmalpha]{Proposition}
\newtheorem{lemalpha}[thmalpha]{Lemma}

\newtheorem*{thm*}{Theorem}

\theoremstyle{definition}
\newtheorem{defn}[thm]{Definition}
\theoremstyle{remark}
\newtheorem*{rem}{Remark}

\newtheorem*{claim*}{Claim}

\newcommand{\scrT}{\mathscr{T}}
\newcommand{\scrS}{\mathscr{S}}
\newcommand{\scrL}{\mathscr{L}}
\newcommand{\scrU}{\mathscr{U}}

\newcommand{\calI}{\mathcal{I}}	
\newcommand{\calE}{\mathcal{E}}	
\newcommand{\calA}{\mathcal{A}}	
\newcommand{\calC}{\mathcal{C}}	
\newcommand{\calH}{\mathcal{H}}	
\newcommand{\calP}{\mathcal{P}}	
\newcommand{\calK}{\mathcal{K}}	
\newcommand{\calR}{\mathcal{R}}	

\newcommand{\bbZ}{\mathbb{Z}}	
\newcommand{\bbR}{\mathbb{R}}	

\newcommand\fakeslant[1]{\pdfliteral{1 0 0.25 1 0 0 cm}#1\pdfliteral{1 0 -0.25 1 0 0 cm}}
\newcommand{\mathbbs}[1]{\mathbb{\fakeslant{#1}}}
\newcommand{\bbsR}{\mathbbs{R}}	

\newcommand{\ts}{\hspace{.11111em}}
\newcommand{\tts}{\hspace{.05555em}}
\newcommand{\nts}{\hspace{-.11111em}}
\newcommand{\ntts}{\hspace{-.05555em}}

\newcommand{\nil}{\varnothing}
\newcommand{\ssminus}{\smallsetminus}
\newcommand{\rank}{\operatorname{rk}}

\newcommand{\bbZtwo}{\mathbb{Z}_{\nts/\ntts2\mathbb{Z}}}

\newcommand{\homeo}{\approx}
\newcommand{\nhomeo}{\not\approx}

\newcommand{\RP}{\bbsR \ntts P}
\newcommand{\SxS}{S^2 \nts\times\nts S^1}

\newcommand{\tri}{\scrT}
\newcommand{\simplices}{\varSigma}
\newcommand{\pairings}{\varPhi}
\newcommand{\quot}{\varpi}
\newcommand{\cell}{\scrU}
\newcommand{\lst}{\scrS}
\newcommand{\lls}{\scrL}

\newcommand{\Index}{\calI}
\newcommand{\Evens}{\calE}
\newcommand{\Adults}{\calA}

\newcommand{\Children}{\calC}
\newcommand{\Klusters}{\calK}
\newcommand{\House}{\calH}
\newcommand{\Posse}{\calP}
\newcommand{\Rogues}{\calR}

\newcommand{\Ecount}{\bm{E}}

\newcommand{\Tcount}{\bm{T}}
\newcommand{\Icount}{\bm{I}}

\newcommand{\fcount}{\bm{f}}
\newcommand{\gcount}{\bm{g}}

\newcommand{\icount}{\bm{i}}

\newcommand{\dcount}{\bm{d}}
\newcommand{\scount}{\bm{s}}
\newcommand{\bcount}{\bm{b}}
\newcommand{\acount}{\bm{a}}

\newcommand{\chineg}{\chi_{\nts\scriptscriptstyle-}}

\newcommand{\typee}{{\displaystyle \mathfrak{e}}}

\newcommand{\typet}{{\displaystyle \mathfrak{t}}}
\newcommand{\typeq}{{\displaystyle \mathfrak{q}}}
\newcommand{\typett}{{\displaystyle \mathfrak{tt}}}
\newcommand{\typeqq}{{\displaystyle \mathfrak{q \ntts q}}}
\newcommand{\typeqtt}{{\displaystyle \mathfrak{qtt}}}
\newcommand{\typeqqq}{{\displaystyle \mathfrak{q \ntts q \ntts q}}}

\newcommand{\aclass}{\varphi}			
\newcommand{\zclass}{0}
\newcommand{\iclass}{\aclass_i}
\newcommand{\jclass}{\aclass_j}

\newcommand{\rclass}{\textcolor{black}{\aclass_1}}
\newcommand{\gclass}{\textcolor{black}{\aclass_2}}
\newcommand{\bclass}{\textcolor{black}{\aclass_3}}

\newcommand{\surf}{F}				
\newcommand{\subsurf}{S}
\newcommand{\asurf}{\surf_{\aclass}}
\newcommand{\isurf}{\surf_i}
\newcommand{\jsurf}{\surf_j}
\newcommand{\ksurf}{\surf_k}

\newcommand{\kluster}{K}

\begin{document}


\title[The Complexity and the $\bbZtwo$-Cohomology]
	{The Complexity of Prime 3-manifolds and \\ the First $\bbZtwo$-Cohomology of Small Rank}
\author[Kei Nakamura]{Kei Nakamura}
\thanks{Nakamura is partially supported by NSF FRG grant DMS-1463940.}
\address{Department of Mathematics\\ Rutgers University \\ New Brunswick, NJ}
\email{kei.nakamura@rutgers.edu}

\begin{abstract}
For a closed orientable connected 3-manifold $M$, its complexity $\Tcount(M)$ is defined to be the minimal number of tetrahedra in its triangulations. Under the assumption that $M$ is prime (but not necessarily atoroidal), we establish a lower bound for the complexity $\Tcount(M)$ in terms of the  $\bbZtwo$-coefficient Thurston norm for $H^1(M;\bbZtwo)$:
\begin{enumerate}
\item for any rank-1 subgroup $\{\zclass,\aclass\} \leqslant H^1(M;\bbZtwo)$, we have
\[
\Tcount(M) \geqslant 2+2||\aclass||
\]
unless $M$ is a lens space with $\Tcount(M)=1+2||\aclass||$;
\item for any rank-2 subgroup $\{\zclass,\rclass,\gclass,\bclass\} \leqslant H^1(M;\bbZtwo)$, we have
\[
\Tcount(M) \geqslant 2+||\rclass||+||\gclass||+||\bclass||.
\]
\end{enumerate}
Under the extra assumption that $M$ is \emph{atoroidal}, these inequalities had already been shown; see \cite{JRT:Lens} where the rank-1 inequality is given implicitly, and \cite{JRT:Z2} where the rank-2 inequality is given explicitly. Our work here shows that we do not need to require $M$ to be atoroidal.
\end{abstract}

\maketitle

\section{Introduction} 
\label{sec:Introduction}

For a closed orientable 3-manifold $M$, the \emph{complexity of $M$} is defined to be the minimal number of tetrahedra in a triangulation of $M$; here, and hereafter, we always allow triangulations to be pseudo-simplicial. This notion of complexity agrees with the complexity defined by Matveev \cite{Mat:Table} and studied by Martelli and Petronio \cite{MP:Geometric}, except for $S^3$, $\RP^2$, and $L(3,1)$. Determining the complexity for a given manifold $M$ is known to be a difficult problem; the seemingly easier task of finding a lower/upper bound for the complexity is still highly non-trivial.

In \cite{JRT:Lens,JRT:Covering,JRT:Z2}, Jaco, Rubinstein and Tillmann established certain lower bounds for $\Tcount(M)$, which are attained by particular minimal triangulations for infinite families of 3-manifolds. For $\aclass \in H^1(M;\bbZtwo)$, the \emph{$\bbZtwo$-coefficient Thurston norm} $||\aclass|||$ is defined to be the minimum of $\max\{0,-\chi(F)\}$ with $F$ varying over all closed surfaces that are Poincar\'e dual to $\aclass$. In terms of this norm, their lower bounds for lens spaces can be stated as follows.

\begin{thmalpha}[\cite{JRT:Lens}] \label{thm:JRTLens}
Let $M$ be a lens space. Then, for any cohomology class $\aclass \in H^1(M;\bbZtwo)$, we have
\[
\Tcount(M) \geqslant 1+2||\aclass||.
\]
\end{thmalpha}

There is an infinite family of lens spaces for which the standard \emph{layered lens space triangulation} $\scrL$ with the number of tetrahedra attaining the lower bound $1+2||\aclass||$, and it follows that lens spaces in this family has the complexity $\Tcount(M)=1+2||\aclass||$ \cite{JRT:Lens}. The proof of \hyperref[thm:JRTLens]{Theorem~\ref{thm:JRTLens}} consists of two parts; they first showed that the standard layered lens space triangulations $\scrL$ satisfies the inequality $1+2||\aclass||$, and then they showed that any other triangulation $\scrT$ of a lens space satisfies the inequality $2+2||\aclass||$. The latter part of the proof applies verbatim to any irreducible atoroidal 3-manifold other than lens spaces. So, as a byproduct, the following statement is established implicitly in \cite{JRT:Lens}.

\begin{thmalpha}[\cite{JRT:Lens}] \label{thm:JRT1}
Let $M$ be an oriented connected closed irreducible atoroidal 3-manifold other than a lens space. Then, for any rank-1 subgroup $\{0,\aclass\} \leqslant H^1(M;\bbZtwo)$, we have
\[
\Tcount(M) \geqslant 2+2||\aclass||.
\]
\end{thmalpha}

In \cite{JRT:Z2}, the ideas from \cite{JRT:Lens} are extended to the context involving a rank-2 subgroup $H \leqslant H^1(M;\bbZtwo)$. In particular, they generalized \hyperref[thm:JRT1]{Theorem~\ref{thm:JRT1}}.

\begin{thmalpha}[\cite{JRT:Z2}] \label{thm:JRT2}
Let $M$ be an oriented connected closed irreducible atoroidal 3-manifold. Then, for any rank-2 subgroup
$\{0,\rclass, \gclass, \bclass\} \leqslant H^1(M;\bbZtwo)$, we have
\[
\Tcount(M) \geqslant 2+||\rclass||+||\gclass||+||\bclass||.
\]
\end{thmalpha}

This time, there is an infinite family of certain small Seifert fibered spaces, for which the \emph{twisted layered loop triangulation} with the number of tetrahedra attaining the lower bound $2+2||\aclass||$, and it follows that small Seifert fibered spaces in this infinite family has the complexity $\Tcount(M)=2+||\aclass||$ \cite{JRT:Z2}.

The purpose of this article is to generalize \hyperref[thm:JRT1]{Theorem~\ref{thm:JRT1}} and \hyperref[thm:JRT2]{Theorem~\ref{thm:JRT2}} by allowing $M$ to be any prime 3-manifolds that is not a lens space. In particular, we do not require the condition that $M$ is atoroidal; it should be noted that this condition plays a crucial role in the proof of \cite{JRT:Lens,JRT:Z2}, but we will establish our inequality without this assumption.

\begin{thm} \label{thm:main1}
Let $M$ be an oriented connected closed prime 3-manifold other than a lens space. Then, for any rank-1 subgroup $\{0,\aclass\} \leqslant H^1(M;\bbZtwo)$, we have
\[
\Tcount(M) \geqslant 2+2||\aclass||.
\]
\end{thm}

\begin{thm} \label{thm:main2}
Let $M$ be an oriented connected closed prime 3-manifold. Then, for any rank-2 subgroup
$\{0,\rclass, \gclass, \bclass\} \leqslant H^1(M;\bbZtwo)$, we have
\[
\Tcount(M) \geqslant 2+||\rclass||+||\gclass||+||\bclass||.
\]
\end{thm}

\subsection*{Outline}

Foundational material on 3-manifolds, triangulations, and Thurston norm are collected in \hyperref[ssec:Triang]{\S\ref{ssec:Triang}}. We then prove \hyperref[thm:main1]{Theorem~\ref{thm:main1}} in \hyperref[sec:Rank1]{\S\ref{sec:Rank1}}, and then \hyperref[thm:main2]{Theorem~\ref{thm:main2}} in \hyperref[sec:Rank2]{\S\ref{sec:Rank2}}. Somewhat detailed outlines of the proofs are given at the beginning of \hyperref[sec:Rank1]{\S\ref{sec:Rank1}} and \hyperref[sec:Rank2]{\S\ref{sec:Rank2}}.
We will carry out certain counting argument based on the combinatorial profiles around the so-called \emph{even} edges, and then analyze the norm $||\aclass||$ via the combinatorics and topology of the canonical normal surface $\asurf$ dual to $\aclass$. Heuristically, if some even edges contribute ``negatively'' against establishing the desired inequality, we will look for a way to compensate for it; we will do this first by grouping these edges with other even edges that contribute ``positively'' toward establishing the desired inequality, and then by compressing the surface $\asurf$ to assure a gap between $-\chi(\asurf)$ and $||\aclass||$. To establish the desired inequality, we must show that sufficient number of compressions can be carried out together; this is proved by analyzing the normal surface $\asurf$ locally (i.e.\;identifying the compression disks across the problematic edges) and globally (i.e.\;estimating the number of compressions that can be carried out together).



\section{Canonical Surfaces in Minimal Triangulations}
\label{sec:Canonical}

\subsection{Triangulations}
\label{ssec:Triang}

A \emph{triangulation} $\tri=(\simplices,\quot)$ of a compact 3-manifold $Q$ consists of a finite disjoint union $\simplices=\bigsqcup_i \sigma_i$ of euclidean 3-simplices and a quotient map $\quot: \simplices \rightarrow \simplices/\pairings=Q$ via a set $\pairings$ of affine face pairings, defining a $\varDelta$-complex structure on $Q$. The restriction of the map $\quot$ to the interior of each $k$-simplex in $\simplices$ is injective. We refer to the image of 3-, 2-, 1-, 0-simplices under $\quot$ as \emph{tetrahedra}, \emph{faces}, \emph{edges}, \emph{vertices} in $\tri$, with the induced incidence structure. An edge is called a \emph{boundary edge} if it is in $\partial Q$, and an \emph{interior edge} otherwise; similarly, a face is called a \emph{boundary face} if it is in $\partial Q$, and an \emph{interior face} otherwise. The \emph{degree} of an edge $e$, denoted by $\dcount(e)$, is the number of 1-simplices in the preimage $\quot^{-1}(e)$.

The disjoint union $\simplices' \subset \simplices$ of a subcollection of 3-simplices descends to the subcomplex $R=\quot(\simplices') \subset \quot(\simplices)=Q$ with the \emph{induced cellulization} $\cell=(\simplices',\quot|_{\simplices'})$ defining the $\varDelta$-complex structure on $\quot(\simplices')$. If $e$ is an edge in a subcomplex $R$, the degree of $e$ in $\cell$ is the number of 1-simplices in the preimage $\quot^{-1}(e)$ incident to the 3-simplices in $\simplices'$, and we denote it by $\dcount_\cell(e)$; the notation $\dcount(e)$ is reserved for the degree of $e$ in the ambient triangulation $\tri$.
We point out that $\simplices' \subset \simplices$ and the quotient map $\quot': \simplices' \rightarrow \simplices'/\pairings'=Q'$ via the subset $\pairings' \subset \pairings$ of \emph{all} face pairings between the faces of 3-simplices in $\simplices'$ form a triangulation $\tri'=(\simplices',\quot')$ of a compact 3-manifold $Q'$. The natural map $\iota:Q' \rightarrow Q$ satisfying $\iota \ts\circ\ts \quot'=\quot|_{\simplices'}$ takes the complex $Q'=\quot'(\simplices')$ onto the subcomplex $R=\quot(\simplices') \subset \quot(\simplices)=Q$. We call $\tri'=(\simplices',\pairings')$ the \emph{triangulation of $Q'$ generated by the subcomplex $R \subset Q$}. The map $\iota$ is injective if and only if $R \homeo Q'$ is a submanifold of $Q$ with the \emph{induced triangulation} $\cell \homeo \tri'$; generally, $\iota$ may fail to be injective along edges and vertices in $\partial Q'$, and $R$ may contain non-manifold points in its 1-skeleton.

\subsection{Layered Triangulations}
\label{ssec:Layered}

Given a triangulation $\cell$ of a compact manifold $Q$ with $\partial Q \neq \nil$, suppose $u$ is a boundary edge incident to two distinct boundary faces. The union of these faces is the closure of an open quadrilateral with a diagonal $u$. 

To \emph{layer along the edge $u$} is to construct a new triangulation of $Q$ from $\cell$ and a 3-simplex $\tau$ by identifying the quadrilateral formed by two boundary faces of $\cell$ incident to $u$ with a quadrilateral formed by two faces of $\tau$ incident to an edge $e$ so that the diagonals $u$ and $e$ are identified; see \hyperref[fig:layering]{Figure~\ref*{fig:layering}} (left). In the new triangulation, the diagonals $u$ and $e$ become a single interior edge of degree $\dcount_\cell(u)+1$, and the edge $e'\nts$ opposite to $e$ in $\tau$ becomes a boundary edge of degree~1.

To \emph{fold along the edge $u$} is to construct a new triangulation of a 3-manifold, not necessarily homeomorphic to $Q$, from $\cell$ by identifying the two boundary faces incident to $u$ via the reflection of the quadrilateral about the diagonal $u$; see \hyperref[fig:layering]{Figure~\ref*{fig:layering}} (right). We note that, if $\partial Q \homeo T^2$, the new triangulation yields a closed manifold.

\begin{figure}[h]
\begin{center} 
\includegraphics[height=30mm, width=30mm]{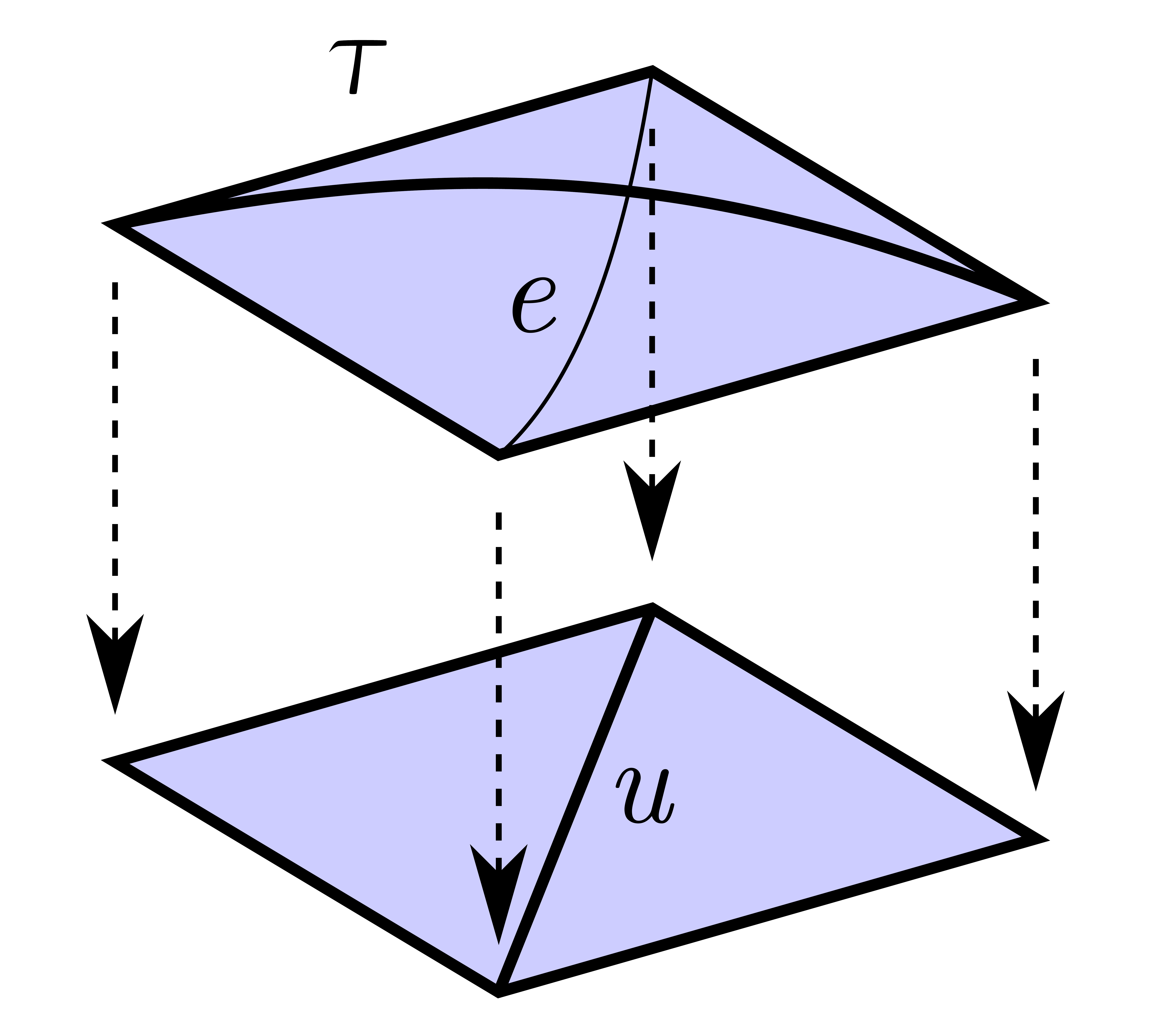} \qquad
\includegraphics[height=30mm, width=30mm]{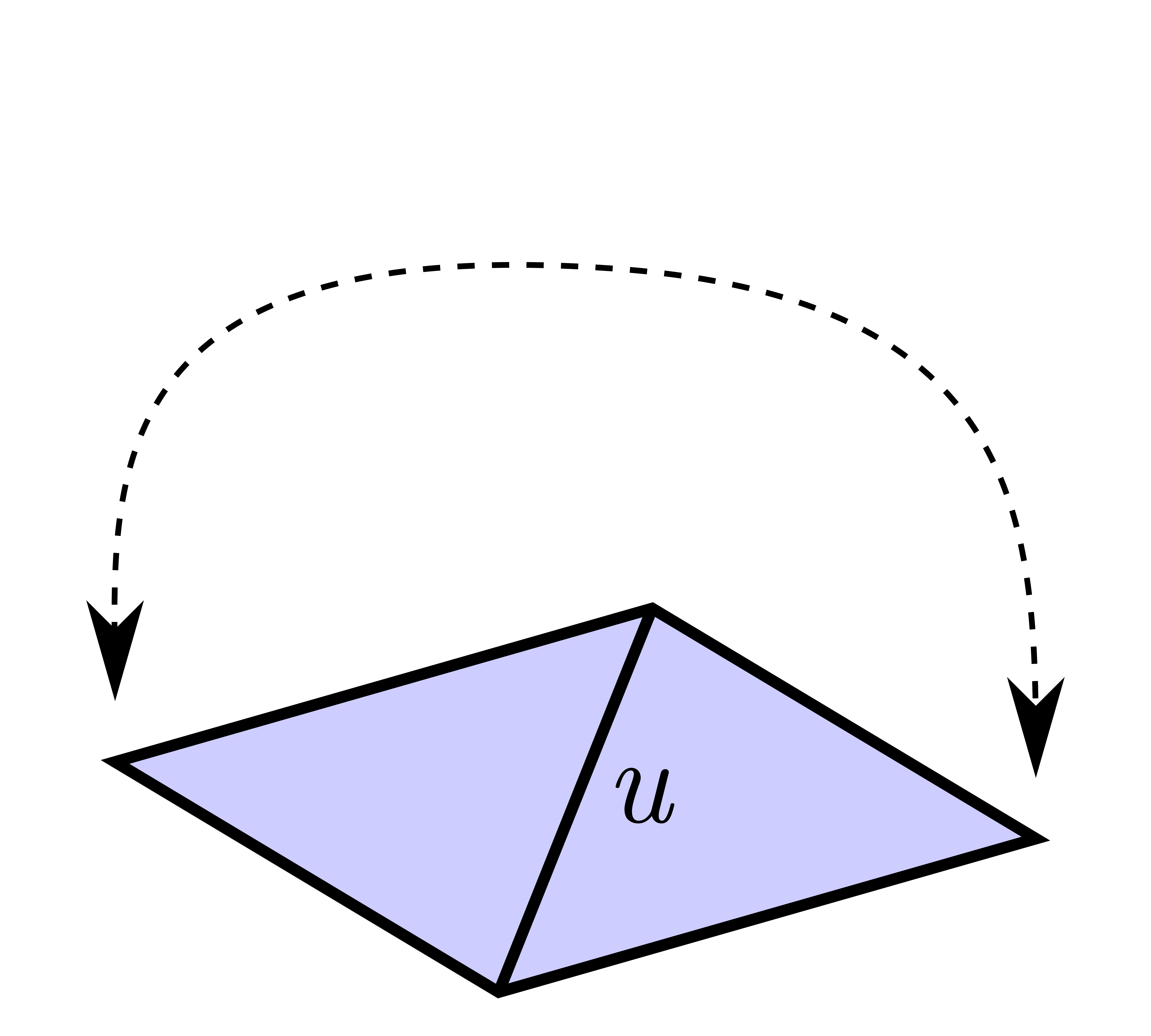}
\end{center}
\vspace{-2mm}
\caption{Layering (left) and folding (right) along the edge $u$.}
\label{fig:layering}
\vspace{-1mm}
\end{figure}

A \emph{layered solid torus} \cite[\S4]{JR:Layered} is a triangulation $\ntts\lst\ntts$ of a solid torus, obtained from the unique 1-tetrahedron triangulation $\lst_{\{1,2,3\}}$ of a solid torus in \hyperref[fig:lst]{Figure~\ref*{fig:lst}}~(left) by layering (possibly zero) tetrahedra iteratively along boundary edges; by an abuse of language, a layered solid torus may also refer to a solid torus equipped with such a triangulation $\lst$. A \emph{layered lens space} \cite[\S6]{JR:Layered} is a triangulation $\lls$ of a lens space, obtained from a layered solid torus $\lst$ by folding along a boundary edge; by an abuse of language, a layered lens space may also refer to a lens space equipped with such a triangulation $\lls$.

By convention, we avoid layering or folding along a boundary edge of degree 1 in the construction of a layered solid torus or a layered lens space in this article. Under this convention, a layered solid torus is canonically defined for almost every meridional slope: for any coprimitive $p,q \in \bbZ$ satisfying $0<p<q$, there exists a unique layered solid torus $\lst_{\{p,q,p+q\}}$ such that its (unique) normal meridional disk has the boundary edge-weights $\{p,q,p+q\}$ \cite[\S5]{JR:Layered}. Layered solid tori $\lst_{\{1,2,3\}}$ and $\lst_{\{1,3,4\}}$ are shown in \hyperref[fig:lst]{Figure~\ref*{fig:lst}}. Note that triples $\{1,1,2\}$ and $\{0,1,1\}$ are excluded above by the inequalities $0<p<q$; these triples can be realized as the meridional boundary edge-weights in ``\emph{exceptional}'' layered triangulations of a solid torus \cite[Fig.\;5]{JR:Layered}, but they are not referred to as layered solid tori in this article.

\begin{figure}[h]
\begin{center} 
\includegraphics[height=25mm, width=25mm]{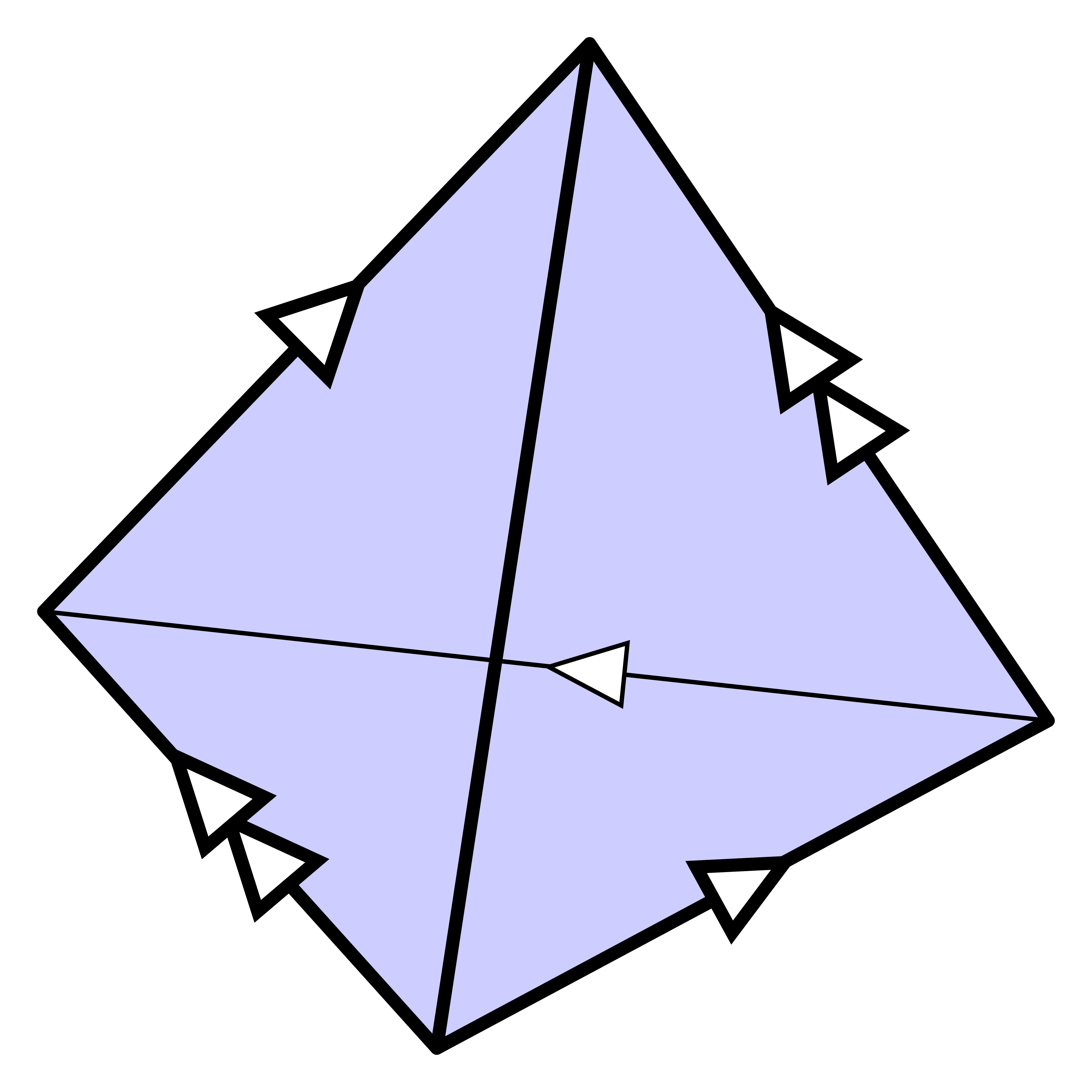} \qquad\qquad
\includegraphics[height=25mm, width=25mm]{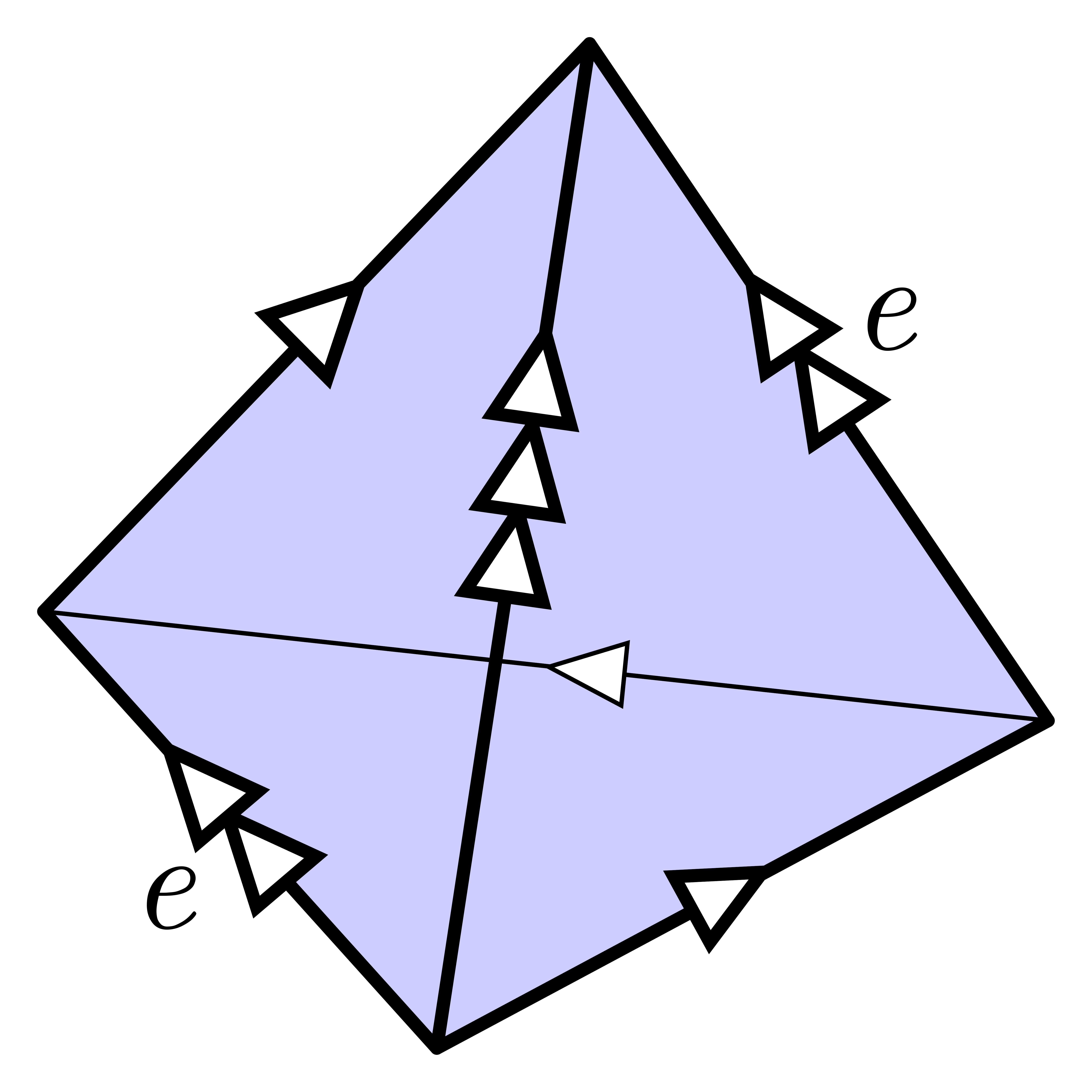} \quad
\includegraphics[height=25mm, width=25mm]{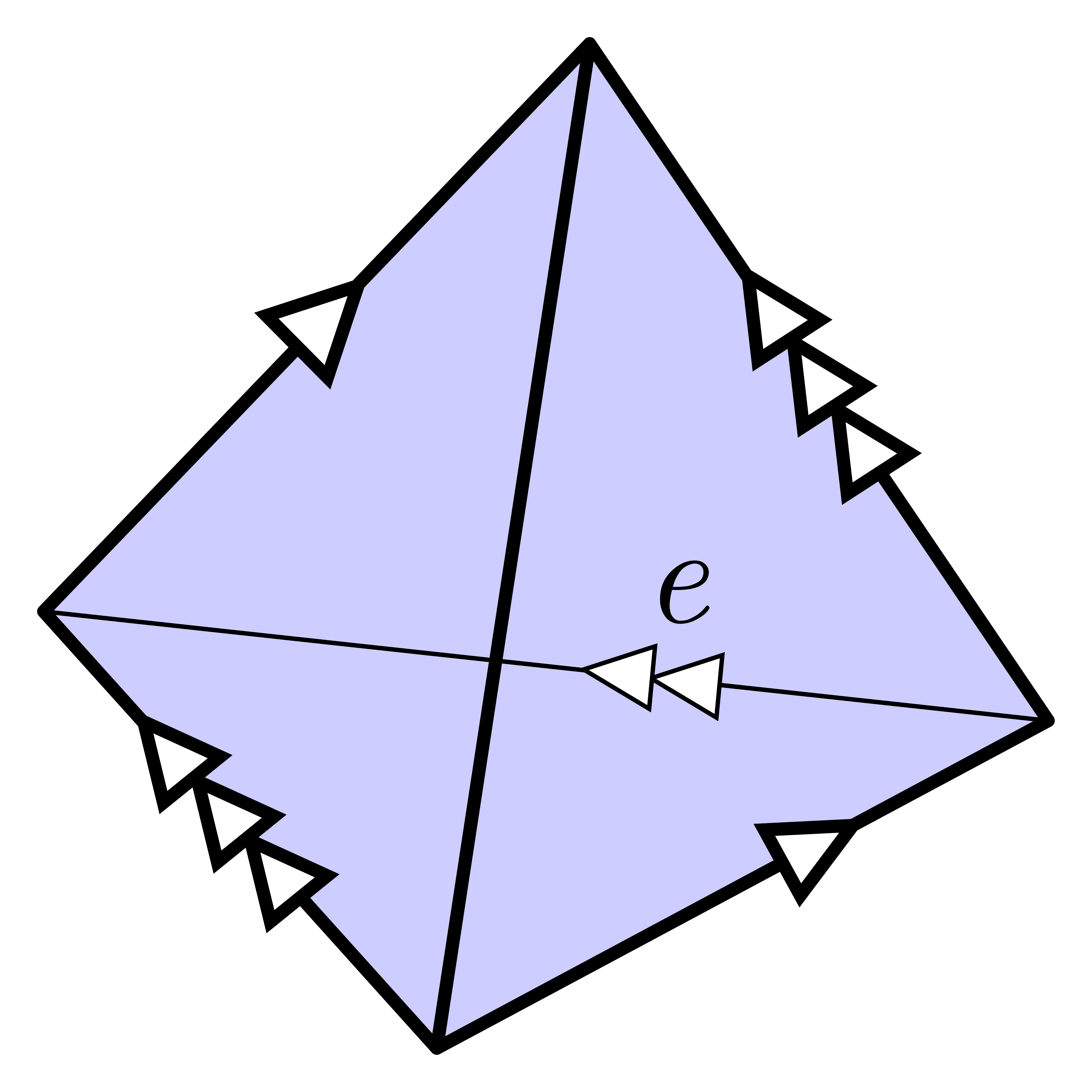}
\end{center}
\vspace{-2mm}
\caption{The layered solid tori $\lst_{\{1,2,3\}}$ (left), and the layered solid torus $\lst_{\{1,3,4\}}$ with an interior edge $e$ of degree~3 (right).}
\label{fig:lst}
\vspace{-1mm}
\end{figure}

\begin{rem}
Our notions of layered solid tori and layered lens spaces are rather restrictive; see \cite[\S4-5]{JR:Layered} for a more general notion of layered solid tori and layered lens space, where one starts with a one-triangle M\"obius band (regarded as a degenerate solid torus) instead of $\lst_{\{1,2,3\}}$, and allows layereing and folding along an edge of degree 1. Analogous constructions in terms of \emph{(almost) special spines} of 3-manifolds are described by Matveev \cite{Mat:Table} and by Martelli and Petronio \cite{MP:Geometric}.
\end{rem}

For reference, we collect some basic facts about degrees of edges in layered solid tori in the following lemma, which readily follow from the layering construction and our convention that we never layer along an edge of degree 1.

\begin{lem} \label{lem:edges_lst}
For any layered solid torus $\lst$, the following statements hold:
\begin{itemize}
\item $\lst$ has no interior edge of degree 1;
\item $\lst$ has no interior edge of degree 2;
\item $\lst$ has no interior edge of degree 3 if $\lst \not\supseteq \lst_{\{1,3,4\}}$;\\
$\lst$ has exactly one interior edge of degree 3 if $\lst \supseteq \lst_{\{1,3,4\}}$;
\item $\lst$ has exactly one boundary edge of degree 1;
\item $\lst$ has exactly one boundary edge of degree 3;
\item $\lst$ has exactly one boundary edge of degree $2$ if $\lst=\lst_{\{1,2,3\}}$;\\
$\lst$ has exactly one boundary edge of degree $d \geq 4$ if $\lst\neq\lst_{\{1,2,3\}}$.
\end{itemize}
\end{lem}

\subsection{Minimal Triangulations}
\label{ssec:Minimal}

From now on, whenever we write $M$ for a 3-manifold, we \emph{always} assume by convention that $M$ is closed, connected, and orientable, often without mentioning these properties. We assume rudiments from \emph{normal surface theory} for triangulated 3-manifolds; see, for example, \cite{JR:0-efficient, JR:Layered}.

For a triangulation $\tri$ of a 3-manifold $M$, the number of tetrahedra in $\tri$ is denoted by $\Tcount(\tri)$, or simply by $\Tcount$. A triangulation $\tri$ is said to be \emph{minimal} if $\Tcount(\tri)$ is minimal among all triangulations of $M$, and \emph{0-efficient} if there is no normal $S^2$ other than the vertex-linking ones \cite{JR:0-efficient}.

\begin{thmalpha}[{\cite[Thm.\,6.1]{JR:0-efficient}}] \label{thm:0-efficient}
A minimal triangulation of a prime 3-manifold $M$ is a one-vertex 0-efficient triangulation unless $M \homeo S^3$, $\SxS$, $\RP^3$ or $L(3,1)$.
\end{thmalpha}

The degree of an edge plays an important role in analysis of minimal triangulations \cite{JR:0-efficient, JRT:Lens, JRT:Covering, JRT:Z2}. The 0-efficiency of minimal triangulations leads to the following essential facts about edges of low degrees.

\begin{thmalpha}[{\cite[Prop.\,9]{JRT:Lens}, cf.\;\cite[Thm.\,6.3]{JR:0-efficient}}] \label{thm:edges_min}
For any minimal triangulation $\tri$ of an irreducible 3-manifold $M$, the following statements hold:
\begin{itemize}
\item $\tri$ has no edge of degree~1, unless $M \homeo S^3$, $\RP^3$, or $L(3,1)$;
\item $\tri$ has no edge of degree~2, unless $M \homeo \SxS$, $\RP^3$, $L(3,1)$, or $L(4,1)$;
\item Every edge of degree~3 in $\tri$ is the unique interior edge of a two-tetrahedra subcomplex homeomorphic to a solid torus and with the induced triangulation isomorphic to $\lst_{\{1,3,4\}}$, unless $M \homeo \RP^3$, $L(5,1)$, $L(5,2)$, or $L(7,2)$.
\end{itemize}
\end{thmalpha}

\begin{rem}
The statement in \cite[Prop.\,9]{JRT:Lens} assumes the 0-efficiency; our reformulation incorporates minimal triangulations of $\RP^3$ and $L(3,1)$ that are not 0-efficient. Also, extending to prime 3-manifolds, we included $S^2 \nts \times \nts S^1$ in the degree 2 statement.
\end{rem}

\subsection{Rank-1 Coloring} 
\label{ssec:Coloring1}

We recall the notion of rank-1 $\bbZtwo$-coloring from \cite{JRT:Lens}. Let $M$ be a 3-manifold with a rank-1 subgroup $H=\{\zclass,\aclass\} \leqslant H^1(M;\bbZtwo)$, and $\tri$ be a one-vertex triangulation of $M$; each edge $e$ defines a 1-cycle in $M$. The \emph{rank-1 $\bbZtwo$-coloring} of $\tri$ by $H$, or the \emph{$H$-coloring}, assigns the value $\aclass[e] \in \bbZtwo$ to each edge $e$; an edge $e$ is \emph{$\aclass$-even} or \emph{$H$-even} if $\aclass[e]=0$, and \emph{$\aclass$-odd} or \emph{$H$-odd} if $\aclass[e]=1$.

The cocycle condition constrains the coloring on faces and tetrahedra in $\tri$. Faces are divided into two types; a face is \emph{$\aclass$-even} or \emph{$H$-even} if all three edges are $\aclass$-even, and \emph{$\aclass$-odd} or \emph{$H$-odd} if one edge is $\aclass$-even and other two edges are $\aclass$-odd. Tetrahedra are divided into three types; a tetrahedron is of type $\typee$, $\typet$, or $\typeq$ as follows.

\begin{figure}[h]
\begin{center} 
\includegraphics[height=22mm, width=22mm]{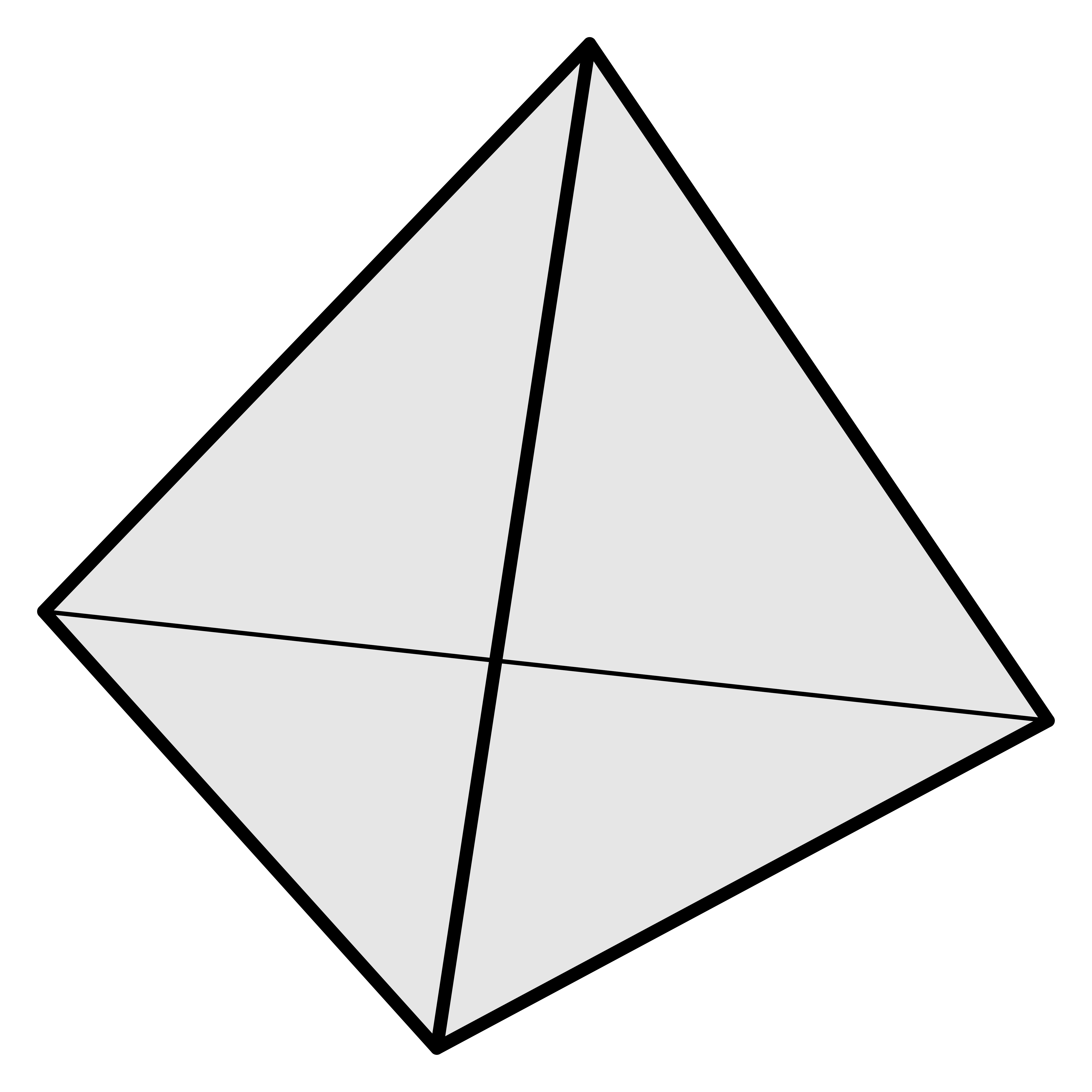} \;\;\;
\includegraphics[height=22mm, width=22mm]{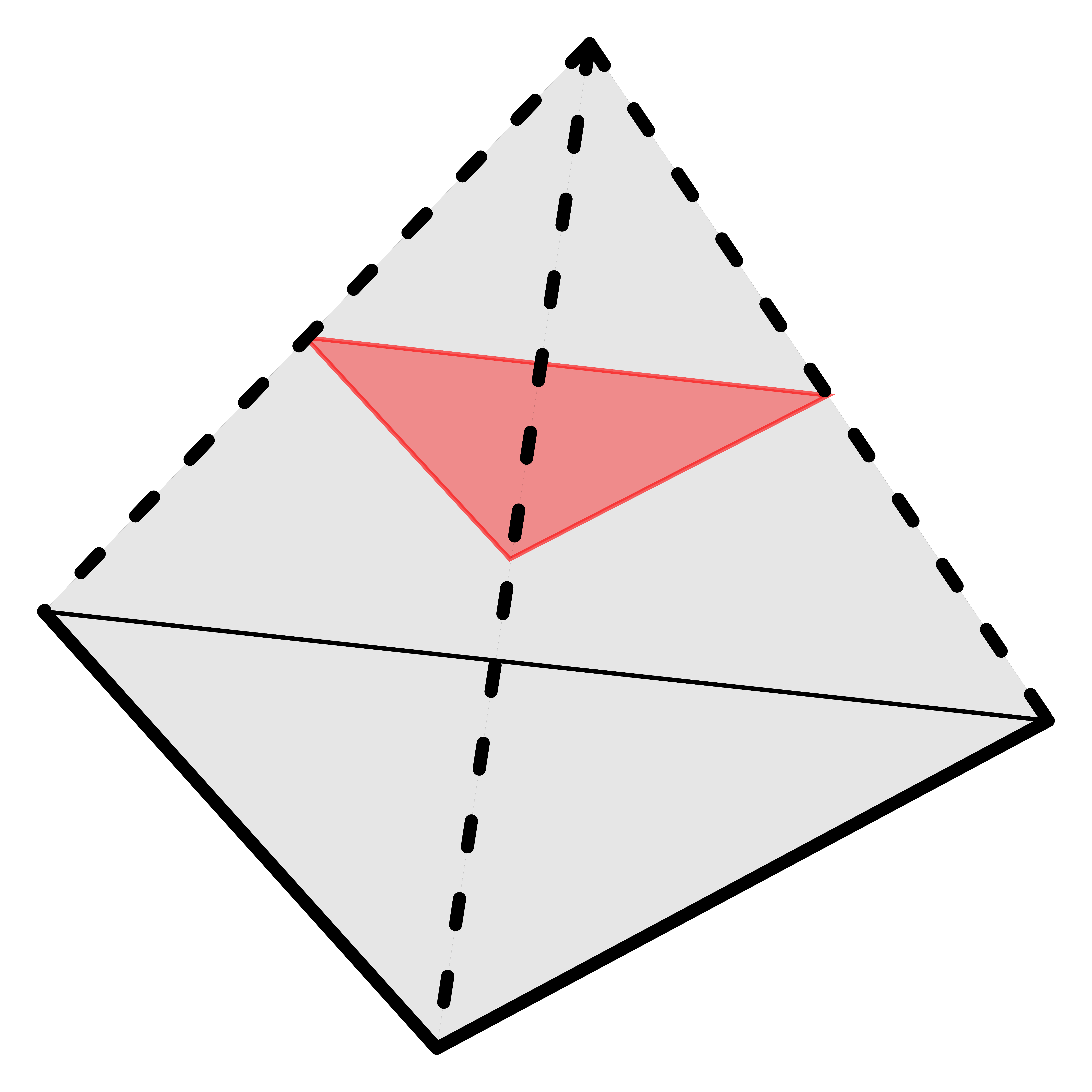} \;\;\;
\includegraphics[height=22mm, width=22mm]{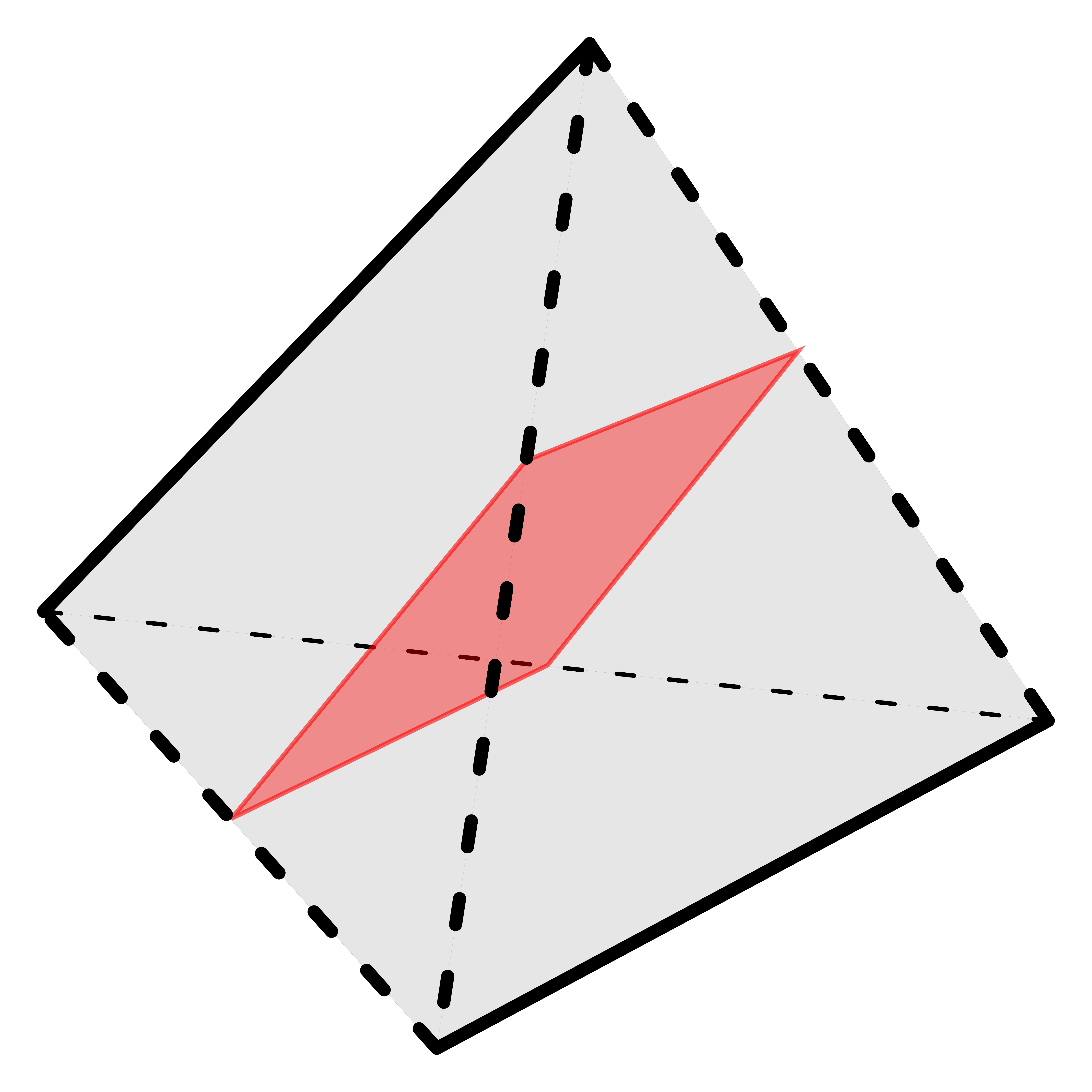}
\end{center}
\vspace{-2mm}
\caption{Rank-1 coloring of Tetrahedra}
\label{fig:types1}
\vspace{-1mm}
\end{figure}

\begin{description}
\item[Type $\typee$]
all six edges are $\aclass$-even.
\item[Type $\typet\tts$]
three edges in a face are $\aclass$-even, and other three edges are $\aclass$-odd.
\item[Type $\ntts\typeq$]
a pair of opposite edges are $\aclass$-even, and other four edges are $\aclass$-odd.
\end{description}
The types $\typee$, $\typet$, and $\typeq$ are called types 3, 2, and 1 respectively in \cite{JRT:Lens}. These three types of tetrahedra are depicted in \hyperref[fig:types1]{Figure~\ref*{fig:types1}}.

\begin{defn}
For a non-zero class $\aclass \in H^1(M;\bbZtwo)$, the \emph{canonical surface} $\asurf$ in $\tri$ is the normal surface whose edge-weight is $\aclass[e] \in \bbZtwo$ on each edge $e$.
\end{defn}

The canonical surface $\asurf$ represents the Poincar\'e dual of $\aclass$. The intersection of $\asurf$ with a tetrahedron $\tau$ is empty if $\tau$ is of type $\typee$, a triangle if $\tau$ is of type $\typet$, and a quadrilateral if $\tau$ is of type $\typeq$. Let us write
\begin{align*}
\Tcount_\typet(\tri):=\;
  &\#\{\text{tetrahedra of type $\typet$}\},\\
\Ecount_{\typee,d}(\tri):=\;
  &\#\{\text{$\aclass$-even edges of degree~$d$}\},
\end{align*}
or simply as $\Tcount_\typet$, $\Ecount_{\typee,d}$; they are related to the Euler characteristic $\chi(\asurf)$.

\begin{lemalpha}[{\cite[Lem.\,12]{JRT:Lens}}] \label{lem:comb1}
Let $M$ be an irreducible 3-manifold and $\tri$ be a minimal triangulation of $M$ with the rank-1 coloring by $\{\zclass,\aclass\} \leqslant H^1(M;\bbZtwo)$. Then, assuming $M \nhomeo \RP^3,L(4,1)$, the surface $\asurf$ and the numbers $\Tcount_\typet$, $\Ecount_{\typee,d}$ satisfy
\begin{align} \label{eqn:comb1}
2\Tcount-4+4 \chi(\asurf)
\;=\;
\sum_{d=3}^\infty (d-4) \Ecount_{\typee,d} + \Tcount_\typet
\;\geqslant\;
\sum_{d=3}^\infty (d-4) \Ecount_{\typee,d}.
\end{align}
\end{lemalpha}

\subsection{Rank-2 Coloring} 
\label{ssec:Coloring2}

We also recall the notion of rank-2 $\bbZtwo$-coloring from \cite{JRT:Z2}. Let $M$ be a 3-manifold with a rank-2 subgroup $H=\{\zclass,\rclass,\gclass,\bclass\} \leqslant H^1(M;\bbZtwo)$, and $\tri$ be a one-vertex triangulation of $M$. For each $i \in \Index:=\{1,2,3\}$, $\iclass$ defines a rank-1 coloring. The \emph{rank-2 $\bbZtwo$-coloring} of $\tri$ by $H$, or the \emph{$H$-coloring}, assigns three values $\iclass[e] \in \bbZtwo$, $i \in \Index$, to each edge $e$. An edge $e$ is \emph{$H$-even} if $\iclass[e]=0$ for all $i \in \Index$, and \emph{$i$-even} if $\iclass[e]=0$ for a unique $i \in \Index$; since $H$ is a rank-2 subgroup, each edge $e$ must be either $H$-even or $i$-even for some $i \in \Index$.

The cocycle condition constrains the coloring on faces and tetrahedra in $\tri$. Faces are divided into three types; a face is \emph{$H$-even} if all three edges are $H$-even, $i$-even if one edge is $H$-even and other two edges are $i$-even for the same $i \in \Index$, and \emph{$H$-odd} if edges are $i$-even, $j$-even, $k$-even with $\{i,j,k\}=\Index$. Tetrahedra are divided into five types; a tetrahedron is said to be of type $\typee$, $\typett$, $\typeqq$, $\typeqtt$, or $\typeqqq$ as follows.

\begin{figure}[h]
\begin{center}
\includegraphics[height=21mm, width=21mm]{typee.pdf} \;
\includegraphics[height=21mm, width=21mm]{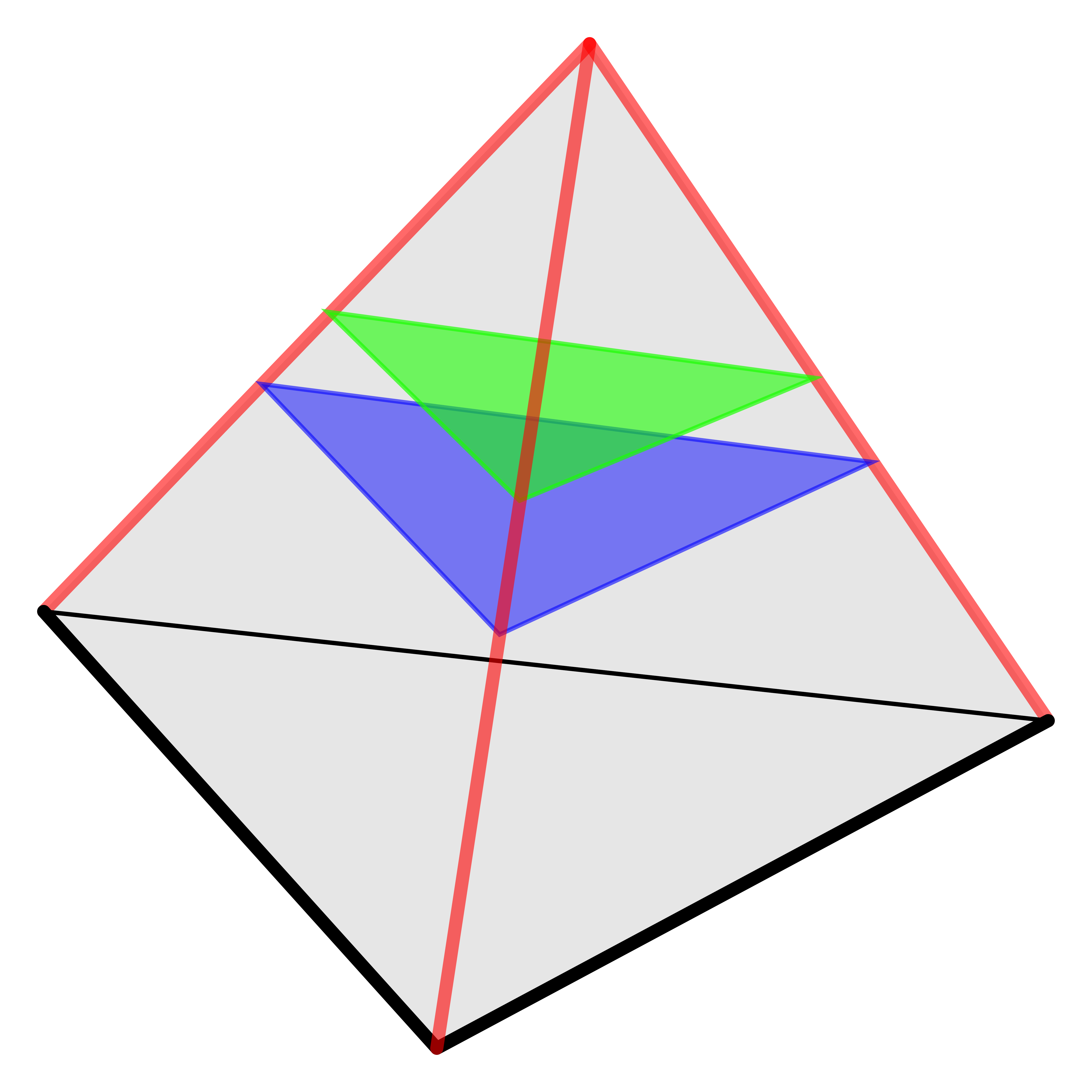} \;
\includegraphics[height=21mm, width=21mm]{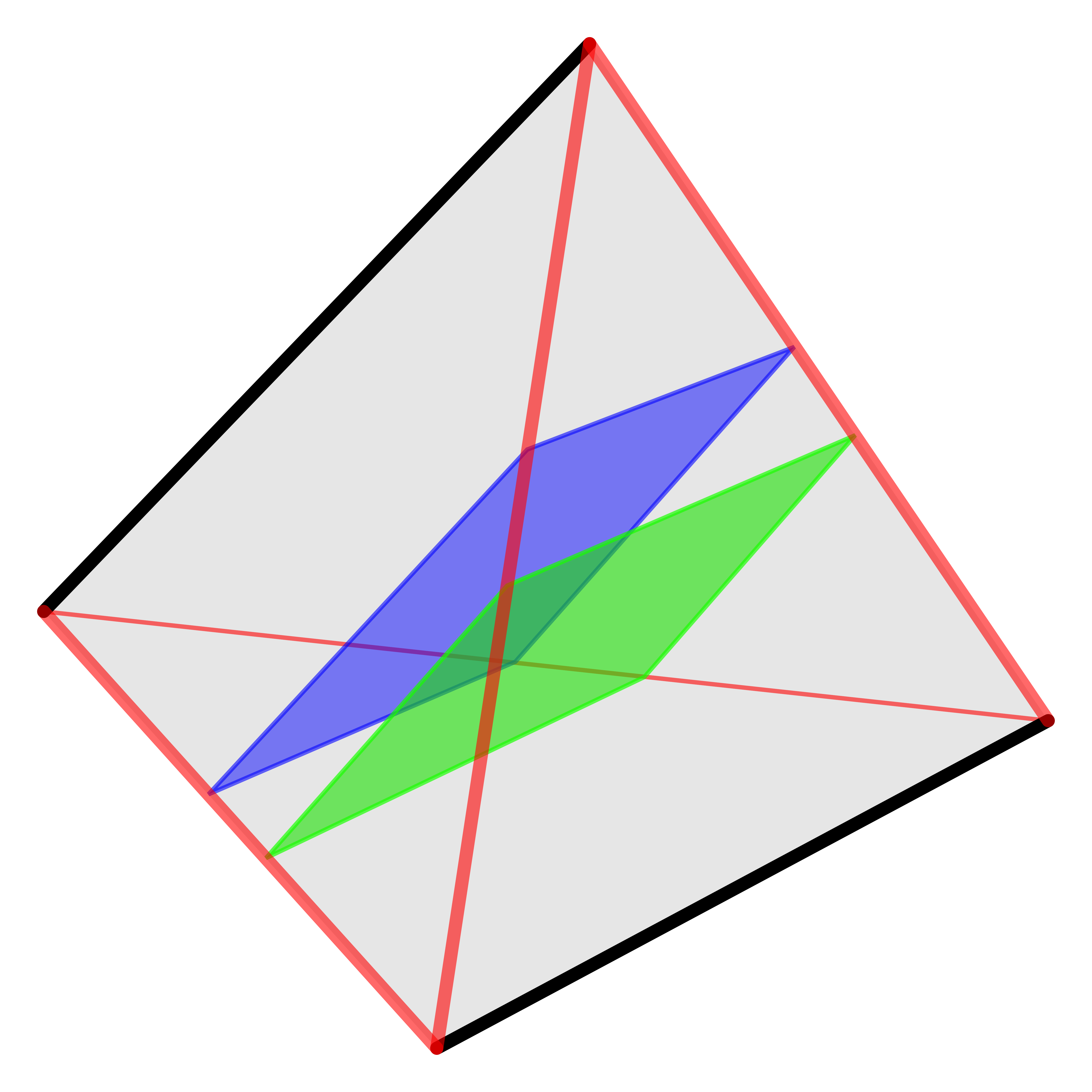} \;
\includegraphics[height=21mm, width=21mm]{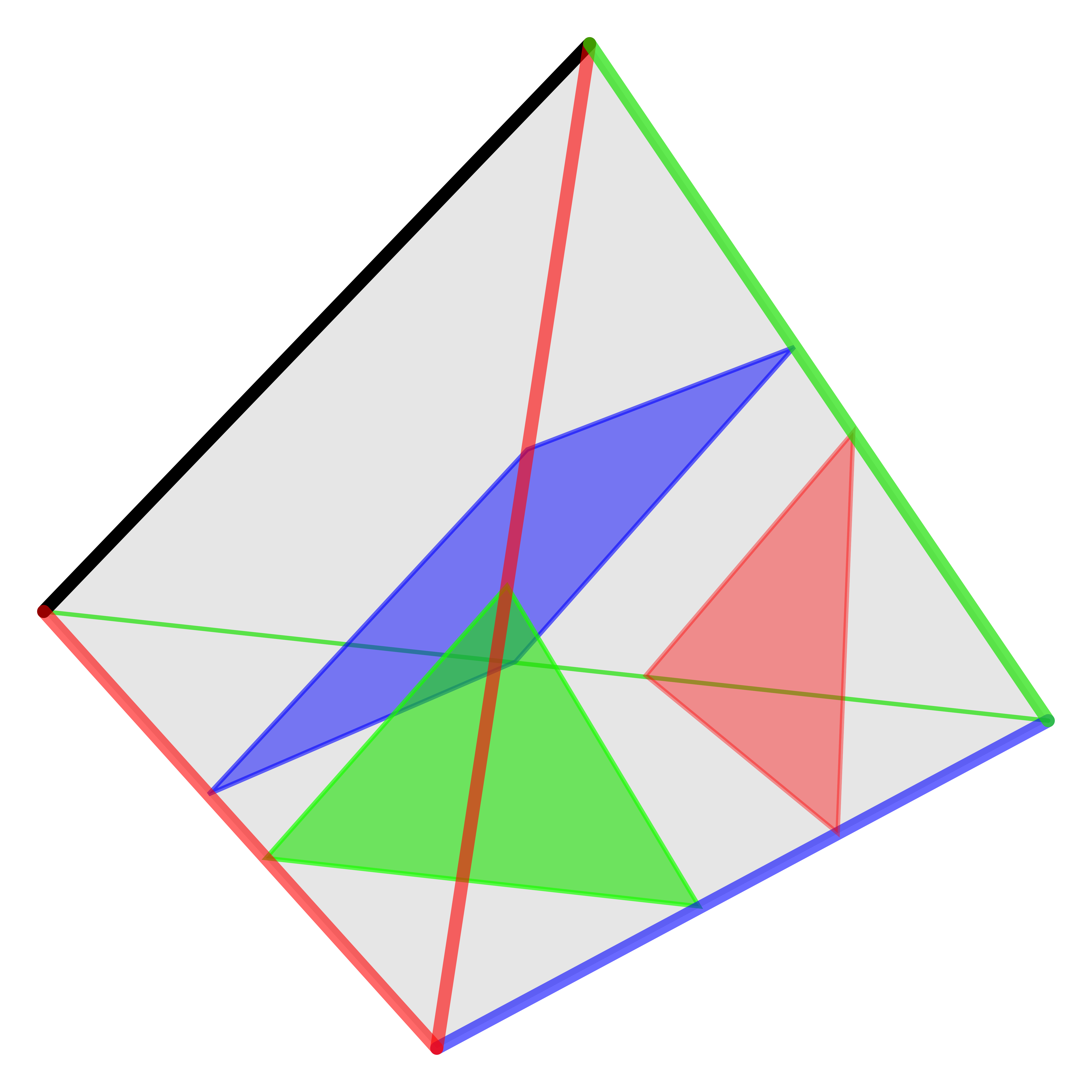} \;
\includegraphics[height=21mm, width=21mm]{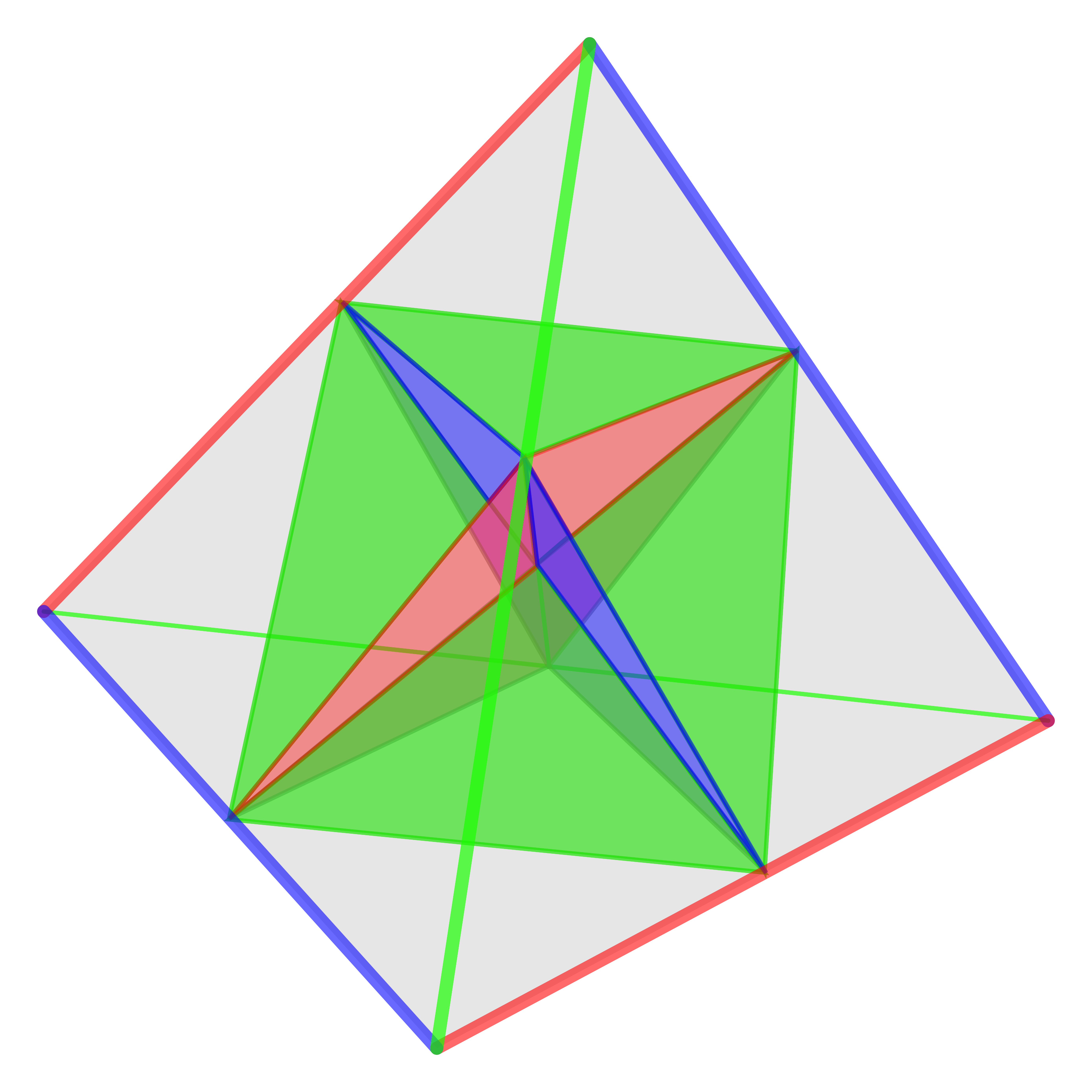}
\end{center}
\vspace{-2mm}
\caption{Rank-2 coloring of tetrahedra.}
\label{fig:types-2}
\vspace{-1mm}
\end{figure}

\begin{description}
\item[Type $\typee$]
all six edges are $H$-even.
\item[Type $\typett\tts$]
three edges in a face are $H$-even and other three edges are $i$-even for the same $i \in \Index$; there are 3 distinct unoriented/oriented sub-types.
\item[Type $\ntts\typeqq$]
a pair of opposite edges are $H$-even and other four edges are $i$-even for the same $i \in \Index$; there are 3 distinct unoriented/oriented sub-types.
\item[Type $\ntts\typeqtt\tts$]
one edge is $H$-even, sharing a face with two $i$-even edges and another face with two $j$-even edges, and the last edge is $k$-even, such that $\{i,j,k\}=\Index$; there are 3 distinct unoriented sub-types and 6 distinct oriented sub-types.
\item[Type $\ntts\typeqqq$]
three pairs of opposite edges are $i$-even, $j$-even, $k$-even respectively, such that $\{i,j,k\}=\Index$; there are 2 distinct oriented sub-types.
\end{description}
The types $\typee$, $\typett$, $\typeqq$, $\typeqtt$, and $\typeqqq$ are called types IV, III, II, I, and V respectively in \cite{JRT:Z2}. These five types of tetrahedra are depicted in \hyperref[fig:types-2]{Figure~\ref*{fig:types-2}}.

The $H$-coloring defines three canonical surfaces, $\isurf=\surf_{\iclass}$, $i \in \Index$. The intersection of $\bigcup_{i \in \Index} \isurf$ with a tetrahedron $\tau$ is empty if $\tau$ is of type $\typee$, two triangles if $\tau$ is of type $\typett$, two quadrilaterals if $\tau$ is of type $\typeqq$, a quadrilateral and two triangles if $\tau$ is of type $\typeqtt$, and three quadrilaterals if $\tau$ is of type $\typeqqq$. Let us write
\begin{align*}
\Tcount_\typett(\tri):=\;
  &\#\{\text{tetrahedra of type $\typett$}\},\\
\Tcount_\typeqtt(\tri):=\;
  &\#\{\text{tetrahedra of type $\typeqtt$}\},\\
\Ecount_{\typee,d}(\tri):=\;
  &\#\{\text{$H$-even edges of degree~$d$}\},
\end{align*}
or simply as $\Tcount_\typett$, $\Tcount_\typeqtt$, $\Ecount_{\typee,d}$; they are related to the Euler characteristic $\chi(\isurf)$, $i \in \Index$. 

\begin{lemalpha}[{\cite[Lem.\,4]{JRT:Z2}}] \label{lem:comb2}
Let $M$ be an irreducible 3-manifold and $\tri$ be a minimal triangulation of $M$ with the rank-2 coloring by $\{\zclass,\rclass,\gclass,\bclass\} \leqslant H^1(M;\bbZtwo)$. Then, the surfaces $\isurf$ and the numbers $\Tcount_\typett$, $\Tcount_\typeqtt$, $\Ecount_{\typee,d}$ satisfy
\begin{align} \label{eqn:comb-2}
2\Tcount-4+2 \sum_{i \in \Index}\chi(\isurf)
\;=\;
\sum_{d=3}^\infty (d-4) \Ecount_{\typee,d} + \Tcount_\typett + \Tcount_\typeqtt
\;\geqslant\;
\sum_{d=3}^\infty (d-4) \Ecount_{\typee,d}.
\end{align}
\end{lemalpha}

\subsection{\texorpdfstring{$\bbZtwo$-Coefficient}{Z/2Z-Coefficient} Norm} 
\label{ssec:Norm}

Following Thurston \cite{Thurston:Norm}, for any closed surface $\surf$, we set $\chineg(\surf):=\max\{0, -\chi(\surf)\}$ if $\surf$ is connected, and $\chineg(\surf):=\sum_{i} \chineg(\surf_i)$ if $\surf=\bigsqcup_i \surf_i$ is disconnected with connected components $\surf_i$. A closed surface $\surf \subset M$ is said to be \emph{$\bbZtwo$-taut} if it is $\chineg$-minimizing in its $\bbZtwo$-homology class and has no $\bbZtwo$-homologically trivial union of components; every component of a $\bbZtwo$-taut surface is (geometrically) incompressible and non-separating. The \emph{$\bbZtwo$-coefficient Thurston norm} \cite{JRT:Z2} of a cohomology class $\aclass \in H^1(M;\bbZtwo)$ is defined be
\[
||\aclass||:=\;\min \{ \chineg(\surf) \mid [M] \smallfrown \aclass=[\surf]\}.
\]
This is conceived as the $\bbZtwo$-coefficient analogue of the original ($\bbR$-coefficient) \emph{Thurston norm}, introduced in \cite{Thurston:Norm}. When $M$ is irreducible and $M \nhomeo \RP^3$, or more generally when $M$ has no $S^2 \times S^1$-summand or $\RP^3$-summand, any $\bbZtwo$-taut dual surface $\surf$ for any non-zero class $\aclass \in H^1(M;\bbZtwo)$ satisfies
\begin{align*}
||\aclass|| = -\chi(\surf).
\end{align*}

Although a $\bbZtwo$-taut surface can be normalized in any triangulation $\tri$, finding it can be quite cumbersome. When $\tri$ is minimal, the canonical surface $\asurf$ is a natural alternative, satisfying the following basic relation.

\begin{lemalpha}[{\cite[Lem.\,1]{JRT:Z2}}] \label{lem:norm-c}
Let $M$ be an irreducible 3-manifold such that $M \nhomeo \RP^3$, with a non-zero class $\aclass \in H^1(M;\bbZtwo)$. Then, with respect to any one-vertex 0-efficient (e.g.\;minimal) triangulation $\tri$ of $M$, the canonical surface $\asurf$ has no $S^2$-components or $\RP^2$-components, and satisfies
\begin{align*}
||\aclass|| \leqslant -\chi(\asurf).
\end{align*}
\end{lemalpha}

The canonical surface $\asurf$ has the lowest edge-weights in its homology class, but need not be $\bbZtwo$-taut in general. It is convenient to work with a class of surfaces that are almost equally manageable and possibly much closer to being $\bbZtwo$-taut.

\begin{defn}
Given a non-zero class $\aclass \in H^1(M;\bbZtwo)$, a \emph{nearly canonical surface} $\surf$ in $\tri$ dual to the class $\aclass$ is a normal surface homologous to $\asurf$, such that its edge-weight is either 0 or 2 on each $\aclass$-even edge and is 1 on each $\aclass$-odd edge.
\end{defn}

The canonical surface $\asurf$ is a special case of a nearly canonical surface; generally, for a given non-zero class $\aclass$, there can also be a large number of nearly canonical surfaces that are not canonical. The next lemma generalizes \hyperref[lem:norm-c]{Lemma~\ref*{lem:norm-c}}, and establishes one basic feature of these surfaces.

\begin{lem} \label{lem:norm-nc}
Let $M$ be an irreducible 3-manifold such that $M \nhomeo \RP^3$, with a non-zero class $\aclass \in H^1(M;\bbZtwo)$. Then, with respect to any one-vertex 0-efficient (e.g.\;minimal) triangulation $\tri$ of $M$, any nearly canonical surface $\surf$ dual to $\aclass$ has no $S^2$-components or $\RP^2$-components, and satisfies
\begin{align*}
||\aclass|| \leqslant -\chi(\surf).
\end{align*}
\end{lem}

\begin{proof}
The inequality $||\aclass|| \leqslant \chineg(\surf)$ holds unconditionally by definition, and the equality $\chineg(\surf) = -\chi(\surf)$ holds if $\surf$ contains no $S^2$-components or $\RP^2$-components. Since $\tri$ is 0-efficient, $\surf$ contains no $S^2$-components that are not vertex-linking. Since $\aclass$ is a non-zero class and $\surf$ is a nearly canonical surface dual to $\aclass$, there exists some $\aclass$-odd edge and the wedge-weight of $\surf$ must be 1 on this edge; this rules out the existence of a vertex-linking $S^2$-component which requires the edge-weights of $\surf$ to be at least 2 on all edges. Finally, since $M \nhomeo \RP^3$ and $M$ is irreducible, $\surf$ contains no $\RP^2$-components. Hence, the equality $\chineg(\surf) = -\chi(\surf)$ holds.
\end{proof}

\subsection{Maximal Layered Solid Tori}
\label{ssec:mLST}

Many examples of minimal triangulations, e.g.\;from \cite{Burton:Thesis, JR:Layered}, are known to contain a layered solid torus as a subcomplex. We shall regard layered solid tori as building blocks in a minimal triangulation $\tri$.

\begin{defn}
Given a triangulation $\tri$ of a 3-manifold $M$, a \emph{layered solid torus $\lst$ in $\tri$}, written $\lst \subset \tri$, is a subcomplex homeomorphic to a solid torus and with the induced triangulation isomorphic to a layered solid torus $\lst$. It is said to be \emph{maximal} if it is not properly contained in any other layered solid torus $\lst' \subset \tri$.
\end{defn}

The counting arguments will be organized according to the configurations of maximal layered solid tori in our triangulation. The essential facts of how these subcomplexes can meet with each other is collected in the following proposition.

\begin{propalpha}[{\cite[Lem.\,20]{JRT:Lens}}] \label{prop:mlst}
If $\tri$ is a 0-efficient minimal triangulation of a 3-manifold and is not a layered lens space, then the intersection of two distinct maximal layered solid tori in $\tri$ is at most an edge in $\tri$.
\end{propalpha}

\section{The Rank-1 Inequality for Prime Manifolds}
\label{sec:Rank1}

In this section, we establish \hyperref[thm:main1]{Theorem~\ref{thm:main1}} regarding a lower bound on the complexity of a prime 3-manifold $M$ with $\rank H^1(M;\bbZtwo) \geqslant 1$. As noted in the introduction, under an extra assumption that $M$ is \emph{atoroidal}, the same lower bound had been established implicitly in \cite{JRT:Lens}, where they proved the optimal bound for lens space; see \hyperref[thm:JRT1]{Theorem~\ref{thm:JRT1}}. We give a unified treatment, and prove the following.

\begin{thm} \label{thm:rank1}
Let $M$ be an orientable connected closed prime 3-manifold with a rank-1 subgroup $\{0,\aclass\} \leqslant H^1(M;\bbZtwo)$, and $\scrT$ be a minimal triangulation of $M$. If $M$ is not a lens space, or if $M$ is a lens space but $\scrT$ is not a layered lens space triangulation of $M$, then we have
\[
\Tcount(\scrT) \geqslant 2+2||\aclass||.
\]
\end{thm}

For $M=L(4,1)$, the unique minimal triangulation $\tri$ is a standard layered lens space, and the statement holds vacuously. For $M=\SxS$ or $\RP^3$, we can also quickly check that the statement holds; $\tri$ is not a standard layered lens space, but we have $\Tcount(\tri)=2$ and $||\aclass||=0$ for the unique non-zero class $\aclass \in H^1(M;\bbZtwo)$. So, we may assume $M \nhomeo \SxS, \RP^3, L(4,1)$, and work with a minimal triangulation $\tri$ that is \emph{not} a standard layered lens space.

Suppose $H=\{\zclass,\aclass\} \leqslant H^1(M;\bbZtwo)$ is a rank-1 subgroup. With respect to the $\bbZtwo$-coloring of $\tri$ by $H$, \hyperref[lem:norm-c]{Lemma~\ref*{lem:norm-c}} and \hyperref[lem:comb1]{Lemma~\ref*{lem:comb1}} together yields
\begin{align}
\label{eqn:roadmap1}
\begin{split}
2\Tcount-4-4||\aclass||
&\;\geqslant\;
2\Tcount-4+4\chi(\asurf) \\
&\;=\;
\sum_{d=3}^\infty (d-4) \Ecount_{\typee,d} + \Tcount_\typet
\;\geqslant\;
\sum_{d=3}^\infty (d-4) \Ecount_{\typee,d}.
\end{split}
\end{align}
We aim to show that the lefthand term is non-negative.
Under an extra topological assumption that $M$ is \emph{atoroidal}, the conclusion of \hyperref[thm:rank1]{Theorem~\ref*{thm:rank1}} is established implicitly in \cite{JRT:Z2} by showing that the righthand term $\sum_{d=3}^\infty (d-4) \Ecount_{\typee,d}$ is bounded below by $-1$. To prove \hyperref[thm:rank1]{Theorem~\ref*{thm:rank1}} in full generality, we work without this lower bound on the righthand term. There are two parts in our argument.

First, we study how the middle term $\sum_{d=3}^\infty (d-4) \Ecount_{\typee,d} + \Tcount_\typet$ fails to be non-negative. Our analysis of the middle term is quite similar to the analysis of the righthand term in \cite{JRT:Lens}, but incorporating the extra term $\Tcount_\typet$ leads to a finer description of the combinatorial structures causing the negativity. 

Second, we exploit the gap between the lefthand term $2\Tcount-4-4||\aclass||$ and the middle term $2\Tcount-4+4\chi(\asurf)$. When the canonical surface $\asurf$ is (geometrically) compressible, there is a gap between $||\aclass||$ and $-\chi(\asurf)$. Analyzing the combinatorics of the triangulation around the surface $\asurf$, we show that $\asurf$ is always sufficiently compressible so that the gap between $||\aclass||$ and $-\chi(\asurf)$ is large enough to make up for the nagativity of the middle term, leading to the desired inequality.

As evident from the discussion above, the middle term in \hyperref[eqn:roadmap1]{(\ref*{eqn:roadmap1})} plays the central role in our argument. Throughout this section, we denote this quantity by
\[
\Icount_1
:=\; 2\Tcount-4+4\chi(\asurf)
\;=\;\sum_{d=3}^\infty (d-4) \Ecount_{\typee,d} + \Tcount_\typet.
\]

\subsection{Demography}
\label{ssec:Demography1}

For the rest of this section, we assume that $M \nhomeo \RP^3, L(4,1)$ is an irreducible 3-manifold with a rank-1 subgroup $H=\{\zclass,\aclass\} \leqslant H^1(M;\bbZtwo)$, and that $\tri$ is a minimal triangulation of $M$ with the $H$-coloring and is not a stanard layered lens space. The \emph{$\typet$-degree} $\dcount_\typet(e)$ of an edge $e$ is defined to be 
\[
\dcount_\typet(e):=\;\#\left\{\begin{array}{cc}
\text{1-simplices in the preimage $\quot^{-1}(e)$ that are incident to}\\
\text{3-simplices in the preimage of type $\typet$ tetrahedra}
\end{array}\right\}.
\]

The set of all $H$-even edges is denoted by $\Evens$. Let us first rewrite the quantity $\Icount_1$ as a sum of contributions from $H$-even edges. For each $e \in \Evens$, we define
\begin{align*}
\icount_1(e):=\;\dcount(e)-4+{\textstyle\frac{\tts1\tts}{3}} \dcount_\typet(e).
\end{align*}

\begin{lem} \label{lem:irewrite1}
With the notations as above, we have
\[
\Icount_1 \,=\, \sum_{e \in \Evens} \icount_1(e).
\]
\end{lem}

\begin{proof}
We note that the $\typet$-degrees of $H$-even edges satisfy $\sum_{e \in \Evens} \dcount_\typet(e)= 3\Tcount_\typet$. Taking the sum of $\icount_1(e)$ over $\Evens$ and regrouping by their degrees, we obtain
\[
\sum_{e \in \Evens} \icount_1(e)=\sum_{e \in \Evens} (\dcount(e)-4)+\sum_{e \in \Evens} {\textstyle\frac{\tts1\tts}{3}}\dcount_\typet(e)=\sum_{d=3}^\infty (d-4)\Ecount_{\typee,d}+\Tcount_\typet=\Icount_1
\]
by the definition of $\Icount_1$ and $\icount_1(e)$, as desired.
\end{proof}

To derive a tractable estimate of $\Icount_1=\sum_{e \in \Evens} \icount_1(e)$, we introduce a certain subset $\Adults \subset \Evens$ and a counting scheme based on a partition of $\Evens$ into subsets $\calH(e)$, $e \in \Adults$.

\begin{defn} \label{defn:edges}
An $H$-even edge is called a \emph{child} edge if it is an interior edge in a maximal layered solid torus, and a child edge of degree~3 is called a \emph{baby} edge; the set of all child edges is denoted by $\Children \subset \Evens$ and the set of all child edges in a maximal layered solid torus $\lst$ is denoted by $\Children(\lst) \subset \Children$. All other $H$-even edges are called \emph{adult} edges; the set of all adult edges is denoted by $\Adults \subset \Evens$.
\end{defn}

By \hyperref[thm:edges_min]{Theorem~\ref*{thm:edges_min}}, every $H$-even edge of degree~3 is a baby edge, while adult edges and non-baby child edges have degree at least 4. Each maximal layered solid torus $\lst$ contains at most one baby edge, possibly some non-baby child edges, and precisely one (if $\lst$ is of type $\typeq$) or three (if $\lst$ is of type $\typee$) adult edges.

\begin{defn} \label{defn:supporter}
The \emph{supporter} of a maximal layered solid torus $\lst$ is defined to~be
\begin{itemize}
\item[($\typeq$)] the unique adult edge $e$ in $\lst$ if $\lst$ is of type $\typeq$, or
\item[($\typee$)] the adult edge $e$ with maximal degree in $\lst$ if $\lst$ is of type $\typee$.
\end{itemize}
If $e \in \Adults$ is the supporter of $\lst$, we say $e$ supports $\lst$ and write $\lst \leadsto e$, and say $e$ supports $e'$ and write $e' \leadsto e$ for any child edge $e' \in \Children(\lst)$.
\end{defn}

Every child edge is supported by a unique adult edge. For our counting purpose, it is convenient to group each child edge $e' \in \Children$ with the adult edge $e \in \Adults$ such that $e' \leadsto e$. So, we define a partiton of $\Evens$ into subsets $\House(e)$, $e \in \Adults$, by setting
\begin{align} \label{eqn:partition}
\House(e):=\;\{e\} \cup \{e' \in \Children \mid e' \leadsto e\}\;=\;\{e\} \cup \bigcup_{\lst \leadsto e} \Children(\lst).
\end{align}
To obtain a tractable estimate of $\Icount_1=\sum_{e \in \Evens} \icount_1(e)$ from this partition, we set
\begin{align*}
\bcount(e):=\;\#\{\text{baby edges $e' \in \House(e)$}\}\;=\,\sum_{\lst \leadsto e} \Ecount_{\typee,3}(\lst),
\end{align*}
and define a counting function $\gcount_1$ on the set $\Adults \subset \Evens$ of adult edges by
\begin{align*}
\gcount_1(e):=\;\dcount(e)-4-\bcount(e)+{\textstyle\frac{\tts1\tts}{3}} \dcount_\typet(e).
\end{align*}

\begin{lem} \label{lem:grewrite1}
With the notations as above, we have
\[
\Icount_1 \,\geqslant\, \sum_{e \in \Adults} \gcount_1(e).
\]
\end{lem}

\begin{proof}
In this proof, we write $a$ for adult edges, $c$ for child edges, and $e$ for arbitrary edges in $\Evens$. Using the partition of $\Evens$ into subsets $\House(a)$, $a \in \Adults$, as defined in \hyperref[eqn:partition]{(\ref*{eqn:partition})},
\begin{align*}
\Icount_1
&=\sum_{e \in \Evens} \icount_1(e)=\sum_{a \in \Adults} \, \sum_{e \in \House(a)} \icount_1(e)
=\sum_{a \in \Adults} \, \bigg(\icount_1(a)+\sum_{\lst \leadsto a} \, \sum_{c \in \Children(\lst)} \icount_1(c)\bigg) 
\end{align*}
by \hyperref[lem:irewrite1]{Lemma~\ref*{lem:irewrite1}}. For any child edge $c \in \Children(\lst)$, we have $\dcount_\typet(c)=0$ since $\lst$ contains no tetrahedra of type $\typet$. Grouping together the edges of the same degree as before,
\begin{align*}
\sum_{c \in \Children(\lst)} \icount_1(c)
=\sum_{c \in \Children(\lst)} (\dcount(c)-4)
=\sum_{d=3}^\infty (d-4) \, \Ecount_{\typee,d} (\lst)
\geqslant -\Ecount_{\typee,3} (\lst).
\end{align*}
Thus, combining the observations above, we have
\begin{align*}
\Icount_1
&=\sum_{a \in \Adults} \, \bigg(\icount_1(a)+\sum_{\lst \leadsto a} \, \sum_{c \in \Children(\lst)} \icount_1(c)\bigg) \\
&\geqslant \sum_{a \in \Adults} \, \bigg(\dcount(a)-4+{\textstyle\frac{\tts1\tts}{3}} \dcount_\typet(a)-\sum_{\lst \leadsto a}\Ecount_{\typee,3} (\lst) \bigg)
=\sum_{a \in \Adults} \gcount_1(a)
\end{align*}
by the definition of $\icount_1(e)$ and $\gcount_1(e)$, as desired.
\end{proof}

\subsection{Insolvent Adults}
\label{ssec:Insolvent1}

A wishful inequality $\Icount_1 \geqslant 0$ follows from \hyperref[lem:grewrite1]{Lemma~\ref*{lem:grewrite1}}, if $\gcount_1(e) \geqslant 0$ for all adult edges $e \in \Adults$. However, this fails to hold in general; so, we shall study edges $e \in \Adults$ with $\gcount_1(e)<0$. We use the following terminology.

\begin{defn}
An adult edge $e \in \Adults$ is said to be \emph{solvent} if $\gcount_1(e) \geqslant 0$, and it is said to be \emph{insolvent} if $\gcount_1(e)<0$.
\end{defn}

We aim to identify the local combinatorics around solvent/insolvent adult edges. Although the $\bcount(e)$ term in $\gcount_1(e)=\dcount(e)-4-\bcount(e)+\frac{1}{3}\dcount_\typet(e)$ cannot be articulated merely by the local data around $e \in \Adults$, we have an estimate
\begin{align*}
\scount(e):=\;\#\{\text{maximal layered solid tori $\lst \ni e$}\}\;\geqslant\;\bcount(e)
\end{align*}
as each maximal layered solid torus contains at most one baby edge.
The quantity $\scount(e)$ is constrained, as shown in \hyperref[prop:mlst]{Proposition~\ref*{prop:mlst}}, by the local structure around $e$. In turn, we can relate $\gcount_1(e)$ to the combinatorics around $e$ and derive basic criteria for an adult edge $e \in \Adults$ to be solvent, in terms of $\dcount(e)$, $\scount(e)$, and $\bcount(e)$.

\begin{lem} \label{lem:halfdeg1}
If an adult edge $e \in \Adults$ satisfies one of the following conditions, then $\gcount_1(e) \geqslant \dcount(e)-4-\bcount(e) \geqslant 0$ holds, and hence in particular $e$ is a solvent adult edge:
(i) $\dcount(e) \geqslant 7$;
(ii) $\dcount(e)=6$, $\bcount(e) \leqslant 2$;
(iii) $\dcount(e)=5$, $\bcount(e) \leqslant 1$;
(iv) $\bcount(e)=0$.
\par
Thus, every insolvent adult edge $e$ must satisfy one of the following conditions:
(1) $\dcount(e)=6$, $\scount(e)=\bcount(e)=3$;
(2) $\dcount(e)=5$, $\scount(e)=\bcount(e)=2$;
(3) $\dcount(e)=4$, $\scount(e)=\bcount(e)=2$;
(4) $\dcount(e)=4$, $\scount(e)=2$, $\bcount(e)=1$;
(5) $\dcount(e)=4$, $\scount(e)=\bcount(e)=1$.
\end{lem}

\begin{proof}
We first record a consequence of \hyperref[prop:mlst]{Proposition~\ref*{prop:mlst}}, extracted from the proof of \cite[Thm.\,5]{JRT:Lens}: for any adult edge $e \in \Adults$, $\lfloor\dcount(e)/2\rfloor \geqslant \scount(e) \geqslant \bcount(e)$. This implies $\gcount_1(e) \geqslant \dcount(e)-4-\bcount(e) \geqslant \dcount(e)-4-\lfloor \dcount(e)/2 \rfloor=\lceil \dcount(e)/2 \rceil-4$; since $\lceil \dcount(e)/2 \rceil \geqslant 4$ for $\dcount(e) \geqslant 7$, the case (i) follows. The cases (ii)-(iii) are trivial and the case (iv) follows from $\dcount(e) \geqslant 4$ for $e \in \Adults$. This completes the proof of the first statement; the second statement follows from the first, together with $\lfloor\dcount(e)/2\rfloor \geqslant \scount(e) \geqslant \bcount(e)$.
\end{proof}

Although \hyperref[lem:halfdeg1]{Lemma~\ref*{lem:halfdeg1}} is formulated with $\gcount_1(e)$, the profiles of edges given in the lemma coincide with the ones studied in \cite[\S3.2,\,\S5.2,\,\S6.1]{JRT:Lens}. \hyperref[lem:halfdeg1]{Lemma~\ref*{lem:halfdeg1}} does not establish $\gcount_1(e) < 0$ under the conditions (1)-(5); we have $\dcount(e)-4-\bcount(e)<0$ for these cases, but we must consider the extra term $\frac{1}{3} \dcount_\typet(e)$ in $\gcount_1(e)$. We establish this converse statement presently in \hyperref[lem:eprofile1]{Lemma~\ref*{lem:eprofile1}} and \hyperref[lem:tprofile1]{Lemma~\ref*{lem:tprofile1}}. The degree~4 cases (3)-(5) in these lemmas are analyzed already and utilized in the proof of \cite[Thm.\,5]{JRT:Lens}; the remaining cases (1)-(2) are essential in our subsequent arguments. We give a direct unified proof for all cases (1)-(5) of these lemmas.

\begin{lem} \label{lem:eprofile1}
The following statements (in which $\lst$ denotes a maximal layered solid torus) holds for any adult edge $e \in \Adults$:
\begin{enumerate} 
\item $\dcount(e)=6$, $\scount(e)=3$, $\bcount(e)=3$ $\Rightarrow$
$\lst \leadsto e$ and $\lst$ is of type $\typeq$ for each $\lst \ni e$;
\item $\dcount(e)=5$, $\scount(e)=2$, $\bcount(e)=2$ $\Rightarrow$
$\lst \leadsto e$ and $\lst$ is of type $\typeq$ for each $\lst \ni e$;
\item $\dcount(e)=4$, $\scount(e)=2$, $\bcount(e)=2$ $\Rightarrow$
$\lst \leadsto e$ and $\lst$ is of type $\typeq$ for each $\lst \ni e$;
\item $\dcount(e)=4$, $\scount(e)=2$, $\bcount(e)=1$ $\Rightarrow$
$\lst \leadsto e$ and $\lst$ is of type $\typeq$ for one $\lst \ni e$;
\item $\dcount(e)=4$, $\scount(e)=1$, $\bcount(e)=1$ $\Rightarrow$
$\lst \leadsto e$ and $\lst$ is of type $\typeq$ for $\lst \ni e$;
\end{enumerate}
moreover, in all cases, we must have $\dcount_\lst(e)=1$ for each $\lst \ni e$.
\end{lem}

\begin{proof}
We first record two consequences of \hyperref[lem:edges_lst]{Lemma~\ref*{lem:edges_lst}}. (A) if $\dcount_\lst(e)=1$ and $\lst \leadsto e$, then $\lst$ is of type $\typeq$; for $\dcount_\lst(e)=1$ implies that $e$ is the boundary edge of minimal degree in $\lst$, and $\lst$ must be of type $\typeq$ to satisfy $\lst \leadsto e$. (B) if $\dcount_\lst(e)=2$, then $\scount(e)>\bcount(e)$; for $\dcount_\lst(e)=2$ forces $\lst=\lst_{\{1,2,3\}}$ with no child edge, hence $\scount(e)>\bcount(e)$. We refer to these observatoins as (A) and (B) below.

Consider the cases (1)-(3) with $\lfloor \dcount(e)/2 \rfloor = \scount(e)=\bcount(e)$. Note that $\scount(e) = \bcount(e)$ forces $\lst \leadsto e$ for each $\lst \ni e$. By \hyperref[prop:mlst]{Proposition~\ref*{prop:mlst}}, $\dcount_\lst(e)=1$ for each $\lst \ni e$, except for at most one $\lst$ with $\dcount_\lst(e)=2$ when $\dcount(e)=5$. Ruling out $\dcount_\lst(e)=2$ by (B), we have $\dcount_\lst(e)=1$ for every $\lst \ni e$; hence, each $\lst \ni e$ is of type $\typeq$ by (A).

Next, consider the case (4) with $\dcount(e)=4$, $\scount(e)=2$, $\bcount(e)=1$. By \hyperref[prop:mlst]{Proposition~\ref*{prop:mlst}}, we have $\dcount_\lst(e)=1$ for each $\lst \ni e$. Since $\bcount(e)=1$, there exists one $\lst \ni e$ such that $\lst \leadsto e$; this $\lst$ is of type $\typeq$ by (A).

Finally, consider the case (5) with $\dcount(e)=4$, $\scount(e)=\bcount(e)=1$. We first note that $\scount(e) = \bcount(e)$ forces $\lst \leadsto e$ for $\lst \ni e$. \emph{A priori}, we may have $\dcount_\lst(e)=1,2$ or $3$. We can rule out $\dcount_\lst(e)=2$ by (B). If $\dcount_\lst(e)=3$, $e$ meets one tetrahedron $\tau$ not contained in $\lst$. In this case, $\tau$ is layered along $e$; adjoining $\tau$ to $\lst$ yields a layered solid torus containing $\lst$ or a layered lens space, and we have a contradiction either way. Hence, we must have $\dcount_\lst(e)=1$, and $\lst \ni e$ is of type $\typeq$ by (A). 
\end{proof}

\begin{rem}
For the other $\lst' \ni e$ in the case (4), $\dcount_{\lst'}(e)=1$ and either (i) $\lst' \not\leadsto e$ and $\lst'$ is of type $\typee$, or (ii) $\lst \leadsto e$ and $\lst'$ is of type $\typeq$, with no baby edge.
\end{rem}

\begin{lem} \label{lem:tprofile1}
An adult edge $e \in \Adults$ is insolvent if and only if the types of tetrahedra incident to $e$, expressed by cyclically ordered $\dcount(e)$-tuples of type-symbols up to dihedral symmetry (with a dot, such as $\dot\typeq$, if the underlying tetrahedron is in a layered solid torus $\lst$ and with two dots, such as $\ddot\typeq$, if this $\lst$ contains a baby edge $e' \leadsto e$), is one of the following:
\begin{enumerate}
\item $\dcount(e)=6$, $\scount(e)=3$, $\bcount(e)=3$:
$(\typeq,\ddot\typeq,\typeq,\ddot\typeq,\typeq,\ddot\typeq)$;
\item $\dcount(e)=5$, $\scount(e)=2$, $\bcount(e)=2$:
$(\typeq,\ddot\typeq,\typeq,\ddot\typeq,\typeq)$ or $(\typet,\ddot\typeq,\typeq,\ddot\typeq,\typet)$;
\item $\dcount(e)=4$, $\scount(e)=2$, $\bcount(e)=2$:
$(\typeq,\ddot\typeq,\typeq,\ddot\typeq)$;
\item $\dcount(e)=4$, $\scount(e)=2$, $\bcount(e)=1$:
$(\typeq,\ddot\typeq,\typeq,\dot\typeq)$ or $(\typet,\ddot\typeq,\typet,\dot\typee)$;
\item $\dcount(e)=4$, $\scount(e)=1$, $\bcount(e)=1$:
$(\typeq,\ddot\typeq,\typeq,\typeq)$, $(\typet,\ddot\typeq,\typeq,\typet)$, or $(\typet,\ddot\typeq,\typet,\typee)$.
\end{enumerate}
In all cases, the edge $e$ must be incident to $\dcount(e)$ distinct tetrahedra.
\end{lem}

\begin{proof}
For each case in the list, we can readily verify $\gcount_1(e)<0$, i.e.\;$e$ is insolvent. For the converse, suppose $e$ is insolvent. Possible combinations (1)-(5) of $\dcount(e), \scount(e), \bcount(e)$ are given in \hyperref[lem:halfdeg1]{Lemma~\ref*{lem:halfdeg1}}. The types of tetrahedra, incident to $e$ and contained in maximal layered solid tori, are given in \hyperref[lem:eprofile1]{Lemma~\ref*{lem:eprofile1}}; they are always type $\typeq$ except for one of type $\typee$ in case (4). The coloring on the boundary faces of maximal layered solid tori restricts the possible types of remaining tetrahedra; the only possible combinations are the ones listed in the statement of the lemma.

Suppose for contradiction that these tetrahedra are not distinct, and let $\tau$ be a terahedron incident to $e$ more than once; by \hyperref[prop:mlst]{Proposition~\ref*{prop:mlst}}, the \emph{undotted} type-symbol of $\tau$ must appear more than once in the cyclic $\dcount(e)$-tuples of type-symbols.

Suppose $\tau$ is of type $\typeq$, appearing twice non-consecutively around $e$; this occurs in
$(\typeq,\ddot\typeq,\typeq,\ddot\typeq,\typeq,\ddot\typeq)$,
$(\typeq,\ddot\typeq,\typeq,\ddot\typeq,\typeq)$,
$(\typeq,\ddot\typeq,\typeq,\ddot\typeq)$,
$(\typeq,\ddot\typeq,\typeq,\dot\typeq)$,
$(\typeq,\ddot\typeq,\typeq,\typeq)$.
In the cyclic order, these symbols $\typeq$ are adjacent to a symbol $\ddot\typeq$ between them, representing the outermost tetrahedron in a maximal layered solid torus $\lst$. Two faces of $\tau$ are identified with two boundary faces of $\lst$, respecting the coloring. Since $\tri$ is not a layered lens space, it follows that $\tau$ is layered on a $\aclass$-odd boundary edge of $\lst$, producing a larger layered solid torus by \hyperref[lem:embedded]{Lemma~\ref*{lem:embedded}}. This contradicts the maximality of $\lst$.

Suppose $\tau$ is of type $\typeq$, appearing twice consecutively around $e$; this occurs in
$(\typeq,\ddot\typeq,\typeq,\ddot\typeq,\typeq)$,
$(\typeq,\ddot\typeq,\typeq,\typeq)$.
In the cyclic order, one of these symbols $\typeq$ is adjacent to a symbol $\ddot\typeq$, representing the outermost tetrahedron in a maximal layered solid torus $\lst$. Since $\tau$ appears twice consecutively, two faces of $\tau$ must be identified, respecting the coloring. It follows that this face-identification of $\tau$ produces $\lst_{\{1,2,3\}}$ that shares a face with $\lst$. This contradicts \hyperref[prop:mlst]{Proposition~\ref*{prop:mlst}}.

Suppose $\tau$ is of type $\typet$, appearing twice non-consecutively around $e$; this occurs in
$(\typet,\ddot\typeq,\typet,\dot\typee)$,
$(\typet,\ddot\typeq,\typet,\typee)$.
This requires that the type $\typee$ tetrahedron shares two distinct $\aclass$-even faces (incident to $e$) with $\tau$; this is impossible for the type $\typet$ tetrahedron $\tau$.

Suppose $\tau$ is of type $\typet$, appearing twice consecutively around $e$; this occurs in
$(\typet,\ddot\typeq,\typeq,\ddot\typeq,\typet)$,
$(\typet,\ddot\typeq,\typeq,\typet)$.
In the cyclic order, each of these symbols $\typet$ is adjacent to a symbol $\ddot\typeq$ or $\typeq$. In both cases, these type $\typeq$ tetrahedra must be distinct. Moreover, since the $\tau$ appears twice consecutively, two faces in $\tau$ must be identified. This requires that $\tau$ shares a $\aclass$-odd face with each of the type $\typeq$ tetrahedra, and has two $\aclass$-odd faces identified; this is impossible for the type $\typet$ tetrahedron $\tau$.

These cases cover all occurrences of undotted type-symbols, repeated in a cyclic $\dcount(e)$-tuple from the list; in all cases, they cannot represent a tetrahedron appearing twice around $e$. Hence, tetrahedra around an insolvent edge $e$ are distinct.
\end{proof}

An \emph{edge flip} is a re-triangulating operation, replacing an edge of degree~4 incident to four distinct tetrahedra with another such edge, while preserving $\Tcount$; this amounts to replacing an edge connecting a pair of opposite vertices of an octahedron with another such edge. It never produces, or applies to, a triangulation with no edge of degree~4 incident to four distinct tetrahedra, e.g.\,a layered lens space.

Edge flips are used in \cite[\S6.1]{JRT:Lens} to eliminate insolvent edges of degree~4, listed as cases (3)-(5) in \hyperref[lem:tprofile1]{Lemma~\ref*{lem:tprofile1}}. If such an edge exists, a suitably chosen edge flip reduces one of the following quantities without increasing the other: (i) the number of maximal layered solid tori of type $\typeq$ incident to insolvent edges of degree~4, or (ii) the number of tetrahedra of type $\typee$. Since an insolvent edge of degree~4 is incident to at least one maximal layered solid torus of type $\typeq$, it follows that we can eliminate them by a finite number of edge flips.
For reference, we extract the following statement from the proof of \cite[Thm.\,5]{JRT:Lens}.

\begin{propalpha}[{\cite[\S6.1]{JRT:Lens}}] \label{prop:edgeflip1}
Let $M \nhomeo \RP^3,L(4,1)$ be an irreducible 3-manifold with a rank-1 subgroup $H=\{\zclass,\aclass\} < H^1(M;\bbZtwo)$. If $M$ admits a minimal triangulation that is not a layered lens space, then $M$ also admits a minimal triangulation with no insolvent edge of degree~4 with respect to the $H$-coloring, that is again not a layered lens space.
\end{propalpha}

\begin{rem}
At this point, one can establish \hyperref[thm:rank1]{Theorem~\ref*{thm:rank1}} for \emph{atoroidal} manifolds. If an insolvent edge $e'$ of degree~5 or 6 exists, arguments in \cite[Prop.\,28]{JRT:Lens} assure that it is the unique insolvent edge. We have $\gcount_1(e) \geqslant 0$ for all $e \in \Adults$ except for possibly one insolvent edge $e'$ with $\gcount_1(e') \geqslant -1$. Hence, \hyperref[eqn:roadmap1]{(\ref*{eqn:roadmap1})} and \hyperref[lem:grewrite1]{Lemma~\ref*{lem:grewrite1}} together yield $2\Tcount-4-4||\aclass|| \geqslant \Icount_1 \geqslant \sum_{e \in \Adults} \gcount_1(e) \geqslant -1$. Since $2\Tcount-4-4||\aclass||$ is an even integer, we have $2\Tcount-4-4||\aclass|| \geqslant 0$ and thus $\Tcount \geqslant 2+2||\aclass||$.
\end{rem}

\subsection{Decent Adults}
\label{ssec:Decent1}

Invoking \hyperref[prop:edgeflip1]{Proposition~\ref*{prop:edgeflip1}}, we may now assume that $\tri$ has no insolvent edges of degree~4 with respect to the rank-1 $H$-coloring. Every insolvent edge $e$ must occur as one of the following cases from \hyperref[lem:tprofile1]{Lemma~\ref*{lem:tprofile1}}.
\begin{itemize}
\item[(1)] $\dcount(e)=6$, $\scount(e)=\bcount(e)=3$: $(\typeq,\ddot\typeq,\typeq,\ddot\typeq,\typeq,\ddot\typeq)$ with $\gcount_1(e)=-1$;
\item[(2a)] $\dcount(e)=5$, $\scount(e)=\bcount(e)=2$: $(\typeq,\ddot\typeq,\typeq,\ddot\typeq,\typeq)$ with $\gcount_1(e)=-1$;
\item[(2b)] $\dcount(e)=5$, $\scount(e)=\bcount(e)=2$: $(\typet,\ddot\typeq,\typeq,\ddot\typeq,\typet)$ with $\gcount_1(e)=-\frac{\tts1\tts}{3}$.
\end{itemize}
We shall now make small adjustments to the counting function $\gcount_1$. With the adjustments, we see that not every insolvent edge is troublesome in our counting. We refer to manageable ones as \emph{decent} edges, and troublesome ones as \emph{rogue} edges.

\begin{defn} \label{defn:rogue}
Distinct adult edges $e,e' \in \Adults$ are said to be \emph{neighbors} of each other if there is a (necessarily $H$-even) face containing both $e$ and $e'$; an adult edge is said to be \emph{isolated} if it has no neighbors. A insolvent edge is said to be \emph{decent} if it has a solvent neighbor, and is said to be \emph{rogue} otherwise.
\end{defn}

An insolvent edge in the case (1) or (2a) is incident to no $H$-even faces, and hence always isolated and rogue. An insolvent edge in the case (2b) is not isolated, and it may be decent or rogue. We shall see that a decent edge $e'$, necessarily in case (2b), stays out of troubles because its ``generous'' solvent neighbor $e$ can afford to give up $\frac{\tts1\tts}{3}$ from $\gcount_1(e)$ and pass it onto $e'$ with $\gcount_1(e')=-\frac{\tts1\tts}{3}$. Let us formalize such adjustments; for each $e \in \Adults$, we define
\begin{align*}
\acount_1(e):=\;\begin{cases}
-\frac{\tts1\tts}{3} \times \#\{\text{insolvent neighbors of $e$}\} \qquad \text{if $e$ is solvent},\\
+\frac{\tts1\tts}{3} \times \#\{\text{solvent neighbors of $e$}\} \qquad \text{if $e$ is insolvent},
\end{cases}
\end{align*}
and modify our counting function $\gcount_1(e)$ with this adjustment term by setting
\[
\fcount_1(e):=\;\gcount_1(e)+\acount_1(e)\;=\;\dcount(e)-4-\bcount(e)+{\textstyle\frac{\tts1\tts}{3}\dcount_\typet(e)}+\acount_1(e).
\]

\begin{lem} \label{lem:frewrite1}
With the notations as above, we have
\[
\Icount_1 \,\geqslant\, \sum_{e \in \Adults} \fcount_1(e).
\]
\end{lem}

\begin{proof}
By definition, an increase by the increment $+\frac{\tts1\tts}{3}$ for an insolvent edge can be paired uniquely with a deduction $-\frac{\tts1\tts}{3}$ for a solvent edge. Together with \hyperref[lem:grewrite1]{Lemma~\ref*{lem:grewrite1}},
\[
\Icount_1 \geqslant \sum_{e \in \Adults} \gcount_1(e)=\sum_{e \in \Adults} \fcount_1(e)
\]
as the total adjustments on all insolvent and solvent edges cancel out. 
\end{proof}

By \hyperref[lem:frewrite1]{Lemma~\ref*{lem:frewrite1}}, we may use $\fcount_1(e)$ in place of $\gcount_1(e)$ for our counting. In particular, we have $\Icount_1 \geqslant 0$ if $\fcount_1(e) \geqslant 0$ for all $e \in \Adults$. A small but significant advantage of this modified counting function $\fcount_1(e)$ is evident in the following lemma which shows that the only problematic insolvent edges are the rogue edges.

\begin{lem} \label{lem:decent1}
An adult edge $e \in \Adults$ satisfies $\fcount_1(e) \geqslant 0$ if and only if the edge is either (i) a solvent edge or (ii) a decent insolvent edge.
\end{lem}

The necessity is immediate; a rogue edge $e$ satisfies $\gcount_1(e)<0$, $\acount_1(e)=0$, and hence $\fcount_1(e)=\gcount_1(e)+\acount_1(e)<0$. We shall verify the sufficiency.

\begin{proof}
Suppose first that an edge $e$ is a decent insolvent edge. Having a neighbor forces it to be an edge from the case (2b) with $\gcount_1(e)=-\frac{\tts1\tts}{3}$, and having at least one solvent neighbor guarantees $\acount_1(e) \geqslant \frac{\tts1\tts}{3}$. Hence $\fcount_1(e)=\gcount_1(e)+\acount_1(e) \geqslant 0$.

Suppose now that $e$ is a solvent edge. We have $\dcount(e)-4-\bcount(e) \geqslant 0$ by \hyperref[lem:halfdeg1]{Lemma~\ref*{lem:halfdeg1}}. Since $\fcount_1(e)=\gcount_1(e)+\acount_1(e)=\dcount(e)-4-\bcount(e)+{\textstyle\frac{\tts1\tts}{3}\dcount_\typet(e)}+\acount_1(e)$, it suffices to show ${\textstyle\frac{\tts1\tts}{3}\dcount_\typet(e)}+\acount_1(e) \geqslant 0$, or equivalently
\[
\dcount_\typet(e) \;\geqslant\; -3\acount_1(e) \;=\; \#\{\text{insolvent neighbors of $e$}\}.
\]
By definition, if the solvent edge $e$ has an insolvent neighbor $e'$, there exists a $H$-even face containing $e$ and $e'$; moreover, by \hyperref[lem:tprofile1]{Lemma~\ref*{lem:tprofile1}}, this $H$-even face must be the common face of two distinct tetrahedra of type $\typet$. Tetrahedra of type $\typet$, each containing the edge $e$ and meeting another tetrahedron of type $\typet$ along its $H$-even face, are naturally paired up along their $H$-even faces; the number of such pairs, and hence the number of $H$-even faces shared between them, is at most $\lfloor \dcount_\typet(e)/2 \rfloor$. Each $H$-even face between such a pair contains at most two insolvent neighbors of $e$. Hence, we have $\dcount_\typet(e) \geqslant 2\lfloor \dcount_\typet(e)/2 \rfloor \geqslant \#\{\text{insolvent neighbors of $e$}\}$ as desired.
\end{proof}

\subsection{Rogue Adults}
\label{ssec:Rogue1}

By \hyperref[lem:frewrite1]{Lemma~\ref*{lem:frewrite1}} and \hyperref[lem:decent1]{Lemma~\ref*{lem:decent1}}, the inequality $\Icount_1 \geqslant 0$ fails only in the presence of rogue insolvent edges. Any isolated insolvent edge is rogue by definition; a non-isolated insolvent edge is incident to one $H$-even face, and it is rogue if and only if its neighbors along the $H$-even face are also insolvent. It follows that the neighbors of a non-isolated rogue edge must be distinct, and hence non-isolated rogue edges always come in a triple along a common $H$-even face.

\begin{defn} \label{defn:posse}
A triple $\Posse=\{e_1,e_2,e_3\}$ of non-isolated rogue edges sharing a $H$-even face is called a \emph{posse} of (non-isolated) rogue edges.
\end{defn}

Building on \hyperref[lem:tprofile1]{Lemma~\ref*{lem:tprofile1}}, we shall further analyze the local structure of the triangulation around rogue edges, and describe how the canonical surface $\asurf$ around these edges can be compressed across these edges.

\begin{defn} \label{defn:cluster}
The \emph{cluster} $\kluster(e)$ around an edge $e$ is the union of tetrahedra containing $e$, and the \emph{open cluster} $\kluster^\circ\ntts (e)$ is the cluster $\kluster(e)$ with the faces not containing $e$ removed. For any collection $\Evens' \subseteq \Evens$ of edges, we define the cluster $\kluster(\Evens')$ by $\kluster(\Evens')=\bigcup_{e \in \Evens'} \kluster(e)$ and the open cluster $\kluster^\circ\ntts (\Evens')$ by $\kluster^\circ\ntts (\Evens')=\bigcup_{e \in \Evens'} \kluster^\circ\ntts (e)$.
\end{defn}

Let $e$ be an isolated rogue edge. By \hyperref[lem:tprofile1]{Lemma~\ref*{lem:tprofile1}}, the cluster $\kluster(e)$ consists of $\dcount(e)$ distinct tetrahedra of type $\typeq$ with $\dcount(e)=5$ or $6$. The open cluster $\kluster^\circ\ntts (e)$ is a $\dcount(e)$-gonal bipyramid with the boundary faces removed. 

\begin{lem} \label{lem:solocomp1}
Let $e$ be an isolated rogue edge. If $\surf$ is a nearly canonical surface dual to $\aclass$ with $\surf \cap \kluster(e)=\asurf \cap \kluster(e)$, then $\surf$ can be compressed once inside $\kluster^\circ(e)$ to a nearly canonical surface $\surf'$ dual to $\aclass$, satisfying $||\aclass|| \leqslant -\chi(\surf')=-\chi(\surf)-2$.
\end{lem}

\begin{figure}[t]
\begin{center}
\includegraphics[width=55.6mm, height=56.0mm]{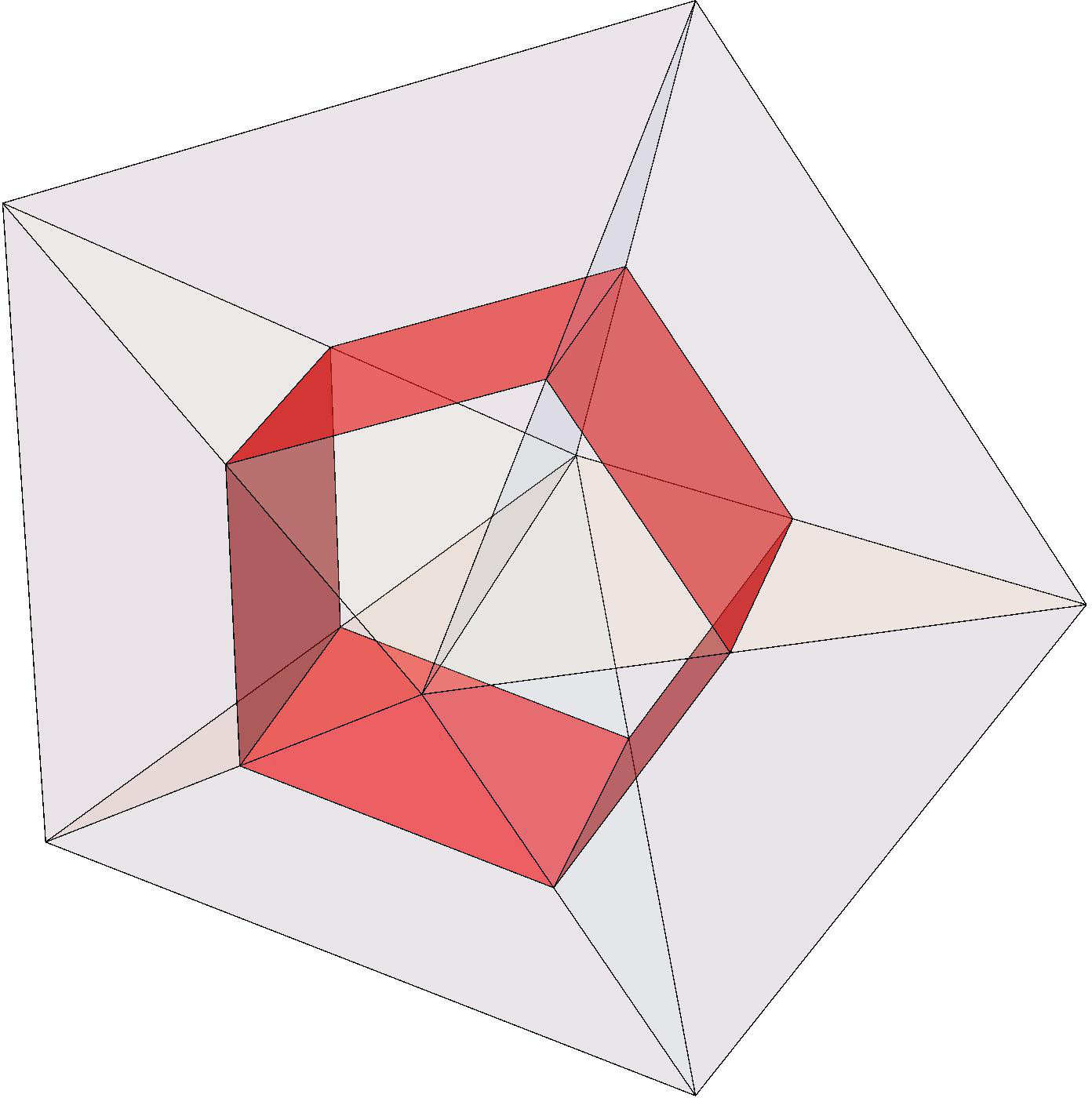} \,\qquad\,
\includegraphics[width=55.6mm, height=56.0mm]{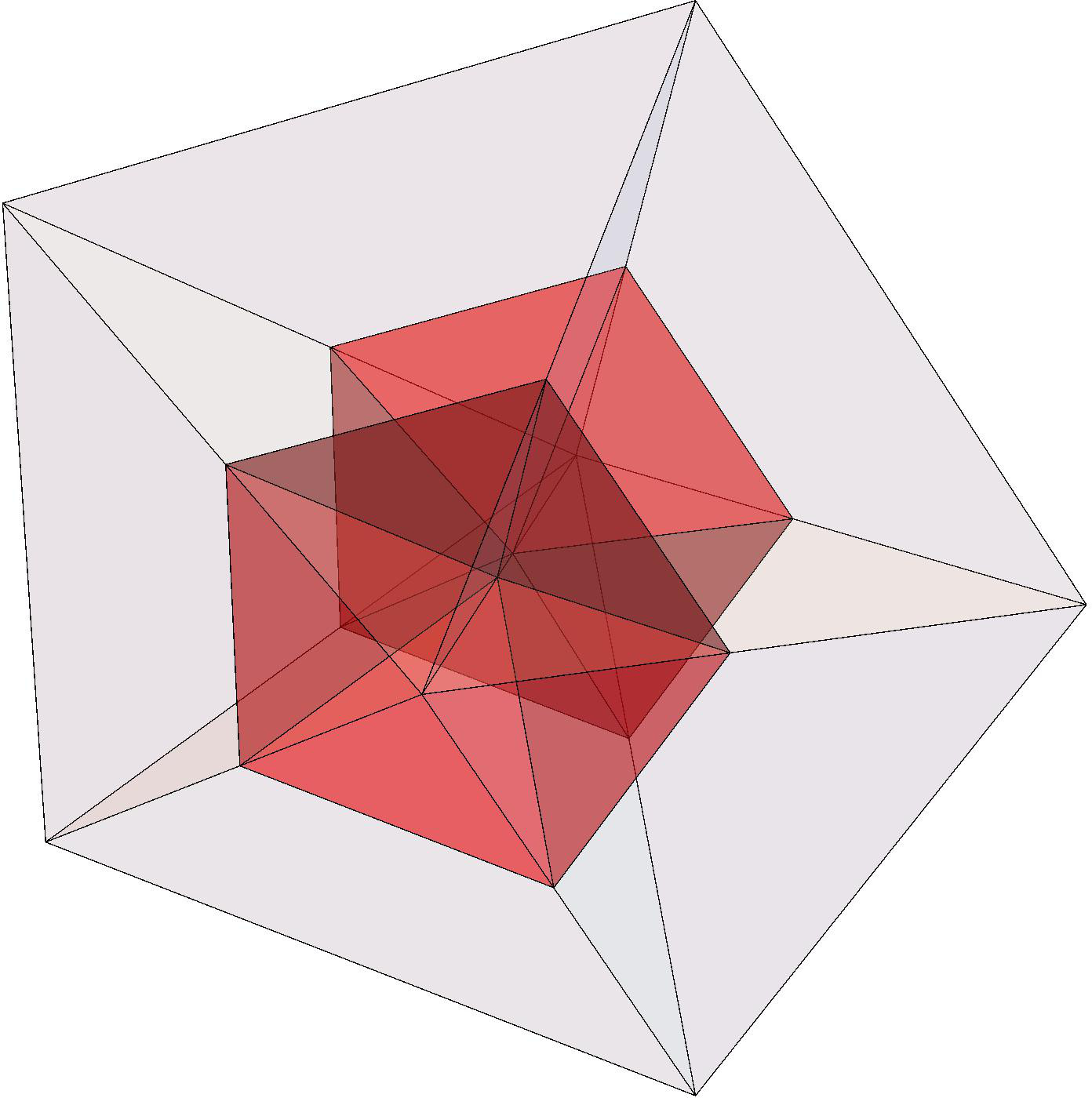}
\end{center}
\caption{A cluster around an isolated rogue edge (left), and compression of the canonical surface within the cluster (right).}
\label{fig:5cluster}
\end{figure}

\begin{proof}
The subsurface $\subsurf=\surf \cap \kluster^\circ\ntts (e)=\asurf \cap \kluster^\circ\ntts (e)$ is the union of $\dcount(e)$ quadrilaterals, forming an open annulus; see \hyperref[fig:5cluster]{Figure~\ref*{fig:5cluster}} (left).
The $\partial$-parallel loop $\gamma \subset \subsurf$ bounds a disk $D \subset \kluster^\circ(e)$ that intersects $e$ once. The surgery along $D$ yields a new surface $\subsurf'$ consisting of a pair of open disks, and it can be realized via isotopy as a normal surface in $\kluster(e)$; see \hyperref[fig:5cluster]{Figure~\ref*{fig:5cluster}} (right).
This surgery takes the nearly canonical surface $\surf$ with $\surf \cap \kluster^\circ\ntts (e)=\subsurf$ and yields a new normal surfece $\surf'$ with $\surf' \cap \kluster^\circ\ntts (e)=\subsurf'$; the edge-weights are unchanged on all edges except on $e$, for which the edge-weight is changed from 0 to 2, and the surface $\surf'$ is homologous to $\surf$. In particular, $\surf'$ is nearly canonical for the same class $\aclass$. By \hyperref[lem:norm-nc]{Lemma~\ref*{lem:norm-nc}}, the inequality $||\aclass|| \leqslant -\chi(\surf')$ holds, and this surgery is indeed a compression; the equality $\chi(\surf')=\chi(\surf)+2$ holds in general for a surgery along a disk.
\end{proof}

Let $\Posse=\{e_1,e_2,e_3\}$ be a posse of rogue edges, and $\kluster(\Posse)=\kluster(e_1) \cup \kluster(e_2) \cup \kluster(e_3)$ be the cluster around $\Posse$. By \hyperref[lem:tprofile1]{Lemma~\ref*{lem:tprofile1}}, each cluster $\kluster(e_m)$ consists of 5 distinct tetrahedra $\tau_{1(m)},\tau_{2(m)},\tau_{3(m)},\tau_{4(m)},\tau_{5(m)}$ of types $\typet,\typeq,\typeq,\typeq,\typet$ respectively in a cyclic order around $e_m$; each open cluster $\kluster^\circ\ntts (e_m)$ is a pentagonal bipyramid with the boundary faces removed. The 3 distinct edges $e_m$ belong to a single $H$-even face, which occurs as the face between 2 tetrahedra of type $\typet$ around each $e_m$. Hence, the clusters $\kluster(e_m)$ shrare this $H$-even face and the 2 tetrahedra of type $\typet$ containing it, say $\tau_1\nts=\nts\tau_{1(m)}$ and $\tau_5\nts=\nts\tau_{5(m)}$ for all $m$. The 6 tetrahedra $\tau_{2(m)}$, $\tau_{4(m)}$, $m=1,2,3$, of type $\typeq$ that are adjacent to one of the type $\typet$ tetrahedra $\tau_1,\tau_5$, must be distinct by \hyperref[lem:tprofile1]{Lemma~\ref*{lem:tprofile1}}, since they belong to 6 distinct maximal layered solid tori. The remaining 3 tetrahedra $\tau_{3(m)}$, $m=1,2,3,$ of type $\typeq$ may or may not be distinct. If $\tau_{3(m)}$ are distinct, then $\kluster(\Posse)$ is a union of 11 tetrahedra. If $\tau_{3(m)}$ are not distinct, then exactly 2 of them coincide and $\kluster(\Posse)$ is a union of 10 tetrahedra.

\begin{lem} \label{lem:posse11comp1}
Let $\Posse$ be a posse of rogue edges, with the 11-tetrahedra cluster $\kluster(\Posse)$ around it. If $\surf$ is a nearly canonical surface dual to $\aclass$ with $\surf \cap \kluster(\Posse)=\asurf \cap \kluster(\Posse)$, then $\surf$ can be compressed twice inside $\kluster^\circ(\Posse)$ to a nearly canonical surface $\surf'$ dual to $\aclass$, satisfying $||\aclass|| \leqslant -\chi(\surf')=-\chi(\surf)-4$.
\end{lem}

\begin{figure}[t]
\begin{center}
\includegraphics[width=56.6mm, height=61.8mm]{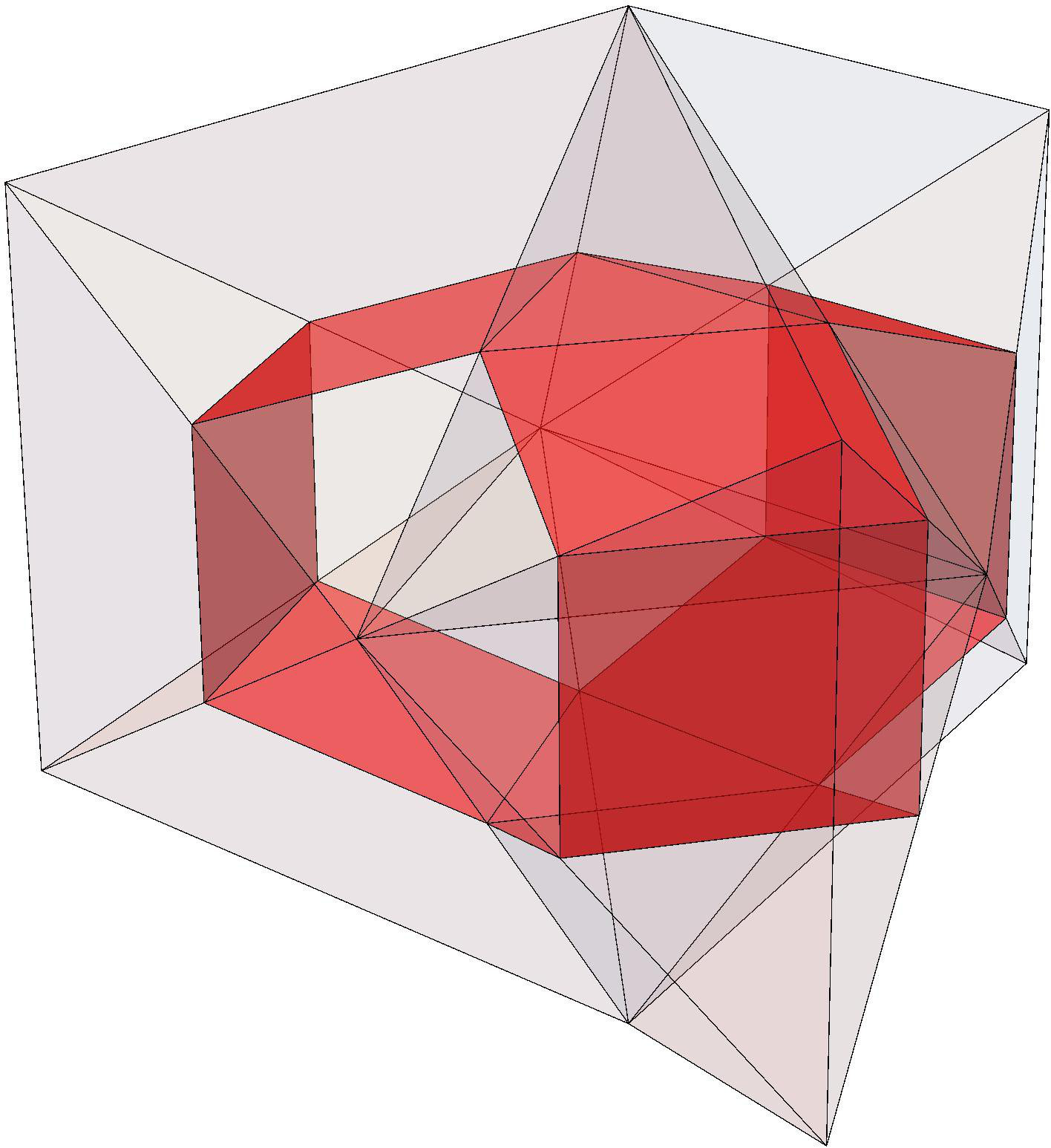} \,\qquad\,
\includegraphics[width=56.6mm, height=61.8mm]{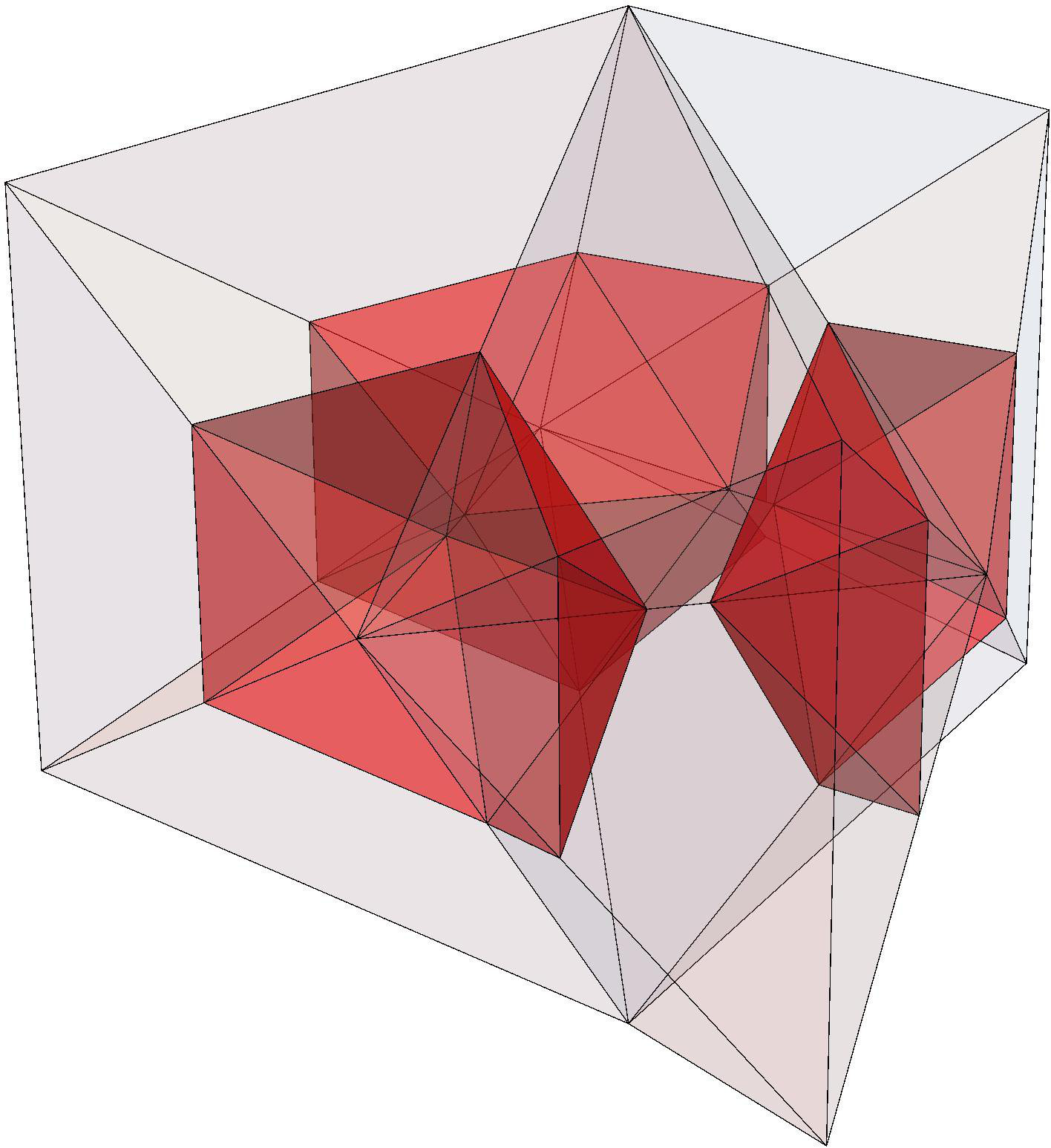}
\end{center}
\caption{An 11-tetrahedra cluster around a posse of non-isolated rogue edges (left), and double compression of the canonical surface within the cluster (right).}
\label{fig:11cluster}
\end{figure}

\begin{proof}
Let $\Posse=\{e_1,e_2,e_3\}$. The subsurface $\subsurf=\surf \cap \kluster^\circ\ntts (\Posse)=\asurf \cap \kluster^\circ\ntts (e)$ is an open pair of pants, formed by 2 triangles and 3 bands of 3 quadrilaterals connecting them; see \hyperref[fig:11cluster]{Figure~\ref*{fig:11cluster}} (left).
Consider an embedded $\theta$-shaped graph in $\subsurf$, formed by 3 arcs $\gamma_k$ connecting centers of two triangles through the bands of 3 quadrilaterals in $\kluster(e_k)$. The union $\gamma_{ij}=\gamma_i \cup \gamma_j$ of a pair of such arcs is a $\partial$-parallel loop on $\subsurf$, and it bounds a disk $D_{ij} \subset \kluster^\circ(e_i) \cup \kluster^\circ(e_j)$ that intersects each of $e_i, e_j$, once and is disjoint from the other edge $e_k$; we may isotope them to be pairwise disjoint.
The surgery along any pair of such disks, say $D_{12}$, $D_{13}$, yields a new surface $\subsurf'$ consisting of three open disks $D'_{ij}$ parallel to disks $D_{ij}$, and it can be realized via isotopy (that removes, in particular, 2 points of intersection between $D'_{23}$ and $e_1$) as a normal suface in $\kluster(\Posse)$; see \hyperref[fig:11cluster]{Figure~\ref*{fig:11cluster}} (right).
This surgery takes the nearly canonical surface $\surf$ with $\surf \cap \kluster^\circ\ntts (e)=\subsurf$ and yields a new normal surfece $\surf'$ with $\surf' \cap \kluster^\circ\ntts (e)=\subsurf'$; the edge-weights are unchanged on all edges except on $e_1,e_2,e_3$, for which the edge-weights are changed from 0 to 2, and the surface $\surf'$ is homologous to $\surf$. In particular, $\surf'$ is nearly canonical for the same class $\aclass$. By \hyperref[lem:norm-nc]{Lemma~\ref*{lem:norm-nc}}, the inequality $||\aclass|| \leqslant -\chi(\surf')$ holds, and each surgery is indeed a compression; the equality $\chi(\surf')=\chi(\surf)+4$ holds in general for a surgery along two disks.
\end{proof}


\begin{lem} \label{lem:posse10comp1}
Let $\Posse$ be a posse of rogue edges with the 10-tetrahedra cluster $\kluster(\Posse)$ around it. If $\surf$ is a nearly canonical surface dual to $\aclass$ with $\surf \cap \kluster(\Posse)=\asurf \cap \kluster(\Posse)$, then $\surf$ can be compressed once inside $\kluster^\circ(\Posse)$ to a nearly canonical surface $\surf'$ dual to $\aclass$, satisfying $||\aclass|| \leqslant -\chi(\surf')=-\chi(\surf)-2$.
\end{lem}

\begin{figure}[t]
\begin{center}
\includegraphics[width=56.6mm, height=61.8mm]{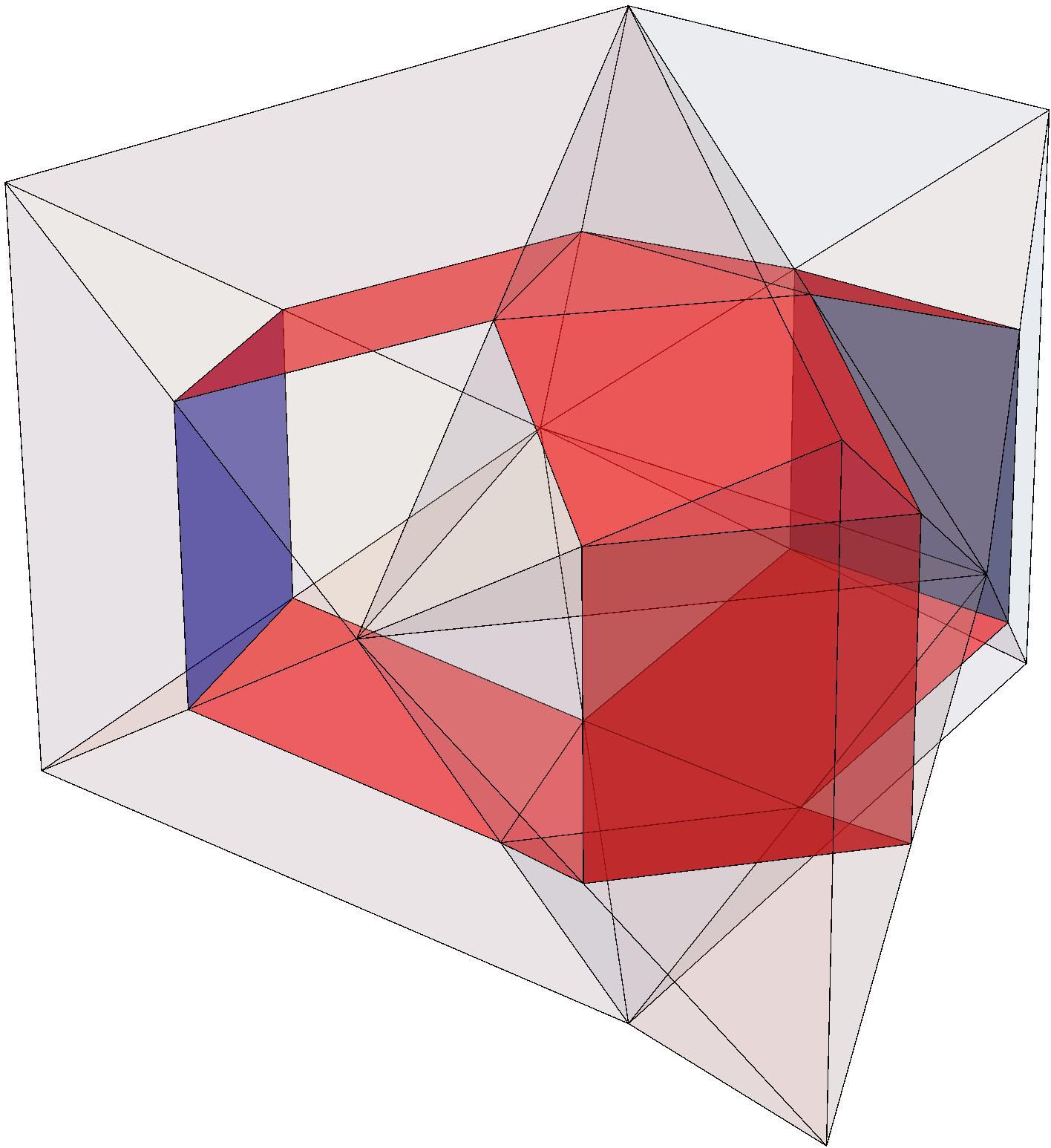} \,\qquad\,
\includegraphics[width=56.6mm, height=61.8mm]{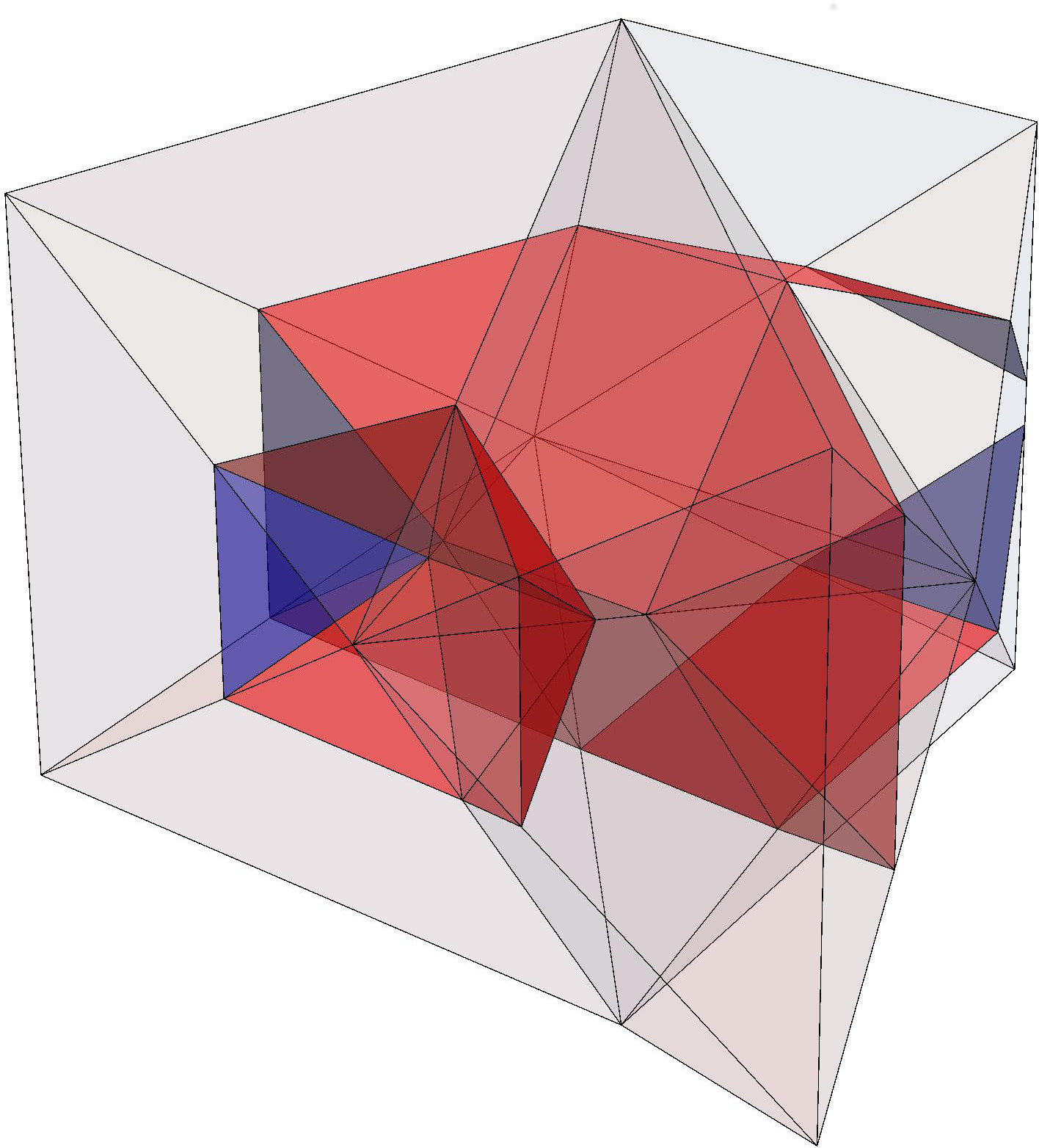}
\end{center}
\caption{A 10-tetrahedra cluster around a posse of non-isolated rogue edges (left), and compression of the canonical surface within the cluster (right); although we drew two tetrahedra containing blue disks for the sake of visualization, they actually represent the same single tetrahedron.}
\label{fig:10cluster}
\end{figure}

\begin{proof}
Let $\Posse=\{e_1,e_2,e_3\}$. Without loss of generality, we assume $\tau_{3(2)} \subset \kluster(e_2)$ and $\tau_{3(3)} \subset \kluster(e_3)$ coincide. The surface $\subsurf=\surf \cap \kluster^\circ\ntts (\Posse)=\asurf \cap \kluster^\circ\ntts (\Posse)$ is an open disk with 3 cross-caps, obtained from the pair of pants in \hyperref[lem:posse11comp1]{Lemma~\ref*{lem:posse11comp1}} by identifying the quadrilaterals in $\tau_{3(2)}$ and $\tau_{3(3)}$; see \hyperref[fig:10cluster]{Figure~\ref*{fig:10cluster}} (left).
As in the proof of \hyperref[lem:posse11comp1]{Lemma~\ref*{lem:posse11comp1}}, let $\gamma_k$ be embedded arcs connecting the centers of 2 triangles through the bands of 3 quadrilaterals in $\kluster(e_k)$, and consider homotopically non-trivial loops $\gamma_{ij}=\gamma_i \cup \gamma_j$ in $\subsurf$. In the present case, however, the arcs $\gamma_2$ and $\gamma_3$ intersect once transversely in the quadrilateral in $\tau_{3(2)}=\tau_{3(3)}$; consequentially, $\gamma_{23}$ is an immersed loop on $\subsurf$ that intersects itself once, while $\gamma_{12}$ and $\gamma_{13}$ are embedded loops on $\subsurf$ that intersect each other once (and $\gamma_{23}$ once as well). The loop $\gamma_{12}$ bounds an embedded disk $D_{12} \subset \kluster^\circ(e_1) \cup \kluster^\circ(e_2)$ that intersects each of $e_1, e_2$, once and is disjoint from $e_3$; similarly, the loop $\gamma_{13}$ bounds an embedded disk $D_{13} \subset \kluster^\circ(e_1) \cup \kluster^\circ(e_3)$ that intersects each of $e_1, e_3$ once and is disjoint from $e_2$. The surgery along either one of these disks, say $D_{12}$, yields a new surface $\subsurf'$ homeomorphic to an open M\"obius band, and it can be realized via isotopy as a normal suface in $\kluster(\Posse)$; see \hyperref[fig:10cluster]{Figure~\ref*{fig:10cluster}} (right).
This surgery takes the nearly canonical surface $\surf$ with $\surf \cap \kluster^\circ\ntts (e)=\subsurf$ and yields a new normal surfece $\surf'$ with $\surf' \cap \kluster^\circ\ntts (e)=\subsurf'$; the edge-weights are unchanged on all edges except on $e_1,e_2$, for which the edge-weights are changed from 0 to 2, and the surface $\surf'$ is homologous to $\surf$. In particular, $\surf'$ is nearly canonical for the same class $\aclass$. By \hyperref[lem:norm-nc]{Lemma~\ref*{lem:norm-nc}}, the inequality $||\aclass|| \leqslant -\chi(\surf')$ holds, and this surgery is indeed a compression; the equality $\chi(\surf')=\chi(\surf)+2$ holds in general for a surgery along a disk.
\end{proof}

\begin{rem}
The surface $\subsurf=\surf \cap \kluster^\circ\ntts (\Posse)$ can be compressed along a disk $D_{12}$ or $D_{13}$ but not both; the compression along one disk destroys the other disk.
The loop $\gamma_{23}$ bounds a disk $D_{23}$ with an embedded interior, but we cannot surger along this disk since $\partial D_{23}=\gamma_{23}$ is an immersed loop that intersects itself.
\end{rem}

\subsection{Busting}
\label{ssec:Busting1}

We now \emph{bust} the rogue edges; more precisely, we compress the canonical surface $\asurf$ across the rogue edges inside clusters. We apply \hyperref[lem:solocomp1]{Lemma~\ref*{lem:solocomp1}}, \hyperref[lem:posse11comp1]{Lemma~\ref*{lem:posse11comp1}}, and \hyperref[lem:posse10comp1]{Lemma~\ref*{lem:posse10comp1}} to achieve these compressions; we refer to these three lemmas collectively as the \emph{Busting Lemmas}.

The idea is to consider all clusters, and apply the Busting Lemmas as much as possible; the undesirable contribution from the rogue edge will be compensated by the drop in the norm of the surface. There is a one subtle point here: the clusters are not necessarily disjoint. More specifically, the tetrahedra around the rogue edge that are not contained in the maximal layered solid tori may belong to one cluster on one side and another cluster on the other side; in such a case, we can only apply the compression to one side within one of the clusters. Hence, we must show that, even if many clusters overlap at many tetrahedra and the surface cannot be compressed in all clusters simultaneously, there are enough compressions we can apply to make the desired inequality valid.

Let $\Klusters$ be the collection of all clusters around either an isolated rogue edge or a posse of non-isolated rogue edges; we denote the number of constituent clusters by
\[
\kappa:=\#\Klusters.
\]
Although we can't always compress $\isurf$ ($i=1,2,3$) repeatedly inside \emph{all} clusters in $\Klusters$, we can and shall compress them repeatedly inside clusters that share no tetrahedra to assure a large gap between $\sum ||\iclass||$ and $-\sum \chi(\isurf)$.

\begin{lem} \label{lem:maxcomp1}
Let $\Klusters' \subseteq \Klusters$ be a subcollection of clusters such that no two clusters $\kluster_1,\kluster_2 \nts \in \Klusters'$ share a tetrahedron, and write $\kappa'\nts:=\#\Klusters'$ for the number of constituent clusters in $\Klusters'$. Then, we have 
\[
||\aclass|| \leqslant -\chi(\asurf)-2\kappa'.
\]
\end{lem}

\begin{proof}
Let $\kluster_1, \cdots, \kluster_{\kappa'}$ be the clusters in $\Klusters'$. We compress $\surf^{(0)}=\asurf$ inside the cluster $\kluster_1$ by one of the \emph{Busting Lemmas}, and produce a nearly canonical surface $\surf^{(1)}$ which differs from $\surf^{(0)}$ only within $\kluster_1$ and hence satisfies $\surf^{(1)} \cap \kluster^\circ_\ell = \asurf \cap \kluster^\circ_\ell$ for $\ell=2,\cdots,\kappa'$.
We repeat inductively for $k=1, \cdots, \kappa'$: given a nearly canonical surface $\surf^{(k-1)}$ such that $\surf^{(k-1)} \cap \kluster^\circ_\ell=\asurf \cap \kluster^\circ_\ell$ for $\ell=k, \cdots, \kappa'$, we compress $\surf^{(k-1)}$ inside the cluster $\kluster_k$ by one of the \emph{Busting Lemmas}, and produce a nearly canonical surface $\surf^{(k)}$ which differs from $\surf^{(k-1)}$ only within $\kluster_k$ and hence satisfies $\surf^{(k)} \cap \kluster^\circ_\ell = \asurf \cap \kluster^\circ_\ell$ for $\ell=k+1,\cdots,\kappa'$. In particular, since we compress $\surf^{(k-1)}$ at least once to produce $\surf^{(k)}$ at each step in this series of compressions, we have
\[
||\aclass|| \,\leqslant\, -\chi(\surf^{(k)}) \,\leqslant\, -\chi(\surf^{(k-1)})-2
\]
for $k=1, \cdots, \kappa'$ by the \emph{Busting Lemmas}. Hence, we obtain
\[
||\aclass|| \,\leqslant\, -\chi(\surf^{(\kappa')}) \,\leqslant\, -\chi(\asurf)-2\kappa'
\]
by combining all inequalities in $\kappa'$ steps and rewriting $\surf^{(0)}=\asurf$.
\end{proof}

The inequality in \hyperref[lem:maxcomp1]{Lemma~\ref*{lem:maxcomp1}} involves $\kappa'$ which depends on the choice of the subcollection $\Klusters'$. To obtain an estimate independent of such choices, we would like to compare $\kappa'$ with $\kappa$. This can be done under another condition on $\Klusters'$.

\begin{lem} \label{lem:network1}
Let $\Klusters' \subseteq \Klusters$ be a subcollection of clusters such that every cluster $\kluster \nts \in \Klusters \ssminus \Klusters'$ share a tetrahedron with some cluster $\kluster' \nts \in \Klusters'$, and write $\kappa':=\#\Klusters'$ for the number of constituent clusters in $\Klusters'$. Then, we have
\[
4\kappa' \geqslant \kappa.
\]
\end{lem}

\begin{proof}
Since every cluster $\kluster \nts \in \Klusters \ssminus \Klusters'$ share a tetrahedron with some cluster $\kluster' \nts \in \Klusters'$ by assumption, we have an equality of collections
\[
\Klusters \;= \! \bigcup_{\kluster' \in \Klusters'} \! \{ \text{clusters $\kluster \nts \in \Klusters$ sharing a tetrahedron with $\kluster'$} \}.
\]
Possibly counting some clusters $\kluster \nts \in \Klusters$ with multiplicities, we obtain an inequality
\begin{align} \label{eqn:network1}
\kappa \;\leqslant \! \sum_{\kluster' \in \Klusters'} \! \#\{ \text{clusters $\kluster \nts \in \Klusters$ sharing a tetrahedron with $\kluster'$} \}.
\end{align}
We shall estimate the righthand term in terms of $\kappa'$. If a cluster $\kluster' \nts \in \Klusters'$ and another cluster $\kluster \nts \in \Klusters$ share a tetrahedron, the common tetrahedron must be a type $\typeq$ tetrahedron that is not contained in maximal layered solid tori. Hence, by \hyperref[lem:tprofile1]{Lemma~\ref*{lem:tprofile1}} and the description of clusters from \hyperref[ssec:Rogue1]{\S\ref*{ssec:Rogue1}}, each cluster $\kluster' \nts \in \Klusters'$ can share a tetrahedron with at most 3 other clusters $\kluster \nts \in \Klusters$. Including $K'$ itself, we thus have a coarse estimate
\[
4 \,\geqslant\, \#\{ \text{clusters $\kluster \nts \in \Klusters$ sharing a tetrahedron with $\kluster'$} \}
\]
for each $\kluster' \nts \in \Klusters'$. Taking the sum over all $\kluster' \nts \in \Klusters'$ on both sides, we obtain
\begin{align} \label{eqn:network'1}
4\kappa' \;\geqslant \!\! \sum_{\kluster' \in \Klusters'} \!\! \#\{ \text{clusters $\kluster \nts \in \Klusters$ sharing a tetrahedron with $\kluster'$} \}.
\end{align}
From the inequalities \hyperref[eqn:network1]{(\ref*{eqn:network1})} and \hyperref[eqn:network'1]{(\ref*{eqn:network'1})} above, we conclude $4 \kappa' \geqslant \kappa$.
\end{proof}

Combining the last two lemmas, we obtain the following inequality, which refines \hyperref[lem:norm-c]{Lemma~\ref*{lem:norm-c}} under a few additional assumptions on $\tri$ and $M$.

\begin{prop} \label{prop:fullcomp1}
Let $M$ be an irreducible 3-manifold such that $M \nhomeo \RP^3, L(4,1)$, with a non-zero class $\aclass \in H^1(M;\bbZtwo)$. Then, with respect to any minimal triangulation $\tri$ of $M$, that is not a layered lens space, the canonical surface $\asurf$ has no $S^2$-components or $\RP^2$-components, and satisfies
\[
||\aclass|| \leqslant -\chi(\asurf)-\kappa/2.
\]
\end{prop}

\begin{proof}
Choose a subcollection $\Klusters' \subseteq \Klusters$ of clusters, such that (i) no two clusters $\kluster_1,\kluster_2 \in \Klusters'$ share a tetrahedron, and (ii) every cluster $\kluster \nts \in \Klusters \ssminus \Klusters'$ share a tetrahedron with some cluster $\kluster' \nts \in \Klusters'$. We can always choose such a subcollection $\Klusters' \subseteq \Klusters$: starting with the empty collection, we enlarge it by adding clusters while maintaining the condition (i), until no more clusters can be added; the resulted collection then satisfies (ii). With the notation $\kappa':=\#\Klusters'$ as before, the inequalities
\[
||\aclass|| \,\leqslant\, -\chi(\asurf)-2\kappa' \,=\, -\chi(\asurf)-\kappa/2
\]
follow immediately from \hyperref[lem:maxcomp1]{Lemma~\ref*{lem:maxcomp1}} and \hyperref[lem:network1]{Lemma~\ref*{lem:network1}}.
\end{proof}

The last proposition assures that the gap between $||\aclass||$ and $-\chi(\asurf)$ is at least $\kappa/2$. This gap estimate can be related directly to the values $\fcount_1(e)$ on rogue edges.

\begin{lem} \label{lem:clusterogue1}
Let $\Rogues$ be the collection of rogue edges. Then, we have
\[
\kappa=-\!\!\sum_{e \in \Rogues} \fcount_1(e).
\]
\end{lem}

\begin{proof}
For each cluster $\kluster \in \Klusters$, let us write $\fcount_1(\kluster)$ for the sum of $\fcount_1(e)$ over all interior rogue edges $e \in \Rogues$ in the cluster $\kluster$. Since every rogue edge $e \in \Rogues$ is contained as an interior edge in exactly one cluster $K \in \Klusters$, we have
\begin{align} \label{eqn:clusterogue1}
\sum_{e \in \Rogues} \fcount_1(e)=\sum_{\kluster \in \Klusters} \fcount_1(\kluster).
\end{align}

An isolated rogue edge $e$ appears either in the case (1) or (2a), listed at the beginning of \hyperref[ssec:Decent1]{\S\ref*{ssec:Decent1}}, with $\gcount_1(e)=-1$, $\acount_1(e)=0$, and $\fcount_1(e)=-1$; hence, if $\kluster \nts \in \Klusters$ is a cluster around an isolated rogue edge $e$, then $\fcount_1(\kluster)=\fcount_1(e)=-1$. Each edge $e_k$ in a posse $\Posse=\{e_1,e_2,e_3\}$ of non-isolated rogue edges appears in the case (2b), listed at the beginning of \hyperref[ssec:Decent1]{\S\ref*{ssec:Decent1}}, with $\gcount_1(e_k)=-\frac{1}{3}$, $\acount_1(e_k)=0$, and $\fcount_1(e_k)=-\frac{1}{3}$; hence, if $\kluster \nts \in \Klusters$ is a cluster around a posse $\Posse=\{e_1,e_2,e_3\}$ of non-isolated rogue edges, then $\fcount_1(\kluster)=\fcount_1(e_1)+\fcount_1(e_2)+\fcount_1(e_3)=-1$. Therefore, we have $\fcount_1(\kluster)=-1$ for every cluster $\kluster \in \Klusters$, and the equality $\kappa=-\sum_{e \in \Rogues} \fcount_1(e)$ follows from \hyperref[eqn:clusterogue1]{(\ref*{eqn:clusterogue1})}.
\end{proof}

We now complete the proof of \hyperref[thm:rank1]{Theorem~\ref*{thm:rank1}}. By \hyperref[prop:fullcomp1]{Proposition~\ref*{prop:fullcomp1}}, we have
\[
2\Tcount-4-4||\aclass||
\;\geqslant\; 2\Tcount-4+4\chi(\asurf)+2\kappa.
\]
Let us write $\Rogues$ for the set of all rogue edges as before. Then, by \hyperref[lem:frewrite1]{Lemma~\ref*{lem:frewrite1}}, \hyperref[lem:decent1]{Lemma~\ref*{lem:decent1}}, and \hyperref[lem:clusterogue1]{Lemma~\ref*{lem:clusterogue1}}, we obtain
\[
2\Tcount-4+4\chi(\asurf)
\;\geqslant\; \sum_{e \in \Adults} \fcount_1(e)
\;\geqslant\; \sum_{e \in \Rogues} \fcount_1(e)
\;=\; -\kappa.
\]
Hence, combining these inequalities, we have
\[
2\Tcount-4-4||\aclass||
\;\geqslant\; 2\Tcount-4+4\chi(\asurf)+2\kappa
\;\geqslant\; -\kappa+2\kappa
\;=\;\kappa
\;\geqslant\; 0,
\]
or equivalently $\Tcount \geqslant 2+2||\aclass||$. This concludes the proof of \hyperref[thm:rank1]{Theorem~\ref*{thm:rank1}}.

\section{The Rank-2 Inequality for Prime Manifolds}
\label{sec:Rank2}

In this section, we establish \hyperref[thm:main2]{Theorem~\ref{thm:main2}}, regarding a lower bound on the complexity of an prime 3-manifold $M$ with $\rank H^1(M;\bbZtwo) \geqslant 2$. We restate \hyperref[thm:main2]{Theorem~\ref{thm:main2}} as follows.

\begin{thm} \label{thm:rank2}
Let $M$ be an orientable connected closed prime 3-manifold with a rank-2 subgroup $\{0,\rclass, \gclass, \bclass\} \leqslant H^1(M;\bbZtwo)$, and $\scrT$ be a minimal triangulation of $M$. Then, we have
\[
\Tcount(\tri) \geqslant 2+||\rclass||+||\gclass||+||\bclass||.
\]
\end{thm}

We adapt ideas and arguments from the rank-1 setting in \hyperref[sec:Rank1]{\S\ref*{sec:Rank1}}, with suitable adjustments for the rank-2 setting. By the condition $\rank H^1(M;\bbZtwo) \geqslant 2$, a prime manifold $M$ is irreducible, and not a lens space.

Suppose $H=\{\zclass,\rclass,\gclass,\bclass\} \leqslant H^1(M;\bbZtwo)$ is a rank-2 subgroup; we write $\calI=\{1,2,3\}$ for the set indexing the non-zero elements of $H$. With respect to the $\bbZtwo$-coloring of $\tri$ by $H$, \hyperref[lem:norm-c]{Lemma~\ref*{lem:norm-c}} and \hyperref[lem:comb2]{Lemma~\ref*{lem:comb2}} together yields
\begin{align}
\label{eqn:roadmap2}
\begin{split}
2\Tcount-4-2\sum_{i \in \calI} ||\iclass||
&\;\geqslant\;
2\Tcount-4+2\sum_{i \in \calI} \chi(\isurf)\\
&\;=\;
\sum_{d=3}^\infty (d-4) \Ecount_{\typee,d} + \Tcount_\typett + \Tcount_\typeqtt
\;\geqslant\;
\sum_{d=3}^\infty (d-4) \Ecount_{\typee,d}.
\end{split}
\end{align}
We aim to show that the lefthand term is non-negative. Under an extra topological assumption that $M$ is \emph{atoroidal}, the conclusion of \hyperref[thm:rank2]{Theorem~\ref*{thm:rank2}} was established in \cite{JRT:Z2} by showing that the righthand term $\sum_{d=3}^\infty (d-4) \Ecount_{\typee,d}$ is bounded below by $-1$. To prove \hyperref[thm:rank2]{Theorem~\ref*{thm:rank2}} in full generality, we work without this lower bound on the righthand term. As in the rank-1 case, there are two parts to our argument; let us spell out our outline although they are analogous to the rank-1 case.

First, we study how the middle term $\sum_{d=3}^\infty (d-4) \Ecount_{\typee,d} + \Tcount_\typett + \Tcount_\typeqtt$ fails to be non-negative. Our analysis of the middle term is similar to the analysis of the righthand term in \cite{JRT:Z2}, but incorporating the extra term $\Tcount_\typett,\Tcount_\typeqtt$ leads to a much finer description of the combinatorial structures causing the negativity. This part is more involved than the rank-1 case due to the number of cases to be considered.

Second, we exploit the gap between the lefthand term $2\Tcount-4-2\sum_{i \in \calI} ||\iclass||$ and the middle term $2\Tcount-4+2\sum_{i \in \calI} \chi(\isurf)$. When the canonical surfaces $\isurf$ are (geometrically) compressible, there is a gap between $\sum_{i \in \calI}||\iclass||$ and $-\sum_{i \in \calI}\chi(\isurf)$. Analyzing the combinatorics of the triangulation around the surfaces $\isurf$, we show that $\isurf$ are always sufficiently compressible so that the gap between $\sum_{i \in \calI}||\aclass||$ and $-\sum_{i \in \calI}\chi(\asurf)$ is large enough to make up for the nagativity of the middle term, leading to the desired inequality. This part reduces to the rank-1 case, since the profiles of problematic edges turn out to be essntially the same as the rank-1 case.

As evident from the discussion above, the middle term in \hyperref[eqn:roadmap2]{(\ref*{eqn:roadmap2})} plays the central role in our argument. Throughout this section, we denote this quantity by
\[
\Icount_2
:=\; 2\Tcount-4+2\sum_{i \in \calI} \chi(\isurf)
\;=\;\sum_{d=3}^\infty (d-4) \Ecount_{\typee,d} + \Tcount_\typet +\Tcount_\typeqtt.
\]

\subsection{Demography}
\label{ssec:Demography2}

For the rest of this section, we assume that $M$ is an irreducible 3-manifold with a rank-2 subgroup $H=\{\zclass,\rclass, \gclass, \bclass \} \leqslant H^1(M;\bbZtwo)$, and that $\tri$ is a minimal triangulation of $M$ with the $H$-coloring. The \emph{$\typett$-degree} $\dcount_\typett(e)$ and the \emph{$\typeqtt$-degree} $\dcount_\typeqtt(e)$ of an edge $e$ are defined to be 
\begin{align*}
\dcount_\typett(e):=\;\#\left\{\begin{array}{cc}
\text{1-simplices in the preimage $\quot^{-1}(e)$ that are incident to}\\
\text{3-simplices in the preimage of type $\typett$ tetrahedra}
\end{array}\right\},\\
\dcount_\typeqtt(e):=\;\#\left\{\begin{array}{cc}
\text{1-simplices in the preimage $\quot^{-1}(e)$ that are incident to}\\
\text{3-simplices in the preimage of type $\typeqtt$ tetrahedra}
\end{array}\right\}.
\end{align*}

The set of all $H$-even edges is denoted by $\Evens$ again. As in \hyperref[ssec:Demography1]{\S\ref*{ssec:Demography1}}, let us first rewrite $\Icount_2$ as a sum of contributions from $H$-even edges. For each $e \in \Evens$, we define
\begin{align*}
\icount_2(e):=\;\dcount(e)-4+{\textstyle\frac{\tts1\tts}{3}} \dcount_\typett(e) + \dcount_\typeqtt(e).
\end{align*}

\begin{lem} \label{lem:irewrite-2}
With the notations as above, we have
\[
\Icount_2 \,=\, \sum_{e \in \Evens} \icount_2(e).
\]
\end{lem}

\begin{proof}
We note that the $\typett$-degrees of $H$-even edges satisfy $\sum_{e \in \Evens} \dcount_\typett(e)= 3\Tcount_\typett$ and the $\typeqtt$-degrees of $H$-even edges satisfy $\sum_{e \in \Evens} \dcount_\typeqtt(e)= \Tcount_\typeqtt$. Taking the sum of $\icount_2(e)$ over $\Evens$ and regrouping by their degrees, we obtain
\[
\;\sum_{e \in \Evens} \icount_2(e)
=\sum_{e \in \Evens} (\dcount(e)-4)+\sum_{e \in \Evens} {\textstyle\frac{\tts1\tts}{3}}\dcount_\typett(e)+\dcount_\typeqtt(e)
=\sum_{d=3}^\infty (d-4)\Ecount_{\typee,d}+\Tcount_\typett+\Tcount_\typeqtt=\Icount_2
\]
by the definition of $\Icount_2$ and $\icount_2(e)$, as desired.
\end{proof}

Following \hyperref[ssec:Demography1]{\S\ref*{ssec:Demography1}}, we derive a tractable estimate of $\Icount_2=\sum_{e \in \Evens} \icount_2(e)$. We use the notions of \emph{adult} edges, \emph{child} edges, and \emph{baby} edges from \hyperref[defn:edges]{Definition~\ref*{defn:edges}}, and define the \emph{supporter} of a maximal layered solid torus as in \hyperref[defn:supporter]{Definition~\ref*{defn:supporter}}, with the type $\typeqq$ replacing the type $\typeq$; we write $\lst \leadsto e$ and $e' \leadsto e$ as before to mean $e$ supports $\lst$ and $e'$, respectively. Then, the partition \hyperref[eqn:partition]{(\ref*{eqn:partition})} and the notation $\bcount(e)$ make sense verbatim. We define a counting function $\gcount_2$ on the set $\Adults \subset \Evens$ of adult edges by
\begin{align*}
\gcount_2(e):=\;\dcount(e)-4-\bcount(e)+{\textstyle\frac{\tts1\tts}{3}} \dcount_\typett(e)+\dcount_\typeqtt(e).
\end{align*}

\begin{lem} \label{lem:grewrite2}
With the notations as above, we have
\[
\Icount_2 \,\geqslant\, \sum_{e \in \Adults} \gcount_2(e).
\]
\end{lem}

\begin{proof}
The proof is essentially the same as the proof of \hyperref[lem:grewrite1]{Lemma~\ref*{lem:grewrite1}}.
In this proof, we write $a$ for adult edges, $c$ for child edges, and $e$ for arbitrary edges in $\Evens$. Using the partition of $\Evens$ into subsets $\House(a)$, $a \in \Adults$, as defined in \hyperref[eqn:partition]{(\ref*{eqn:partition})},
\begin{align*}
\Icount_2
&=\sum_{e \in \Evens} \icount_2(e)=\sum_{a \in \Adults} \, \sum_{e \in \House(a)} \icount_2(e)
=\sum_{a \in \Adults} \, \bigg(\icount_2(a)+\sum_{\lst \leadsto a} \, \sum_{c \in \Children(\lst)} \icount_2(c)\bigg) 
\end{align*}
by \hyperref[lem:irewrite-2]{Lemma~\ref*{lem:irewrite-2}}. For any child edge $c \in \Children(\lst)$, we have $\dcount_\typett(c)=\dcount_\typeqtt(c)=0$ since $\lst$ contains no tetrahedra of type $\typett$ or type $\typeqtt$. Grouping together the edges of the same degree as before,
\begin{align*}
\sum_{c \in \Children(\lst)} \icount_2(c)
=\sum_{c \in \Children(\lst)} (\dcount(c)-4)
=\sum_{d=3}^\infty (d-4) \, \Ecount_{\typee,d} (\lst)
\geqslant -\Ecount_{\typee,3} (\lst).
\end{align*}
Thus, combining the observations above, we have
\begin{align*}
\Icount_2
&=\sum_{a \in \Adults} \, \bigg(\icount_2(a)+\sum_{\lst \leadsto a} \, \sum_{c \in \Children(\lst)} \icount_2(c)\bigg) \\
&\geqslant \sum_{a \in \Adults} \, \bigg(\dcount(a)-4+{\textstyle\frac{\tts1\tts}{3}} \dcount_\typett(a)+\dcount_\typeqtt(a)-\sum_{\lst \leadsto a}\Ecount_{\typee,3} (\lst) \bigg)
=\sum_{a \in \Adults} \gcount_2(a)
\end{align*}
by the definition of $\icount_2(e)$ and $\gcount_2(e)$, as desired.
\end{proof}

\subsection{Insolvent Adults}
\label{ssec:Insolvent2}

As in \hyperref[ssec:Insolvent1]{\S\ref*{ssec:Insolvent1}}, we shall study edges $e \in \Adults$ with $\gcount_2(e)<0$. We redefine solvent and insolvent edges for the present rank-2 setting.

\begin{defn}
An adult edge $e \in \Adults$ is said to be \emph{solvent} if $\gcount_2(e) \geqslant 0$, and it is said to be \emph{insolvent} if $\gcount_2(e)<0$.
\end{defn}

We aim to identify the local combinatorics around solvent/insolvent adult edges. Note that the proof of \hyperref[lem:halfdeg1]{Lemma~\ref*{lem:halfdeg1}} analyzes the necessary condition for the quantity $\dcount(e)-4-\bcount(e)$ to be non-negative using the inequalities $\lfloor\dcount(e)/2\rfloor \geqslant \scount(e) \geqslant \bcount(e)$. Since these inequality holds in the rank-2 setting as well, we can readily establish the rank-2 analogue of \hyperref[lem:halfdeg1]{Lemma~\ref*{lem:halfdeg1}}; the proof is verbatim and omitted.

\begin{lem} \label{lem:halfdeg2}
If an adult edge $e \in \Adults$ satisfies one of the following conditions, then $\gcount_2(e) \geqslant \dcount(e)-4-\bcount(e) \geqslant 0$ holds, and hence in particular $e$ is a solvent adult edge:
(i) $\dcount(e) \geqslant 7$;
(ii) $\dcount(e)=6$, $\bcount(e) \leqslant 2$;
(iii) $\dcount(e)=5$, $\bcount(e) \leqslant 1$;
(iv) $\bcount(e)=0$.
\par
Thus, every insolvent adult edge $e$ must satisfy one of the following conditions:
(1) $\dcount(e)=6$, $\scount(e)=\bcount(e)=3$;
(2) $\dcount(e)=5$, $\scount(e)=\bcount(e)=2$;
(3) $\dcount(e)=4$, $\scount(e)=\bcount(e)=2$;
(4) $\dcount(e)=4$, $\scount(e)=2$, $\bcount(e)=1$;
(5) $\dcount(e)=4$, $\scount(e)=\bcount(e)=1$.
\end{lem}

\hyperref[lem:halfdeg2]{Lemma~\ref*{lem:halfdeg2}} does not establish $\gcount_2(e) < 0$ under the conditions (1)-(5); we have $\dcount(e)-4-\bcount(e)<0$ for these cases, but we must consider the extra terms $\frac{1}{3} \dcount_\typet(e)$ and $\dcount_\typeqtt(e)$ in $\gcount_2(e)$. We first give the rank-2 analogue of \hyperref[lem:eprofile1]{Lemma~\ref*{lem:eprofile1}}.

\begin{lem} \label{lem:eprofile2}
The following statements (in which $\lst$ denotes a maximal layered solid torus) holds for any adult edge $e \in \Adults$:
\begin{enumerate}
\item $\dcount(e)=6$, $\scount(e)=3$, $\bcount(e)=3$ $\Rightarrow$
$\lst \leadsto e$ and $\lst$ is of type $\typeqq$ for each $\lst \ni e$;
\item $\dcount(e)=5$, $\scount(e)=2$, $\bcount(e)=2$ $\Rightarrow$
$\lst \leadsto e$ and $\lst$ is of type $\typeqq$ for each $\lst \ni e$;
\item $\dcount(e)=4$, $\scount(e)=2$, $\bcount(e)=2$ $\Rightarrow$
$\lst \leadsto e$ and $\lst$ is of type $\typeqq$ for each $\lst \ni e$;
\item $\dcount(e)=4$, $\scount(e)=2$, $\bcount(e)=1$ $\Rightarrow$
$\lst \leadsto e$ and $\lst$ is of type $\typeqq$ for one $\lst \ni e$;
\item $\dcount(e)=4$, $\scount(e)=1$, $\bcount(e)=1$ $\Rightarrow$
$\lst \leadsto e$ and $\lst$ is of type $\typeqq$ for $\lst \ni e$;
\end{enumerate}
moreover, in all cases, we must have $\dcount_\lst(e)=1$ for each $\lst \ni e$.
\end{lem}

\begin{proof}
Recall that, in the statement and the proof of \hyperref[lem:eprofile1]{Lemma~\ref*{lem:eprofile1}}, the $\bbZtwo$-coloring by the rank-1 subgroup $H$ was relevant only in the appearance of the type $\typeq$ tetrahedra. The proof of the present lemma is identical to the proof of \hyperref[lem:eprofile1]{Lemma~\ref*{lem:eprofile1}}, with the type $\typeqq$ tetrahedra replacing the type $\typeq$ tetrahedra.
\end{proof}

We now give the rank-2 analogue of \hyperref[lem:tprofile1]{Lemma~\ref*{lem:tprofile1}}, characterizing insolvent edgse completely. In contrast to the rank-1 case, an edge $e$ satisfying $\dcount(e)-4-\bcount(e)<0$, and hence satisfying one of the conditions (1)-(5) by \hyperref[lem:eprofile2]{Lemma~\ref*{lem:eprofile2}}, need not be insolvent in general because of the extra terms $\frac{1}{3} \dcount_\typet(e)$ and $\dcount_\typeqtt(e)$ in $\gcount_2(e)$.

\begin{lem} \label{lem:tprofile2}
An adult edge $e \in \Adults$ is insolvent if and only if the types of tetrahedra incident to $e$, expressed by cyclically ordered $\dcount(e)$-tuples of type-symbols up to dihedral symmetry (with a dot, such as $\dot\typeqq$, if the underlying tetrahedron is in a layered solid torus $\lst$ and with two dots, such as $\ddot\typeqq$, if this $\lst$ contains a baby edge $e' \leadsto e$), is one of the following:
\begin{enumerate}
\item $\dcount(e)=6$, $\scount(e)=3$, $\bcount(e)=3$:
$(\typeqq,\ddot\typeqq,\typeqq,\ddot\typeqq,\typeqq,\ddot\typeqq)$;
\item $\dcount(e)=5$, $\scount(e)=2$, $\bcount(e)=2$:
$(\typeqq,\ddot\typeqq,\typeqq,\ddot\typeqq,\typeqq)$ or
$(\typett,\ddot\typeqq,\typeqq,\ddot\typeqq,\typett)$;
\item $\dcount(e)=4$, $\scount(e)=2$, $\bcount(e)=2$:
$(\typeqq,\ddot\typeqq,\typeqq,\ddot\typeqq)$;
\item $\dcount(e)=4$, $\scount(e)=2$, $\bcount(e)=1$:
$(\typeqq,\ddot\typeqq,\typeqq,\dot\typeqq)$,
$(\typett,\ddot\typeqq,\typett,\dot\typee)$.
\item $\dcount(e)=4$, $\scount(e)=1$, $\bcount(e)=1$:
$(\typeqq,\ddot\typeqq,\typeqq,\typeqq)$, 
$(\typett,\ddot\typeqq,\typeqq,\typett)$, or
$(\typett,\ddot\typeqq,\typett,\typee)$.
\end{enumerate}
In all cases, the edge $e$ must be incident to $\dcount(e)$ distinct tetrahedra.
\end{lem}

\begin{proof}
For each subcase in the list, we can readily verify $\gcount_2(e)<0$, i.e.\;$e$ is insolvent. For the converse, suppose $e$ is insolvent. Possible combinations (1)-(5) of $\dcount(e), \scount(e), \bcount(e)$ are given in \hyperref[lem:halfdeg2]{Lemma~\ref*{lem:halfdeg2}}. The types of tetrahedra, incident to $e$ and contained in maximal layered solid tori, are given in \hyperref[lem:eprofile2]{Lemma~\ref*{lem:eprofile2}}; they are always type $\typeqq$ except for one of type $\typee$ in case (4). The coloring on the boundary faces of maximal layered solid tori restricts the possible types of remaining tetrahedra; to list all possible combinations, we first consider the types of remaining tetrahedra with the normal arc patterns on faces matching the boundary faces of maximal layered solid tori, and then remove the combinations without consistent coloring.

\begin{enumerate}
\item Each symbol between two $\ddot\typeqq$ must be $\typeqq$ or $\typeqtt$. Among 4 possible combinations up to symmetry, 3 of them admit consistent coloring:\\
\phantom{------}
$(\typeqq,\ddot\typeqq,\typeqq,\ddot\typeqq,\typeqq,\ddot\typeqq)$,
$(\typeqtt,\ddot\typeqq,\typeqtt,\ddot\typeqq,\typeqq,\ddot\typeqq)$,
$(\typeqtt,\ddot\typeqq,\typeqtt,\ddot\typeqq,\typeqtt,\ddot\typeqq)$.
\item The two consecutive symbols two $\ddot\typeqq$ must be both $\typeqq$, one $\typeqq$ and one $\typeqtt$, both $\typeqtt$, or both $\typett$; the remaining symbol must be $\typeqq$ or $\typeqtt$. Among 8 possible combinations up to symmetry, 6 of them admit consistent coloring:\\
\phantom{------}
$(\typeqq,\ddot\typeqq,\typeqq,\ddot\typeqq,\typeqq)$,
$(\typeqtt,\ddot\typeqq,\typeqq,\ddot\typeqq,\typeqtt)$,
$(\typett,\ddot\typeqq,\typeqq,\ddot\typeqq,\typett)$,\\
\phantom{------}
$(\typeqtt,\ddot\typeqq,\typeqtt,\ddot\typeqq,\typeqq)$,
$(\typeqtt,\ddot\typeqq,\typeqtt,\ddot\typeqq,\typeqtt)$,
$(\typett,\ddot\typeqq,\typeqtt,\ddot\typeqq,\typett)$.
\item Each symbol between two $\ddot\typeqq$ must be $\typeqq$ or $\typeqtt$. Among 3 possible combinations up to symmetry, 2 of them admit consistent coloring:\\
\phantom{------}
$(\typeqq,\ddot\typeqq,\typeqq,\ddot\typeqq)$,
$(\typeqtt,\ddot\typeqq,\typeqtt,\ddot\typeqq)$.
\item If $\dot\typeqq$ is opposite to $\ddot\typeqq$, each symbol between them must be $\typeqq$ or $\typeqtt$. Among 3 possible combinations up to symmetry, 2 of them admit consistent coloring:\\
\phantom{------}
$(\typeqq,\ddot\typeqq,\typeqq,\dot\typeqq)$,
$(\typeqtt,\ddot\typeqq,\typeqtt,\dot\typeqq)$.
\\
If $\dot\typee$ is opposite to $\ddot\typeqq$, each symbol between them must be $\typett$. The unique possible combination, up to symmetry, admits consistent coloring:\\
\phantom{------}
$(\typett,\ddot\typeqq,\typett,\dot\typee)$.
\item If $\typeqq$ is opposite to $\ddot\typeqq$, each symbol between them must be $\typeqq$ or $\typeqtt$. Among 3 possible combinations up to symmetry, 2 of them admit consistent coloring:\\
\phantom{------}
$(\typeqq,\ddot\typeqq,\typeqq,\typeqq)$, 
$(\typeqtt,\ddot\typeqq,\typeqtt,\typeqq)$.
\\
If $\typeqtt$ is opposite to $\ddot\typeqq$, each symbol between them must be $\typeqq$ or $\typeqtt$. Among 3 possible combinations up to symmetry, 2 of them admit consistent coloring:\\
\phantom{------}
$(\typeqtt,\ddot\typeqq,\typeqq,\typeqtt)$, 
$(\typeqtt,\ddot\typeqq,\typeqtt,\typeqtt)$.
\\
If $\typett$ is opposite to $\ddot\typeqq$, one symbol between them must be $\typett$, and the other $\typeqq$ or $\typeqtt$. Both possible combinations, up to symmetry, admit consistent coloring:\\
\phantom{------}
$(\typett,\ddot\typeqq,\typeqq,\typett)$.
$(\typett,\ddot\typeqq,\typeqtt,\typett)$.
\\
If $\typee$ is opposite to $\ddot\typeqq$, each symbols between them must be $\typett$. The unique possible combination, up to symmetry, admits consistent coloring:\\
\phantom{------}
$(\typett,\ddot\typeqq,\typett,\typee)$.
\end{enumerate}

Since $e$ is insolvent by assumption, we require $\gcount_2(e)<0$. Computing $\gcount_2(e)$ for 21 candidate combinations above, we find that ones with a type $\typeqtt$ tetrahedron yields $\gcount_2(e) \geqslant 0$ due to the presence of the positive extra term $\dcount_\typeqtt(e)$. The remaining 9 combinations satisfy $\gcount_2(e)<0$, and are listed in the statement of the lemma,

Note that these 9 combinations are essentially the same as the ones in \hyperref[lem:tprofile1]{Lemma~\ref*{lem:tprofile1}}, except that every normal disk (a triangle in a type $\typet$ tetrahedron, or a quadrilateral in a type $\typeq$ tetrahedron) is replaced by two parallel copies (two triangles in a type $\typett$ tetrahedron, or two quadrilaterals in a type $\typeqq$ tetrahedron). Hence, an argument in the proof of \hyperref[lem:tprofile1]{Lemma~\ref*{lem:tprofile1}} apply almost verbatim (requiring only to replace type $\typet$ and type $\typeq$ with type $\typett$ and type $\typeqq$ respectively), proving that the tetrahedra around an insolvent edge $e$ are distinct.
\end{proof}

The degree~4 cases (3)-(5) in \hyperref[lem:tprofile2]{Lemma~\ref*{lem:tprofile2}} are studied to some extent and utilized in \cite[\S5.1]{JRT:Z2}; the degree~4 edges they considered are equivalent to adult edges $e$ satisfying $\dcount(e)-4-\bcount(e)<0$ in our language. We carried out a finer analysis using $\gcount_2(e)$, with extra terms $\frac{1}{3} \dcount_\typet(e)$ and $\dcount_\typeqtt(e)$, substantially cutting down the number of cases of problematic edges, and essentially reducing them to the rank-1 setting.

The edge flips in the rank-2 context were studied in \cite[\S5.1]{JRT:Z2} to eliminate edges of degree~4 satisfying $\dcount(e)-4-\bcount(e)<0$. We only need to eliminate insolvent edges, so many cases considered in \cite[\S5.1]{JRT:Z2} are obsolete; indeed, since profile of insolvent edges are essentially identical in rank-1 and rank-2 settings, the effect of edge flips around insolvent edges in the rank-2 coloring can be deduced \emph{directly} from the corresponding analysis for the rank-1 coloring.

\begin{propalpha}[{\cite[\S5.1]{JRT:Z2}}] \label{prop:edgeflip2}
Let $M$ be an irreducible 3-manifold with a rank-2 subgroup $H=\{\zclass,\rclass,\gclass,\bclass\} < H^1(M;\bbZtwo)$. Then $M$ admits a minimal triangulation with no insolvent edge of degree~4 with respect to the $H$-coloring.
\end{propalpha}

\begin{rem}
At this point, one can establish \hyperref[thm:rank2]{Theorem~\ref*{thm:rank2}} for \emph{atoroidal} manifolds. If an insolvent edge $e'$ of degree~5 or 6 exists, arguments in \cite[Lem.\,9, Lem.\,10]{JRT:Z2} assure that it is the unique insolvent edge. We have $\gcount_2(e) \geqslant 0$ for all $e \in \Adults$ except for possibly one insolvent edge $e'$ with $\gcount_2(e') \geqslant -1$. Hence, \hyperref[eqn:roadmap2]{(\ref*{eqn:roadmap2})} and \hyperref[lem:grewrite2]{Lemma~\ref*{lem:grewrite2}} together yield $2\Tcount-4-2 \sum_{i \in \Index} ||\iclass|| \geqslant \Icount_2 \geqslant \sum_{e \in \Adults} \gcount_2(e) \geqslant -1$. Since $2\Tcount-4-2 \sum_{i \in \Index} ||\iclass||$ is an even integer, we have $2\Tcount-4-2 \sum_{i \in \Index} ||\iclass|| \geqslant 0$ and thus $\Tcount \geqslant 2+\sum_{i \in \Index} ||\iclass||$.
\end{rem}

\subsection{Decent Adults}
\label{ssec:Decent2}

Invoking \hyperref[prop:edgeflip2]{Proposition~\ref*{prop:edgeflip2}}, we may now assume that $\tri$ has no insolvent edges of degree~4 with respect to the rank-2 $H$-coloring. Every insolvent edge $e$ must occur as one of the following cases from \hyperref[lem:tprofile2]{Lemma~\ref*{lem:tprofile2}}.
\begin{itemize}
\item[(1)] $\dcount(e)=6$, $\scount(e)=\bcount(e)=3$: $(\typeqq,\ddot\typeqq,\typeqq,\ddot\typeqq,\typeqq,\ddot\typeqq)$ with $\gcount_2(e)=-1$;
\item[(2a)] $\dcount(e)=5$, $\scount(e)=\bcount(e)=2$: $(\typeqq,\ddot\typeqq,\typeqq,\ddot\typeqq,\typeqq)$ with $\gcount_2(e)=-1$;
\item[(2b)] $\dcount(e)=5$, $\scount(e)=\bcount(e)=2$: $(\typett,\ddot\typeqq,\typeqq,\ddot\typeqq,\typett)$ with $\gcount_2(e)=-\frac{\tts1\tts}{3}$.
\end{itemize}
Following \hyperref[ssec:Decent1]{\S\ref*{ssec:Decent1}}, we shall now make adjustments to $\gcount_2$, essentially in the same way as in the rank-1 setting. Insolvent edges are divided into \emph{decent} edges and \emph{rogue} edges according to \hyperref[defn:rogue]{Definition~\ref*{defn:rogue}}, verbatim but reinterpreted in the context of the rank-2 $H$-coloring. An insolvent edge in the case (1) or (2a) is incident to no $H$-even faces, and hence always isolated and rogue. An insolvent edge in the case (2b) is not isolated, and it may be decent or rogue. For each $e \in \Adults$, we define
\begin{align*}
\acount_2(e):=\;\begin{cases}
-\frac{\tts1\tts}{3} \times \#\{\text{insolvent neighbors of $e$}\} \qquad \text{if $e$ is solvent},\\
+\frac{\tts1\tts}{3} \times \#\{\text{solvent neighbors of $e$}\} \qquad \text{if $e$ is insolvent},
\end{cases}
\end{align*}
and modify our counting function $\gcount_2(e)$ with this adjustment term by setting
\[
\fcount_2(e):=\;\gcount_2(e)+\acount_2(e)\;=\;\dcount(e)-4-\bcount(e)+{\textstyle\frac{\tts1\tts}{3}\dcount_\typet(e)}+\dcount_\typeqtt(e)+\acount_2(e).
\]
We have the rank-2 analogue of \hyperref[lem:frewrite1]{Lemma~\ref*{lem:frewrite1}}; the proof is verbatim and omitted.

\begin{lem} \label{lem:frewrite2}
With the notations as above, we have
\[
\Icount_2 \,\geqslant\, \sum_{e \in \Adults} \fcount_2(e).
\]
\end{lem}

By \hyperref[lem:frewrite2]{Lemma~\ref*{lem:frewrite2}}, we may use $\fcount_2(e)$ in place of $\gcount_2(e)$ for our counting. The following rank-2 analogue of \hyperref[lem:decent1]{Lemma~\ref*{lem:decent1}} assures that the rogue edges are the only troublesome insolvent edges in our counting; the proof is almost verbatim (requiring only to replace type $\typet$ with type $\typett$) and omitted.

\begin{lem} \label{lem:decent2}
An adult edge $e \in \Adults$ satisfies $\fcount_2(e) \geqslant 0$ if and only if the edge is either (i) a solvent edge or (ii) a decent insolvent edge.
\end{lem}

\subsection{Rogue Adults}
\label{ssec:Rogue2}

Following \hyperref[ssec:Rogue1]{\S\ref*{ssec:Rogue1}}, we now deal with rogue insolvent edges. As in the rank-1 setting, any isolated insolvent edge is rogue, and non-isolated rogue edges always come in a triple along a common $H$-even face. We define a \emph{posse} of non-isolated rogue edges by \hyperref[defn:posse]{Definition~\ref*{defn:posse}} verbatim, but reinterpreted in the context of the rank-2 $H$-coloring. As before, we aim to compress the canonical surface $\asurf$ around rogue edges. For this, we define \emph{clusters} and \emph{open clusters} around an isolated edge and around non-isolated rogue edges by \hyperref[defn:cluster]{Definition~\ref*{defn:cluster}} verbatim, but reinterpreted in the context of the rank-2 $H$-coloring.

First, let $e$ be an isolated rogue edge. By \hyperref[lem:tprofile2]{Lemma~\ref*{lem:tprofile2}}, the structure of the cluster $\kluster(e)$ is essentially identical to the rank-1 case, except that all tetrahedra have two normal disks of the same type, instead of one normal disk. Namely, the cluster $\kluster(e)$ consists of $\dcount(e)$ distinct tetrahedra of type $\typeqq$ with $\dcount(e)=5$ or $6$, and the open cluster $\kluster^\circ\ntts (e)$ is a $\dcount(e)$-gonal bipyramid with the boundary faces removed. The gluing of these tetrahedra forces that all tetrahedra have two normal disks of the same colors, say colors $i$ and $j$ but not $k$ where $\{i,j,k\}=\{1,2,3\}$. In other words, the canonical surfaces inside the cluster $\kluster(e)$ is just two parallel copies (one with color $i$ and the other with color $j$) of the cylinders we dealt with in \hyperref[lem:solocomp1]{Lemma~\ref*{lem:solocomp1}}. So, restricting our attention to a rank-1 subgroup $\{0,\iclass\}<H$, the lemma below follows immediately from \hyperref[lem:solocomp1]{Lemma~\ref*{lem:solocomp1}}.

\begin{lem} \label{lem:solocomp2}
Let $e$ be an isolated rogue edge. If $\surf$ is a nearly canonical surface dual to $\iclass$ with $\surf \cap \kluster(e)=\isurf \cap \kluster(e) \neq \nil$, then $\surf$ can be compressed once inside $\kluster^\circ(e)$ to a nearly canonical surface $\surf'$ dual to $\iclass$, satisfying $||\iclass|| \leqslant -\chi(\surf')=-\chi(\surf)-2$.
\end{lem}

Next, let $\Posse=\{e_1,e_2,e_3\}$ be a posse of rogue edges, and $\kluster(\Posse)=\kluster(e_1) \cup \kluster(e_2) \cup \kluster(e_3)$ be the cluster around $\Posse$. Again, by \hyperref[lem:tprofile2]{Lemma~\ref*{lem:tprofile2}}, the structure of the cluster $\kluster(e)$ is essentially identical to the rank-1 case, except that all tetrahedra have two normal disks of the same type, instead of one normal disk. Namely, the cluster $\kluster(\Posse)$ consists of 11 tetrahedra (see \hyperref[fig:11cluster]{Figure~\ref*{fig:11cluster}}, left) or 10 tetrahedra (see \hyperref[fig:10cluster]{Figure~\ref*{fig:10cluster}}, left), except that all tetrahedra has two copies of normal disks of the same type. The gluing of these tetrahedra forces that all tetrahedra have two normal disks of the same colors, say colors $i$ and $j$ but not $k$ where $\{i,j,k\}=\{1,2,3\}$. In other words, the canonical surfaces inside the cluster $\kluster(\Posse)$ is just two parallel copies (one with color $i$ and the other with color $j$) of the surfaces we dealt with in \hyperref[lem:posse11comp1]{Lemma~\ref*{lem:posse11comp1}} and \hyperref[lem:posse10comp1]{Lemma~\ref*{lem:posse10comp1}}. So, restricting our attention to a rank-1 subgroup $\{0,\iclass\}<H$, the lemmas below follows immediately from \hyperref[lem:posse11comp1]{Lemma~\ref*{lem:posse11comp1}} and \hyperref[lem:posse10comp1]{Lemma~\ref*{lem:posse10comp1}}.

\begin{lem} \label{lem:posse11comp2}
Let $\Posse$ be a posse of rogue edges, with the 11-tetrahedra cluster $\kluster(\Posse)$ around it. If $\surf$ is a nearly canonical surface dual to $\iclass$ with $\surf \cap \kluster(\Posse)=\isurf \cap \kluster(\Posse) \neq \nil$, then $\surf$ can be compressed twice inside $\kluster^\circ(\Posse)$ to a nearly canonical surface $\surf'$ dual to $\iclass$, satisfying $||\iclass|| \leqslant -\chi(\surf')=-\chi(\surf)-4$.
\end{lem}

\begin{lem} \label{lem:posse10comp2}
Let $\Posse$ be a posse of rogue edges with the 10-tetrahedra cluster $\kluster(\Posse)$ around it. If $\surf$ is a nearly canonical surface dual to $\iclass$ with $\surf \cap \kluster(\Posse)=\isurf \cap \kluster(\Posse) \neq \nil$, then $\surf$ can be compressed once inside $\kluster^\circ(\Posse)$ to a nearly canonical surface $\surf'$ dual to $\iclass$, satisfying $||\iclass|| \leqslant -\chi(\surf')=-\chi(\surf)-2$.
\end{lem}

\subsection{Busting}
\label{ssec:Busting2}

We now \emph{bust} the rogue edges, as we did in the rank-1 case; more precisely, we compress the canonical surface $\isurf$ ($i=1,2,3$) across the rogue edges inside clusters, by applying \emph{the Busting Lemmas} \hyperref[lem:solocomp2]{Lemma~\ref*{lem:solocomp2}}, \hyperref[lem:posse11comp2]{Lemma~\ref*{lem:posse11comp2}}, and \hyperref[lem:posse10comp2]{Lemma~\ref*{lem:posse10comp2}}.

The procedures are essentially the same as in the rank-1 case, but with some extra book-keeping arising from dealing with colors. Let $\Klusters$ be the collection of all clusters around either an isolated rogue edge or a posse of non-isolated rogue edges; we denote the number of constituent clusters by
\[
\kappa:=\#\Klusters.
\]
Note that each constituent cluster $K \in \Klusters$ intersects two of the canonical surface non-trivially while being disjoint from the other canonical surface. So, it is natural to consider the subcollections of $\Klusters$ accordingly: for each $\{i,j\} \subset \{i,j,k\}=\{1,2,3\}$,
\[
\Klusters_{i,j}:=\{K \in \Klusters \mid \isurf \cap \kluster \neq \nil, \jsurf \cap \kluster \neq \nil, \ksurf \cap \kluster = \nil \},
\quad
\kappa_{i,j}:=\#\Klusters_{i,j}
\]
and for each $i \in \{i,j,k\}=\{1,2,3\}$, 
\[
\Klusters_{i}:=\Klusters_{i,j} \sqcup \Klusters_{i,k},
\quad
\kappa_{i}:=\kappa_{i,j}+\kappa_{i,k}=\#\Klusters_{i}.
\]
Hence, we have
\begin{equation} \label{eqn:kappa2}
\Klusters=\Klusters_{1,2} \sqcup \Klusters_{1,3} \sqcup \Klusters_{2,3},
\quad
\kappa=\kappa_{1,2}+\kappa_{1,3}+\kappa_{2,3}=(\kappa_1+\kappa_2+\kappa_3)/2.
\end{equation}
For each $i$, we can compress $\isurf$ inside each cluster $K \in \Klusters_i$ using the Busting Lemmas; however, we can't always compress $\isurf$ inside \emph{all} clusers in $\Klusters_i$ simultaneously; so, as in the rank-1 case, we can and shall compress $\isurf$ repeatedly inside clusters in $\Klusters_i$ that share no tetrahedra.

It is tempting to just apply the counting from the rank-1 case to the rank-1 subgroup $\{0,\iclass\} \leqslant H$. However, strictly speaking, a rank-1 rogue edge $e$ with respect a subgroup $\{0,\iclass\} \leqslant H$ need not be a rank-2 rogue edge with respect $H$; such a rank-1 rogue edge $e$ is a rank-2 rogue edge with respect to $H$ if and only if it is also a rank-1 rogue edge with respect to another subgroup $\{0,\jclass\} \leqslant H$, $j \neq i$. To avoid confusions arising from this subtle issue, we shall restate the versions of the lemmas and propositions from the rank-1 case, using the notations above. The proof of the following
\hyperref[lem:maxcomp2]{Lemma~\ref*{lem:maxcomp2}},
\hyperref[lem:network2]{Lemma~\ref*{lem:network2}},
\hyperref[prop:fullcomp2]{Proposition~\ref*{prop:fullcomp2}}, and
\hyperref[lem:clusterogue2]{Lemma~\ref*{lem:clusterogue2}}
are omitted since they are essentially identical to the proof of the corresponding 
\hyperref[lem:maxcomp1]{Lemma~\ref*{lem:maxcomp1}},
\hyperref[lem:network1]{Lemma~\ref*{lem:network1}},
\hyperref[prop:fullcomp1]{Proposition~\ref*{prop:fullcomp1}}, and
\hyperref[lem:clusterogue1]{Lemma~\ref*{lem:clusterogue1}},

\begin{lem} \label{lem:maxcomp2}
Let $i \in \{1,2,3\}$.
Let $\Klusters'_i \subseteq \Klusters_i$ be a subcollection of clusters such that no two clusters $\kluster_1,\kluster_2 \nts \in \Klusters'_i$ share a tetrahedron, and write $\kappa'_i \nts:=\#\Klusters'_i$ for the number of constituent clusters in $\Klusters'_i$. Then, we have 
\[
||\iclass|| \leqslant -\chi(\isurf)-2\kappa'_i.
\]
\end{lem}

\begin{lem} \label{lem:network2}
Let $i \in \{1,2,3\}$.
Let $\Klusters'_i \subseteq \Klusters_i$ be a subcollection of clusters such that every cluster $\kluster \nts \in \Klusters_i \ssminus \Klusters'_i$ share a tetrahedron with some cluster $\kluster'_i \nts \in \Klusters'_i$, and write $\kappa'_i:=\#\Klusters'_i$ for the number of constituent clusters in $\Klusters'_i$. Then, we have
\[
4\kappa'_i \geqslant \kappa_i.
\]
\end{lem}

\begin{prop} \label{prop:fullcomp2}
Let $M$ be an irreducible 3-manifold with a rank-2 subgroup $H=\{0,\rclass,\gclass,\bclass\} \leqslant H^1(M;\bbZtwo)$. Then, with respect to any minimal triangulation $\tri$ of $M$, that is not a layered lens space, the canonical surface $\isurf$ has no $S^2$-components or $\RP^2$-components, and satisfies
\[
||\iclass|| \leqslant -\chi(\isurf)-\kappa_i/2.
\]
\end{prop}

\begin{lem} \label{lem:clusterogue2}
Let $\Rogues_i$ be the collection of rogue edges that are contained in some cluster $K \in \Klusters_i$. Then, we have
\[
\kappa_i=-\!\!\sum_{e \in \Rogues_i} \fcount_2(e).
\]
\end{lem}

We now complete the proof of \hyperref[thm:rank2]{Theorem~\ref*{thm:rank2}}. By \hyperref[prop:fullcomp2]{Proposition~\ref*{prop:fullcomp2}}, we have
\begin{equation} \label{eqn:proof2}
2\Tcount-4-2\sum_{i \in \calI} ||\iclass||
\;\geqslant\; 2\Tcount-4+2\sum_{i \in \calI} \chi(\isurf)+\sum_{i \in \calI}\kappa_i.
\end{equation}
Let us write $\Rogues$ for the set of all rogue edges. Then, by \hyperref[lem:frewrite2]{Lemma~\ref*{lem:frewrite2}} and \hyperref[lem:decent2]{Lemma~\ref*{lem:decent2}}, we obtain
\begin{equation} \label{eqn:rogues2}
2\Tcount-4+2\sum_{i \in \calI} \chi(\isurf)
\;\geqslant\; \sum_{e \in \Adults} \fcount_2(e)
\;\geqslant\; \sum_{e \in \Rogues} \fcount_2(e).
\end{equation}
As above, for each $i \in \{1,2,3\}$, let $\Rogues_i$ be the collection of rogue edges that are contained in some cluster $K
\in \Klusters_i$; also, for each $\{i,j\} \subset \{i,j,k\} =\{1,2,3\}$, let $\Rogues_{i,j}$ be the collection of rogue edges that are contained in some cluster $K \in \Klusters_{i,j}$.
Then, since
$\Klusters=\Klusters_{i,j} \sqcup \Klusters_{i,k} \sqcup \Klusters_{j,k}$,
we have
$\Rogues=\Rogues_{i,j} \sqcup \Rogues_{i,k} \sqcup \Rogues_{j,k}$;
also, since
$\Klusters_i=\Klusters_{i,j} \sqcup \Klusters_{i,k}$,
we have
$\Rogues=\Rogues_{i,j} \sqcup \Rogues_{i,k}$.
Hence, together with \hyperref[lem:clusterogue2]{Lemma~\ref*{lem:clusterogue2}},
\[
\sum_{e \in \Rogues} \fcount_2(e)
= \!\!\sum_{e \in \Rogues_{i,j}} \!\!\fcount_2(e)+\!\!\!\!\sum_{e \in \Rogues_{i,k}} \!\!\fcount_2(e)+\!\!\!\!\sum_{e \in \Rogues_{j,k}} \!\!\fcount_2(e)\\
= \frac{1}{2} \sum_{i \in \calI}\sum_{e \in \Rogues_i} \fcount_2(e)
= -\frac{1}{2} \sum_{i \in \calI}\kappa_i.
\]
Substituting this into the equation \hyperref[eqn:rogues2]{(\ref{eqn:rogues2})}, we obtain
\begin{equation}
2\Tcount-4+2\sum_{i \in \calI} \chi(\isurf)
\;\geqslant\; \sum_{e \in \Rogues} \fcount_2(e)
= -\frac{1}{2} \sum_{i \in \calI}\kappa_i.
\end{equation}
Hence, combining this with \hyperref[eqn:proof2]{(\ref{eqn:proof2})}, we have
\[
2\Tcount-4-2\sum_{i \in \calI} ||\iclass||
\;\geqslant\; -\frac{1}{2} \sum_{i \in \calI}\kappa_i+\sum_{i \in \calI} \kappa_i
\;=\; \frac{1}{2} \sum_{i \in \calI}\kappa_i
\;\geqslant\; 0,
\]
or equivalently $\Tcount \geqslant 2+\sum_{i \in \calI} ||\iclass||$. This concludes the proof of \hyperref[thm:rank2]{Theorem~\ref*{thm:rank2}}.

\bibliography{MinTri-Z2}

\providecommand{\bysame}{\leavevmode\hbox to3em{\hrulefill}\thinspace}
\providecommand{\MR}{\relax\ifhmode\unskip\space\fi MR }
\providecommand{\MRhref}[2]{%
  \href{http://www.ams.org/mathscinet-getitem?mr=#1}{#2}
}
\providecommand{\href}[2]{#2}
\begin{thebibliography}{JRT13}

\bibitem[Bur03]{Burton:Thesis}
B.~A. Burton, \emph{Minimal triangulations and normal surfaces}, Ph.D. thesis,
  The University of Melbourne, 2003.

\bibitem[JR03]{JR:0-efficient}
W.~Jaco and H.~Rubinstein, \emph{$0$-efficient triangulations of 3-manifolds},
  J. Differential Geom. \textbf{65} (2003), no.~1, 61--168.

\bibitem[JR06]{JR:Layered}
W.~Jaco and J.~H. Rubinstein, \emph{Layered-triangulations of 3-manifolds},
  preprint, 2006,  \href{http://front.math.ucdavis.edu/math.GT/0603601}{arXiv:
  math.GT/0603601}.

\bibitem[JRT09]{JRT:Lens}
W.~Jaco, J.~H. Rubinstein, and S.~Tillmann, \emph{Minimal triangulations for an
  infinite family of lens spaces}, J. Topol. \textbf{2} (2009), no.~1,
  157--180.

\bibitem[JRT11]{JRT:Covering}
\bysame, \emph{Coverings and minimal triangulations of 3-manifolds}, Algebr.
  Geom. Topol. \textbf{11} (2011), no.~3, 1257--1265.

\bibitem[JRT13]{JRT:Z2}
\bysame, \emph{{$\mathbb{Z}_2$-Thurston norm and complexity of 3-manifold}},
  Math. Ann. \textbf{356} (2013), no.~1, 1--22.

\bibitem[Mat98]{Mat:Table}
S.~Matveev, \emph{Tables of 3-manifolds up to complexity 6}, available from
  \texttt{http://www.mpim-bonn.mpg.de}, 1998.

\bibitem[MP04]{MP:Geometric}
B.~Martelli and C.~Petronio, \emph{Complexity of geometric three-manifolds},
  Geom. Dedicata \textbf{108} (2004), 15--69.

\bibitem[Thu86]{Thurston:Norm}
W.~P. Thurston, \emph{A norm for the homology of 3-manifolds}, Mem. Amer. Math.
  Soc. \textbf{339} (1986), i--vi and 99--130.

\end{thebibliography}
\bibliographystyle{hamsalpha}

\end{document}